\documentclass[oneside]{article}
\usepackage{amssymb, amsmath}
\usepackage{euscript}
\usepackage{mathrsfs}
\usepackage{xypic}
\usepackage{ntheorem}
\usepackage{enumitem}
\usepackage{url, hyperref}
\usepackage{tikz, tikz-cd}
\usepackage{stmaryrd}

\theoremseparator{.}
\theoremprework{\smallskip\smallskip}
\newtheorem{thm}{Theorem}[section]
\theoremprework{\smallskip\smallskip}
\newtheorem{cor}[thm]{Corollary}
\theoremprework{\smallskip\smallskip}
\newtheorem{lem}[thm]{Lemma}
\theoremprework{\smallskip\smallskip}
\newtheorem{prop}[thm]{Proposition}
\theoremprework{\smallskip\smallskip}
\theorembodyfont{\upshape}
\newtheorem{defn}[thm]{Definition}
\theoremprework{\smallskip\smallskip}

\theoremprework{\smallskip\smallskip}

\theoremprework{\smallskip\smallskip}
\newtheorem{rmk}[thm]{Remark}
\theoremprework{\smallskip\smallskip}

\theoremprework{\smallskip\smallskip}

\theoremprework{\smallskip\smallskip}
\newtheorem{ex}[thm]{Example}

\theoremprework{\smallskip\smallskip}

\theoremprework{\smallskip\smallskip}

\theoremprework{\smallskip\smallskip}

\newenvironment{proof}[1][Proof]{\begin{trivlist}\item[]{\textsc{#1}.\
}}{\nolinebreak $\Box$ \end{trivlist}} 

\theoremprework{\smallskip\smallskip}
\newtheorem{introthm}{Theorem}

\newcommand{\stack}[2]{\begin{subarray}{c} #1 \\ #2 \end{subarray}}

\newcommand{\hide}[1]{{}}

\newcommand{\mb}{\mathbb}

\newcommand{\mc}{\mathcal}
\newcommand{\mr}{\mathrm}

\DeclareMathAlphabet{\mathpzc}{OT1}{pzc}{m}{it}
\DeclareMathAlphabet{\mf}{U}{euf}{m}{n}
\DeclareMathAlphabet{\ms}{U}{rsfso}{m}{n}

\newenvironment{items}{\begin{enumerate}[topsep=3pt, itemsep=3pt, parsep=0pt, label=(\roman*)]}{\end{enumerate}}

\newcommand{\ol}{\overline}

\newcommand{\ul}{\underline}

\newcommand{\muhat}{{{\hat{\mu}}}}

\newcommand{\sfrac}[2]{{\textstyle\frac{#1}{#2}}}
\renewcommand{\setminus}{\mathbin{\mathchoice{\mysetminusD}{\mysetminusT}{\mysetminusS}{\mysetminusSS}}}
	\newcommand{\mysetminusD}{\hbox{\tikz{\draw[line width=0.6pt,line cap=round] (3pt,0) -- (0,6pt);}}}
	\newcommand{\mysetminusT}{\mysetminusD}
	\newcommand{\mysetminusS}{\hbox{\tikz{\draw[line width=0.45pt,line cap=round] (2pt,0) -- (0,4pt);}}}
	\newcommand{\mysetminusSS}{\hbox{\tikz{\draw[line
                width=0.4pt,line cap=round] (1.5pt,0) -- (0,3pt);}}}

\renewcommand{\tilde}{\widetilde}

\newcommand{\cfg}{\mr{cfg}}
\newcommand{\et}{\text{{\'et}}}

\newcommand{\inv}{^{-1}\!}

\newcommand{\DM}{{\mr{DM}}}

\newcommand{\St}{{\mr{St}}}
\newcommand{\Var}{{\mr{Var}}}

\newcommand{\Coh}{\ensuremath{\ms{C}\kern-.245em\mathpzc{oh}}}

\renewcommand{\phi}{\varphi}
\renewcommand{\epsilon}{\varepsilon}

\newcommand{\qq}{\mb Q}
\newcommand{\pp}{\mb P}
\newcommand{\aaa}{\mb A}

\newcommand{\cc}{\mb C}

\newcommand{\zz}{\mb Z}

\newcommand{\YY}{\mf Y}      
\newcommand{\EE}{\mf E} 
\newcommand{\MM}{\mf M}
\newcommand{\NN}{\mf N} 
\newcommand{\A}{\mf A} 
\renewcommand{\P}{\mf P} 

\newcommand{\WW}{\mf W}      
\newcommand{\XX}{\mf X}
\renewcommand{\AA}{\mf A}
\newcommand{\ZZ}{\mf Z}      

\newcommand{\SSS}{\ms S}

\newcommand{\E}{\ms E} 
\newcommand{\Q}{\ms Q}
\newcommand{\F}{\ms F}
\newcommand{\C}{\ms C}
\newcommand{\N}{\ms N}
\newcommand{\K}{\ms K}

\newcommand{\G}{\ms G}
\newcommand{\I}{\ms I}
\renewcommand{\H}{\ms H}
\renewcommand{\O}{{\ms O}}

\newcommand{\dmo}[2]{\DeclareMathOperator{#1}{#2}}
\dmo{\Ext}{Ext}

\dmo{\Id}{Id}
\dmo{\idemp}{idemp}
\dmo{\Ann}{Ann}
\dmo{\id}{id}
\dmo{\cl}{cl}
\dmo{\im}{im}
\dmo{\rank}{rank}
\dmo{\rk}{rk}
\dmo{\res}{res}
\dmo{\ord}{ord}
\dmo{\radd}{rad}
\dmo{\Char}{char}

\dmo{\uAut}{\ul{Aut}}
\dmo{\uEnd}{\ul{End}}
\dmo{\uHom}{\ul{Hom}}
\dmo{\uIsom}{\ul{Isom}}
\dmo{\spec}{Spec}
\dmo{\Aut}{Aut}
\dmo{\cok}{cok}
\dmo{\B}{B}
\dmo{\BGL}{BGL}
\dmo{\DT}{DT}
\dmo{\End}{End}
\dmo{\Gal}{Gal}
\dmo{\GL}{GL}
\dmo{\Hom}{Hom}
\dmo{\Log}{Log}
\dmo{\Epi}{Epi}
\dmo{\Inj}{Inj}
\dmo{\Pic}{Pic}
\dmo{\Spec}{Spec}
\dmo{\supp}{supp}
\dmo{\Sym}{Sym}
\dmo{\Stab}{Stab}
\dmo{\SStable}{{SS}}
\dmo{\Stable}{{St}}
\newcommand{\Gm}{\mb G_{m}}
\newcommand{\sheafEnd}{\ensuremath{\ms{E}\kern-.245em\mathpzc{nd}}}
\newcommand{\sheafOrth}{\ensuremath{\ms{O}\kern-.245em\mathpzc{rth}}}
\newcommand{\shom}{\mathop{\rm {\mit{ \ms H\! om}}}\nolimits}

\newcommand{\fix}{{\text{\it fix}}}

\renewcommand{\ss}{{\text{\it ss}}}
\newcommand{\ad}{{\text{\it ad}}}
\newcommand{\reg}{{\text{\it reg}}}

\newcommand{\vir}{{\text{\it vir}}}
\newcommand{\vi}{{\text{\it vi}}}
\newcommand{\prim}{{\text{\it prim}}}

\newcommand{\strat}{{\text{\it strat}}}

\newcommand{\longiso}{\stackrel{\simeq}{\longrightarrow}}

\usepackage{graphicx}
\usepackage{calc}
\newlength{\depthofsumsign}
\newlength{\totalheightofsumsign}
\newlength{\heightanddepthofargument}

\makeatletter
\newcommand*{\DivideLengths}[2]{%
  \strip@pt\dimexpr\number\numexpr\number\dimexpr#1\relax*65536/\number\dimexpr#2\relax\relax sp\relax
}
\makeatother

\allowdisplaybreaks[3]

\newcommand{\noprint}[1]{}
\newcommand{\unsure}[1]{{\footnotesize #1}}
\newcommand{\comment}[1]{{$\mbox{}^{\spadesuit}$}{\marginpar{\footnotesize
$\spadesuit$  #1\hfill \scriptsize{-kai}}}}

\def\comment{\noprint}
\def\unsure{\noprint}

\title{The Inertia Operator on  the Motivic Hall Algebra}
\author{Kai Behrend, Pooya Ronagh}
\date{\today}

\usepackage{hyperref}
\begin{document}
\maketitle
\abstract{We study the action of the inertia operator on the motivic Hall
  algebra, and prove that it is diagonalizable.  This leads to a
  filtration of the Hall algebra,
  whose associated graded algebra is commutative. In particular, the
  degree 1 subspace forms a Lie algebra, which we call the Lie algebra
  of {\em virtually indecomposable }elements, following Joyce. We prove that
  the integral of  virtually indecomposable elements admits an Euler
  characteristic specialization. In order to take advantage of the
  fact that our inertia groups are unit groups in algebras, we
  introduce the notion of {\em algebroid}.}

\tableofcontents
\section*{Introduction}
\addcontentsline{toc}{section}{Introduction}

For simplicity, let us work over a field $k$. (Later, $k$ will be
replaced by a noetherian ring $R$.)

Let $\MM$ be an abelian $k$-linear algebraic stack. Roughly, this
means that $\MM$ is at the same time a $k$-linear abelian category
with finite-dimensional hom-spaces, and an algebraic stack, locally of
finite type over $k$.  (The precise definition of linear algebraic
stack is Definition~\ref{defn-lin}.  In the body of the paper we work
with exact, instead of abelian categories, see the beginning of
Section~\ref{sec-filo}.)

Examples  we are interested in include
\begin{items}
  \item $\MM=\mf{Coh}_Y$, the moduli stack of coherent $\O_Y$-modules, for
    a projective $k$-variety $Y$,
  \item $\MM=\mf{Rep}_Q$, the moduli stack of representations of a
    quiver $Q$ on 
    finite-dimensional $k$-vector spaces,
  \item (the case $Y=Q=\spec k$) $\MM=\mf{Vect}$, the stack of
    finite-dimensional $k$-vector spaces. In this case, $\mf{Vect}(S)$,
    for a $k$-scheme $S$, is the exact $\O(S)$-linear category of
    vector bundles over $S$, and $\mf{Vect}(k)$ is the abelian
    $k$-linear category of finite-dimensional $k$-vector spaces. As an
    algebraic stack, $\mf{Vect}$ is
$$\mf{Vect}=\spec k\mathrel{\amalg} \BGL_1\mathrel{\amalg}
\BGL_2\mathrel{\amalg}\ldots$$ 
\end{items}

\subsubsection{Algebroids}

There is a canonical sheaf of algebras
$\A\to\MM$ over $\MM$. The set of sections of $\A$ over the $S$-valued point
$x$ of $\MM$ is the $\O(S)$-algebra $\A_x=\End(x)$. For
$\MM=\mf{Vect}$, the point
$x$ is a vector bundle over $S$, and $\End(x)$ is the
$\O(S)$-module of endomorphisms of $x$.

There is also a canonical isomorphism of group sheaves $
\A^\times\to I_\MM$ over $\MM$, where $I_\MM$ is the {\em inertia stack }of
$\MM$. (Recall that the sections of $I_\MM$ over the $S$-valued point
$x$ of $\MM$ are the automorphisms of $x$, in other words the units in
the algebra of endomorphisms.)

We call a triple $(X,A_X,\iota)$ an \emph{algebroid} (see
Definition~\ref{defnalgebroid} and 
Remark~\ref{explainterm}), if $X$ is an algebraic stack,
$A_X\to X$ is a representable sheaf of finite $\O$-algebras over $X$
(or {\em finite type algebras}, as we call them, see
Definition~\ref{ftalg}), and $\iota: A^\times_X\to I_X$ is an
open immersion 
of relative group schemes over $X$, making the canonical diagram 
$$\xymatrix{ A^\times_X\rto^\iota \drto & I_X\dto\\ & \uAut_X(A_X)}$$
commute.

So $\MM$ with its canonical sheaf of algebras $\A$ is an example of
an algebroid. In this case, $\iota$ is an isomorphism, yielding what
we call a {\em strict }algebroid. 

Algebroids are generalizations of linear algebraic stacks (they are
linear over their coarse moduli spaces, if they are strict).  They are slightly more
flexible, for example, schemes can be considered as algebroids in a
canonical way.  If the algebraic stack  $X$ is the base of an
algebroid, then the connected component $I^\circ_X$ of its inertia
stack $I_X$ is the group of units in an 
algebra.  This is the main significance of algebroids for us. 

Just like algebraic stacks, algebroids form a 2-category, in which
2-fibered products exist. Whenever $(X,A_X)$ is an algebroid, and
$Y\to X$ is an {\em inert }morphism of algebraic stacks,
i.e. $I_Y^\circ=I_X^\circ|_Y$ (see Definition~\ref{define}), the stack
$Y$ is endowed with a natural 
structure of an algebroid via $A_X|_Y$. Examples of inert
morphisms include monomorphisms, and projections $Z\times X\to X$, for
schemes $Z$. A {\em locally closed immersion }of algebroids
$(Y,A_Y)\to(X,A_X)$ is a morphism
where $Y\to X$ is a locally closed immersion of algebraic stacks, such that
$A_Y=A_X|_Y$.   Every scheme $Z$ is an algebroid via the
definition $A_Z=0_Z$. 

A key observation is that if $(X,A)$ is an algebroid, then
$(I_X,I_A)$ is another algebroid. In fact, $I_A$, the inertia stack of
the stack $A$ (the total `space' of the sheaf of algebras $A$), is
equal to the subalgebra of $A|_{I_X}$ fixed under its tautological
automorphism.  We call $(I_X,I_A)$ the {\em inertia algebroid }of
$(X,A)$. It comes with a canonical morphism to $(X,A)$. There is also a
semi-simple connected version of the algebroid inertia, denoted by 
$I^{\circ,\ss}$.

\subsubsection{The motivic Hall algebra}

We define {\em stack functions }to be representable morphisms of
algebroids 
$(X,A_X)\to(\MM,\A)$, where $X$ is of finite type.  
The {\em Hall algebra }$K(\MM)$ of $\MM$ is the $\qq$-vector
space 
on the isomorphism classes of
stack functions modulo the {\em scissor relations relative $\MM$}: 
$$[X\to\MM]=[Z\to X\to \MM]+[X\setminus Z\to X\to\MM]\,,$$
and the {\em bundle relations relative $\MM$}:
$$[Y\to X\to\MM]=[F\times X\to X\to \MM]\,.$$
Here $[X\to\MM]$ denotes the Hall algebra element defined by a stack
function with base $X$. Also, $Z\to X$ is a closed immersion of
algebroids, with open complement $X\setminus Z$, and $Y\to X$ is an
{\em inert }fibre bundle (Definition~\ref{bundledef}) with special
structure group and fibre $F$, all 
endowed with their canonical algebroid structure. (For strict
algebroids, the bundle relations follow from the scissor relations.)

We have the following structures on the Hall algebra:
\begin{enumerate}
\item {\bf Module structure. } Let $K(\text{Var})$ denote the
  Grothendieck ring of varieties over $k$. We denote the motivic weight
  of the affine line by $q=[\aaa^1]\in K(\text{Var})$.  By $[Z]\cdot
  [X\to\MM]=[Z\times X\to\MM]$ we define a $K(\text{Var})$-module
  structure on $K(\MM)$.
\item {\bf Multiplication. } $[X\to\MM]\cdot[Y\to\MM]=[X\times
  Y\to\MM\times\MM\stackrel{\oplus}{\longrightarrow}\MM]$ defines a
  commutative multiplication on $K(\MM)$, and $K(\MM)$ is a
  $K(\text{Var})$-algebra with this multiplication.
\item {\bf Hall product. } Using the stack of short exact sequences in
  $\MM$, we 
  can define a Hall algebra 
  product $[X\to\MM]\ast[Y\to\MM]$ on $K(\MM)$.  For details, see
  Section~\ref{sec-filo}.  The module $K(\MM)$ is a
  $K(\text{Var})$-algebra also with respect to the Hall product.
\item {\bf Unit. } The multiplicative unit with respect to both
  products is represented 
by $1=[\Spec k \stackrel{0}{\longrightarrow}\MM]$. Via this unit, we
get an inclusion $K(\Var)\subset K(\MM)$. 
\item {\bf Inertia endomorphism. } The algebroid inertia defines an
  operator $I:K(\MM)\to K(\MM)$ via $I[X\to\MM]=[I_X\to
    X\to\MM]$. This inertia operator is linear over $K(\text{Var})$,
  and multiplicative $I(x\cdot y)=I(x)\cdot I(y)$, with respect to the
  commutative multiplication.  The same facts hold for the connected semi-simple
  inertia operators $I^{\circ,\ss}:K(\MM)\to K(\MM)$.
\end{enumerate} 

There is also a `non-representable' version of the Hall algebra, where
we drop the representability requirement for stack functions, and
simply define a stack function to be a morphism of algebroids
$X\to\MM$, with $X$ of finite type.  The representable Hall algebra is
a subalgebra (with respect to both products) of the non-representable
one.  Our results on the diagonalizability of the various operators
$I$, $I^{\circ,\ss}$, $E_n$ hold true also in the non-representable Hall
algebra, but the algebraic results on the structure of the Hall
algebra need representability.
For simplicity, we restrict ourselves therefore to the representable
case.

Usually, when defining the Hall algebra of $\MM$, one  requires the
bundle relations also for non-inert morphisms.  The connected inertia
operator does not respect such relations, and 
we therefore do not include them.

\subsubsection{Example}

A stack function $X\to \mf{Vect}$ is the same thing as an algebroid
$(X,A)$, together with a faithful representation. 

Examples of stack functions with values in $\mf{Vect}$ include
subalgebras $A\subset M_{n\times    n}$. The elements of
$K(\mf{Vect})$ defined by $A\subset M_{n\times n}$ 
and $B\subset M_{n\times n}$ are equal if and only if $A$ and $B$ are
conjugate in $M_{n\times n}$. 

The subalgebra $U$ of $K(\mf{Vect})$ with respect to the Hall
product,  generated by the $[n]=[B\GL_n\to
  \mf{Vect}]$ is free on these elements $[n]$, for $n>0$, as a unitary
$\qq$-algebra. In the literature, $U$ is known as the  Hopf algebra of
{\em non-commutative
symmetric functions}, see \cite[Example~4.1~(F)]{CartierHopf}.

(If we add the (non-inert) vector bundle relations relative $\mf{Vect}$, see,
e.g. \cite{Bridgeland-Hall-Algebras}, we get 
$$[\lambda_1]\ast\ldots\ast[\lambda_r]
=\frac{[\GL_n]}{[P(\lambda)]}[n]
=\binom{n}{\lambda_1\ldots\lambda_n}_q[n]\,.$$
Here $n=\sum \lambda_i$, and $\binom{n}{\lambda_1\ldots\lambda_n}_q$
denotes the $q$-deformed multinomial coefficient, which  gives the
motivic weight of the flag variety of type $\lambda$. We have also
denoted the parabolic subgroup of $\GL_n$ of type $\lambda$ by
$P(\lambda)$. 
Hence the $\qq$-algebra obtained by dividing $U$ by the vector bundle
relations is the  commutative polynomial algebra over $\qq$, on the
symbols $[1],[2],[3],\ldots$. This is the Hopf algebra of symmetric
functions.)

\subsubsection{The spectrum of semi-simple inertia}

The main point of this work is to study the spectral theory of the
semi-simple inertia operator $I^{\circ,\ss}$ on $K(\MM)$.

Before announcing our results, let us do a few sample
calculations. They contain some of the central ideas of this paper.
Only strict algebroids will occur, so we write $I^{\ss}$ instead of
$I^{\circ,\ss}$.

We consider $\MM=\mf{Vect}$.  The linear stack of  line bundles defines the
stack function $[\BGL_1\to\mf{Vect}]\in K(\mf{Vect})$. We have
\begin{align*}
  I^\ss[\BGL_1\to\mf{Vect}]
  &=[\GL_1^\ss\times \BGL_1\to\mf{Vect}]\\
  &=[\GL_1\times \BGL_1\to\mf{Vect}]\\
  &=(q-1)[\BGL_1\to\mf{Vect}]\,.
\end{align*}
This proves that $[\BGL_1\to\mf{Vect}]$ is an eigenvector of $I^\ss$,
with corresponding eigenvalue $(q-1)\in K(\Var)$.

Because $I^\ss$ is an algebra morphism with respect to the commutative
product, it immediately follows that every $(q-1)^r$, for $r\geq0$ is
and eigenvalue of $I^\ss$, with corresponding eigenvector
$[\BGL_1^r\to \BGL_n\to\mf{Vect}]\in K(\mf{Vect})$. 

These are not the only eigenvalues of $I^\ss$.  In fact, let us
consider the stack function of all rank 2 vector bundles
$[\BGL_2\to\mf{Vect}]$. Recall that the inertia stack of $\BGL_2$ is
the quotient stack $\GL_2/_\ad \GL_2$, where $\GL_2$ acts on itself by
the adjoint action. The semi-simple part of $\GL_2$ decomposes as
$\GL_2^{\text{ \it eq}}\sqcup \GL_2^{\text{\it neq}}$, according
to whether the 
two eigenvalues of an element of $\GL_2$  are equal or not equal. By the
scissor relations we have,
\begin{align*}
  I^\ss[\BGL_2]
  &=[\GL_2^{\text{eq}}/_\ad\GL_2]+[\GL_2^{\text{neq}}/_\ad\GL_2]\\
  &=[\Delta\times \BGL_2]+[T^\ast/_\ad N]\\
  &=(q-1)[\BGL_2]+x\,.
\end{align*}
Here $\Delta$ is the one-parameter subgroup of scalar matrices, and
$T$ is the maximal torus of diagonal matrices in $\GL_2$. Further
notation: $T^\ast=T\setminus \Delta$,  $N$ is the normalizer of $T$
in $\GL_2$, and $x=[T^\ast/_\ad N]$.

Next, we calculate $I^\ss x$.  In fact, we have
$I^\ss_{T^\ast/N}=I_{T^\ast/N}= (T^\ast\times T)/N$, by the `stabilizer
formula' for the inertia stack of a quotient stack 
$$I_{Y/G}=\{(y,g)\in
Y\times G\mid yg=y\}/G\,.$$ 

We note that $N=T\rtimes \zz_2$ acts on $T^\ast\times T$ diagonally,
via its quotient $\zz_2$ by swapping the entries of $T$.  We embed $T$
into $\aaa^2$ equivariantly with respect to $\zz_2$, and then
decompose $\aaa^2$ as 
$T \sqcup (\GL_1\times 0) \sqcup (0\times
\GL_1)\sqcup (0,0)$.  This gives
\begin{multline}\label{first}
[(T^\ast\times\aaa^2)/N]\\=[(T^\ast\times T)/N]+
[T^\ast\times (\GL_1\times 0\sqcup 0\times
  \GL_1)/N]+[T^\ast\times (0,0)/N]\,.
\end{multline}
We have a pullback diagram of algebroids
$$\xymatrix{
  (T^\ast\times \aaa^2)/N\rto\dto & T^\ast/N\dto\\
  (T^\ast\times\aaa^2)/\zz_2\rto & T^\ast/\zz_2\rlap{\,.}}$$
It shows that the vector bundle $(T^\ast\times \aaa^2)/N\to
T^\ast/N$ is a pullback of the vector bundle
$(T^\ast\times\aaa^2)/\zz_2\to T^\ast/\zz_2$.  The latter is a vector
bundle over a scheme, and is therefore Zariski-locally trivial by Hilbert's Theorem 90. The
same is then true for any pullback bundle.  Hence, we conclude that
$$[(T^\ast\times\aaa^2)/N]=q^2[T^\ast/N]\,,$$
using only the scissor relations.  So from (\ref{first}) we conclude that
$$q^2x=I^\ss x +(q-1)[T^\ast/T]+x\,,$$
which we solve for $I^\ss x$ to get:
$$I^\ss x=(q^2-1)x-(q-1)^2(q-2)[BT]\,.$$
We already know that $I^\ss[BT]=(q-1)^2[BT]$, and so we conclude that
$[\BGL_2]$, $x$ and $[BT]$ generate an $I^\ss$-invariant subspace of
$K(\mf{Vect})$, and the matrix of $I^\ss$ on this subspace is
$$\begin{pmatrix}
q-1 & 0 & 0\\
1 & q^2-1 & 0\\
0 & -(q-1)^2(q-2) & (q-1)^2
\end{pmatrix}$$
This matrix is lower triangular, with distinct scalars on the
diagonal, and is therefore diagonalizable over the field $\qq(q)$.
So on this subspace, $I^\ss$ is diagonalizable, with eigenvalues
$(q-1)$, $(q^2-1)$, and $(q-1)^2$. If we decompose $[\BGL_2]$ as a sum
of eigenvectors, we get the eigenvectors 
\begin{align*}
  &[\BGL_2]-\sfrac{1}{q(q-1)}x-\sfrac{1}{q}[BT]
  &&\text{with eigenvalue} &&(q-1)\,,\\
  &\sfrac{1}{q(q-1)}x-\sfrac{q-2}{2q}[BT]
  &&\text{with eigenvalue} &&(q^2-1)\,,\\
  &\sfrac{1}{2}[BT]
  &&\text{with eigenvalue} &&(q-1)^2\,.
\end{align*}
A very important observation is that when we add together the
eigencomponents whose eigenvalues have the same order of vanishing at
$q=1$, we get coefficients in $\qq$, instead of $\qq(q)$.
In the above example, we add together the components of $[\BGL_2]$
with eigenvalues $(q-1)$ and $(q^2-1)$ to obtain
$[\BGL_2]-\frac{1}{2}[BT]$. 

Another important observation is that diagonalizing $I^\ss$ does not,
despite appearances, require us to invert $(q-1)$. In fact, the
algebroid $x$ appearing in the above argument is divisible by $(q-1)$,
although the quotient is not a strict algebroid any longer.  This is,
in fact, the reason for considering non-strict algebroids at all. 

(In the above calculations, we have suppressed the algebra part $A$ of
the various algebroids $(X,A)$. We leave it to the reader to supply
the natural algebra for each algebroid mentioned.)

\subsubsection{Results}

We now summarize the main results of this paper.

\begin{introthm}[Diagonalizability of $I^{\circ,\ss}$]
The operator $I^{\circ,\ss}$ on $$K(\MM)(q)=K(\MM)\otimes_{\qq[u]} \qq(q)$$ is
diagonalizable, the eigenvalues are indexed by partitions $\lambda$,
and the eigenvalue corresponding to the partition $\lambda$ is the
cyclotomic polynomial  $$\Q(\lambda)=\prod (q^{\lambda_i}-1)\,.$$
In other words, we have a direct sum decomposition
$$K(\MM)(q)=\bigoplus_{\lambda} K^\lambda(\MM)\,,$$
into subspaces invariant under $I^\ss$, and $I^\ss|_{K^\lambda(\MM)}$
is multiplication by $\Q(\lambda)$. 
\end{introthm}

The same theorem holds for the operator $I^{\ss}$ in the context of
strict algebroids. We also prove a stronger version avoiding
denominators divisible by $(q-1)$, but this version only works for
algebroids. 

The proof of this theorem is a generalization of the above sample
calculation for the stack of rank 2 vector bundles.  One goal of
Section~1 of the paper is to set up the necessary notation.

\begin{introthm}[Graded structure of $K(\MM)$]
There is a direct sum decomposition
\begin{equation}\label{dsdejk}
  K(\MM)=\bigoplus_{r\geq0} K^r(\MM)\,,
\end{equation}
such that
$$K^r(\MM)(q)=\bigoplus_{\ord_{q=1}\Q(\lambda)=r}K^\lambda(\MM)\,.$$
Moreover, the commutative product is graded with respect
to~(\ref{dsdejk}).
\end{introthm}

Again, the same theorem holds in the context of strict algebroids. 

The fact about the gradedness of the commutative product is 
expected from the fact that the semi-simple inertia  respects
the commutative product (it follows from this fact over $\qq(q)$, but
is true over $\qq$). 

Geometrically, the descending filtration $K^{\geq r}(\MM)$ induced by
the grading (\ref{dsdejk}) can be described as follows: $K^{\geq 
  r}(\MM)$  is the $\qq$-span of all stack functions $[X\to \MM]$, for  
which the algebra of global sections $\Gamma(X,A_X)$ admits at least
$r$  
orthogonal non-zero central idempotents, where $A_X$ is the algebra of the
algebroid  $(X,A_X)$. 

The direct summands $K^r(\MM)$, are the common eigenspaces of the
family of commuting
operators $(E_n)_{n\geq0}$, where $E_n(X)$ is the stack of decompositions
of $1\in A_X$ into a sum of $n$ orthogonal labelled idempotents. The
eigenvalues of the operators $E_n$ are integers, and the whole family
of operators $(E_n)$ is diagonalizable over $\qq$. The proof of this
fact proceeds by proving that the $(E_n)$ preserve the descending
filtration described geometrically above, and have distinct integer
diagonal entries.

It turns out that the ascending filtration $K^{\leq n}(\MM)$
associated to the grading in the above theorem can be described as
$$K^{\leq n}(\MM)=\ker E_{n+1}\,.$$
Let us also point out that
$$K^0(\MM)=K(\DM)\,,$$
and 
$$K^0(\MM)=K(\Var)\,,$$
in the context of strict algebroids. 

If we denote by $\pi_r:K(\MM)\to K(\MM)$ the projection operator onto
the summand $K^r(\MM)$, and form the generating series
$\pi_t=\sum_{r\geq0} \pi_r t^r$,
then we have
\begin{equation}\label{projform}
  \pi_t=\sum_{n\geq0} \binom{t}{n} E_n\,.
\end{equation}

All the above results could be proved for pairs $(X,A)$ of algebraic
stacks $X$ endowed with finite type algebras $A$, instead of
algebroids or strict algebroids. One simply replaces $I^{\circ,\ss}$
by $A^{\times,\ss}$. 

\begin{introthm}[Filtered nature of the Hall algebra]\label{introthm3}
The Hall product is filtered with
respect to the filtration $K^{\leq r}(\MM)$, induced by the
grading~(\ref{dsdejk}).  Moreover, for the associated
graded algebra we have
$$\mathop{\rm gr}\big(K(\MM),\ast\big)=\big(K(\MM),\cdot\,\big)\,.$$
In other words, if $x\in K^{\leq r}(\MM)$ and $y\in K^{\leq s}(\MM)$,
then $x\ast y\in K^{\leq r+s}(\MM)$, and
$$x\ast y\equiv x\cdot y \mod K^{\leq r+s-1}(\MM)\,.$$
\end{introthm}

The proof of this theorem uses not
much more than some simple combinatorics involving relabelling of
direct sum decompositions, and compatibilities between direct sum
decompositions of short exact sequences and splittings of short exact
sequences. 

The theorem implies that the one parameter family of algebras
$\big(\K(\MM),\ast\big)$ given by
the Rees construction
$$\K(\MM)=\bigoplus_{n\geq0} t^n K^{\leq n}(\MM)\,,$$
is a deformation quantization of (i.e., a one-parameter flat family of
algebras with special fibre)  the commutative algebra
$\big(K(\MM),\cdot\,\big)$. 
Hence the graded algebra $\big(K(\MM),\cdot\,\big)$ inherits a Poisson
bracket $\{\,,\}$ of degree $-1$. In particular,  $K^1(\MM)$ is a Lie
algebra, and it turns out that the Lie bracket on $K^1(\MM)$ is equal to the
commutator bracket associated to $\ast$.

Following Joyce~\cite{JoyceII}, we call $K^1(\MM)$ the Lie algebra of
{\em virtually indecomposable }elements of $K(\MM)$, with the notation
$K^\vir(\MM)=K^1(\MM)$.

We denote the projection onto $K^\vir(\MM)$ by $\pi^\vir$.  With this
notation, we have, as a special case of~(\ref{projform}),
$$\pi^\vir=\sum_{n>0}\frac{(-1)^{n+1}}{n} E_n\,.$$
In terms of eigenspaces of semi-simple inertia, we have
$$K^\vir(\MM)(q)=K^{(q-1)}(\MM)\oplus K^{(q^2-1)}(\MM)\oplus
K^{(q^3-1)}(\MM)\oplus\ldots $$

\begin{introthm}[Hall algebra logarithms]\label{introthm4}
Let $\NN\subset\MM$ be a `small enough' substack, closed under
extensions and direct summands, and not intersecting $\spec
k\stackrel{0}{\longrightarrow}\MM$. Then
$$\epsilon_t[\NN]=\sum_{n\geq0} \binom{t}{n}[\NN]^{\ast n}\in
\hat\K(\MM)_+\,.$$
In particular, the $\ast$-logarithm
$$\epsilon[\NN]=\sum_{n\geq1}\frac{(-1)^{n+1}}{n}[\NN]^{\ast n}\in
\hat K^\vir(\MM)_+\,,$$
is virtually indecomposable.
\end{introthm}

For the precise definition of `small enough', see
Section~\ref{seceps}. For example, if $\MM=\mf{Coh}_Y$ for a curve
$Y$, we could take $\NN$ to consists of all non-zero semi-stable
vector bundles of a fixed slope.  Since $\NN$ is typically not of
finite type, to make sense of $[\NN]$, we have to pass to a certain
completion $\hat K(\MM)_+$ of $K(\MM)$. See Section~\ref{seceps} for
details.

\begin{introthm}[No poles theorem]\label{introthm5}
Let $K(\St)$ be the $K$-ring of algebraic stacks, modulo all bundle
relations  with special structure group (inert or not). Consider the map
\begin{align*}
\int:K(\MM)&\longrightarrow
K(\St)\\
[(X,A)\to(\MM,\A)]&\longmapsto [X]\,,
\end{align*}
which forgets the structure map to $(\MM,\A)$, and the algebroid
structure over the stack $X$. If $x\in K^{\leq r}(\MM)$, then
$(q-1)^r\int x$
is a regular element of $K(\St)$, i.e., under the identification
$$K(\St)=
K(\Var)[\sfrac{1}{q},\sfrac{1}{q-1},\sfrac{1}{q^2-1},\ldots]\,,$$ 
it can be written with a denominator that does not vanish at $q=1$. 
\end{introthm}

Moreover, suppose we have a grading monoid $\Gamma$ for $\MM$:
$$\MM=\coprod_{\gamma\in \Gamma} \MM_\gamma\,.$$
We say that $x$ has `degree' $\gamma$ if $x\in \MM_\gamma$.
We need $\Gamma$ to be  endowed with a $\zz$-valued bilinear form
$\chi$, such that 
for every object $x$ in $\MM_\gamma$, and $y$ in
  $\MM_\beta$,
\begin{items}
\item  every extension of $y$ by $x$ is in $\MM_{\beta+\gamma}$,
\item the stack of extensions of $y$ by $x$ is the quotient of a
  vector space $E_1$ by another vector space $E_0$, acting trivially,
  such that $\dim E_0-\dim   E_1=\chi(\beta,\gamma)$.
\end{items}
For the precise assumptions, see Section~\ref{hereditary}. They  are
satisfied if $\MM=\mf{Rep}_Q$, or if 
$\MM=\mf{Coh}_Y$ and $Y$ is a smooth curve, or more generally, if
$\MM$ is {\em hereditary}.

Then we have a commutative diagram
$$\xymatrix{
  \big(\K(\MM),\ast\big)\rrto^{t\longmapsto 0}\dto_{\int}^{t\mapsto (q-1)} &&
  \big(K(\MM),\cdot\,\big)\dto_{\int}^{q\mapsto 1}\\
  K(\Var)_\reg[\Gamma]\rrto^{q\longmapsto1} &&
  K(\Var)/(q-1)[\Gamma]\rlap{\,.}}$$
Here we use the $\Gamma$-graded integral, which is essentially a
generating function, indexed by $\Gamma$, of the integrals of the
components of degree $\gamma\in \Gamma$ of a given stack function.

The upper horizontal arrow in this diagram is the specialization map,
which exists because of Theorem~\ref{introthm3}.  The left vertical arrow
exists by the  `No Poles Theorem'~\ref{introthm5}. It is a standard fact,
that this map is an algebra morphism, i.e., respects the
$\ast$-product, if the target $K(\Var)_\reg[\Gamma]$ is endowed with
its q-deformed product twisted by $\chi$. It is a formal consequence of the
commutativity of this diagram that the right vertical map is a
morphism of Poisson algebras, if the target $K(\Var)/(q-1)[\Gamma]$ is
endowed with its bracket induced by $\chi$.

In particular, we deduce that
\begin{align}
  K^\vir(\MM)&\longrightarrow K(\Var)/(q-1)[\Gamma]\nonumber\\
  x_\gamma&\longmapsto \big((q-1)\textstyle{\int}\label{intgam}
  x_\gamma\big)\big|_{q=1} u^\gamma =\res_{q=1}(\int  x_\gamma) u^\gamma
\end{align}
is a morphism of Lie algebras.

The proof of the No Poles Theorem combines the above results about the
diagonalizability of $I^{\circ,\ss}$, especially in its form avoiding
denominators divisible by $(q-1)$, with the result that for an
algebroid $(X,A)$, the stack $I^{\circ,\ss}_X$ has
regular motivic weight, i.e., 
$$[I^{\circ,\ss}_X]\in K(\St)_\reg\,.$$
We think of this as a motivic version of Burnside's lemma.  The more
natural looking conjecture that for an algebraic stack $X$, the
motivic weight of $I_X$ is contained in $K(\Var)\subset K(\St)$ is
most likely false. 
  
\subsubsection{Discussion}

To produce counting invariants for $\MM$, we need to look for
subcategories $\NN\subset\MM$, to which we can apply
Theorem~\ref{introthm4}, giving us virtually indecomposable elements
$\epsilon[\NN]$, to which we can apply the integral~(\ref{intgam}),
yielding generating functions with coefficients in $K(\Var)/(q-1)$.
We can apply the Euler characteristic to these elements of
$K(\Var)/(q-1)$ to obtain rational numbers.
In the hereditary case, the fact 
that (\ref{intgam}) is a morphism of Lie algebras, gives the relations
among generating functions one is interested in.  This leads to
wall-crossing formulas, and other results. For details we refer to the
works of Joyce,  Joyce-Song, and many others.

To deal with the Calabi-Yau 3 case one needs to insert the correct
motivic vanishing cycle weights, to define the integral. This is done
by Joyce and Song in~\cite{JoyceSong}.

The work of Joyce on configurations in abelian categories contains
results which correspond to ours, but his definitions are more ad-hoc.
In fact, one reason for writing the present article is to give a more
conceptual treatment of Joyce's results. We do not prove that our
notion of `virtual indecomposable object' coincides with Joyce's
(except for in the case of $\MM=\mf{Vect}$, see the appendix), but
instead prove that our notion has the same properties as Joyce's and
is just as useful.  (Of course, the counting invariants we obtain are
the same as the ones obtained by Joyce, as they do not depend on the
definition of virtual indecomposable object.)

We think of the Lie algebra $K^\vir(\MM)$ as an analogue of the Lie
algebra of primitive elements in a cocommutative Hopf algebra. In fact,
one may ask whether $\big(K(\MM),\ast\big)$  is
equal to the 
universal enveloping algebra of the Lie algebra $K^\vir(\MM)$. To
deduce such a statement from structure theorems for Hopf algebras, one
would need to enhance $K(\MM)$ to a 
cocommutative Hopf algebra.  We have not been able to construct the
necessary coproduct. We view the family of operators $(E_n)$ as
somewhat of a replacement. It lets us prove at least some of the
results expected of a cocommutative Hopf algebra, in particular
Theorems~\ref{introthm3} and~\ref{introthm4}.

\subsubsection{Acknowledgements}

The idea to consider the inertia stack as an operator on $K$-groups of
stacks and motivic Hall algebras, and to study its eigenspace
decomposition to understand Joyce's work in a more conceptual fashion
is due to Tom Bridgeland.  In particular, the conjecture that the
semi-simple inertia operator is diagonalizable is due to him.  We
would like to thank Tom Bridgeland for sharing his ideas with us. We
would also like to thank Dominic Joyce and Arend Bayer for fruitful
discussions.

\section{Linear algebraic stacks and algebroids}

\subsection{Algebraic stacks}

Let us briefly summarize our conventions about algebraic stacks. 

We choose a noetherian base ring  $R$ (commutative and with
unit), and we fix our base 
category $\SSS$ to be the category of $R$-schemes, endowed with the
\'etale topology.\comment{make sure we don't inadvertently assume $R=\cc$}
Over $\SSS$ we have a canonical sheaf of $R$-algebras  $\O_\SSS$, it is represented
by $\aaa^1=\aaa^1_{\Spec R}$, and called the {\em structure sheaf}. 

We will assume our algebraic stacks to be locally of finite type.
\comment{this is to make coherent sheaves well-behaved, maybe locally
  noetherian  would be more natural?}
Thus, an {\em algebraic stack} is a stack over the site $\SSS$, which
admits a presentation by a smooth groupoid $X_1\rightrightarrows X_0$, where $X_0$
and $X_1$ are algebraic spaces, {\em locally of finite type }over $R$, the
source and target morphisms $s,t:X_1\to 
X_0$ are smooth, and the diagonal $X_1\to X_0\times X_0$ is of finite
type. In fact, all algebraic stacks we encounter will have affine
diagonal. 

By a {\bf stratification }of an algebraic stack $X$, we mean a
morphism of algebraic stacks $X'\to X$, which is a surjective
monomorphism, and which admits a finite decomposition $X'=\coprod_i
X_i$, such that every $X_i\to X$ is a locally closed embedding of
algebraic stacks.

If $G$ is an algebraic group acting on the algebraic space $X$, we will
denote the quotient stack by $X/G$, because we fear the more common
notation $[X/G]$ would lead to confusion with the notation for
elements of various $K$-groups of schemes and stacks. 

Suppose $G\to X$ is a relative group scheme over the stack $X$.  The
{\em connected component }of $G$, notation $G^\circ$, is the subsheaf
of $G$, defined by requiring a section $g\in G(S)$ to factor through
$G^\circ(S)$, if and only if for all points (equivalently geometric
points) $s$ of $S$, we have $g(s)\in G^\circ_s$. If $G\to X$ is
smooth, the connected component $G^\circ\subset G$ is represented by an
open substack of $G$, which is a smooth group scheme with
geometrically connected fibers over $X$. (See \cite{SGA3},
Expos\'e~VI$_{\text{B}}$, Th\'eoreme~3.10.)

If the inertia stack $I_X$ of an algebraic stack $X$ is smooth over
$X$, the connected component $I^\circ_X$ exists. We can apply the {\em
  rigidification construction }(see for example \cite{TameStacks},
Appendix) to $I_X^\circ\subset I_X$, and obtain a (uniquely
determined) Deligne-Mumford stack $\ol X$, together with a morphism
$X\to \ol X$, making $X$ a {\em connected gerbe }over $\ol X$,
(which means that the relative inertia of $X$ over $\ol X$ has
connected fibres). The structure morphism $X\to \ol X$ is smooth. 

A gerbe $X\to \ol X$ is an {\em isotrivial gerbe}, if it admits a
section over a finite \'etale $\ol X$-stack. If $X\to \ol X$ is
a smooth gerbe over a Deligne-Mumford stack, there exists a
stratification  $\ol X'\to \ol X$, such that 
the restriction of the gerbe $X$ to each piece of $\ol X'$ is
isotrivial.  (This follows from the fact that a quasi-finite morphism
of Deligne-Mumford stacks is generically finite. This, in turn,
follows from Zariski's main theorem \cite[Section~16]{LMB}.)

Let us also remark that every Deligne-Mumford stack admits a
stratification by {\em integral normal }Deligne-Mumford stacks,
although we do not use this fact.

\subsubsection{Sheaves on algebraic stacks}

We need to clarify the notions of vector bundle, coherent sheaf, and representable
sheaf of $\O_X$-modules, and how they relate to each
other.

In particular, an algebraic stack $X$ is a fibered category 
$X\to \SSS$. The category $X$ inherits a topology from $\SSS$, called
the \'etale topology, and $X$ endowed with this topology is the {\em
  big \'etale site }of $X$.  Sheaves over $X$ are by definition
sheaves on this big \'etale site. For example, $\O_\SSS$ induces a
sheaf of $R$-algebras on $X$, which is denoted by $\O_X$, and
called the {\em structure sheaf }of $X$.  It is represented by
$\aaa^1_X$.

A
sheaf $\F$ over $X$ induces for every object $x$ of $X$ lying over
the object $U$ of $\SSS$  a sheaf on the usual
(small) \'etale site $U_{\text{\'et}}$ of the scheme $U$, denoted $\F_U$. Moreover, for
every  morphism $\alpha:y\to x$ lying over $f:V\to U$, we obtain a
morphism of sheaves $\alpha^\ast:f^{-1}(\F_U)\to \F_V$. (The
$\alpha^\ast$ satisfy an obvious cocycle condition, and the condition
that they are isomorphisms if $f$ is \'etale.) For example, 
the structure sheaf $\O_X$ induces the structure sheaf on $U_{\text{\'et}}$, for
every such $x/U$. The data of the
small \'etale sheaves $\F_U$, together with the compatibility morphisms
$\alpha^\ast$, satisfying the two parenthetical conditions, is equivalent to the
data defining $\F$ (see \cite[Exp.\ IV, 4.10]{sga4}).   The functor
$\F\mapsto \F_U$ is the sheaf pullback 
morphism of a morphism of sites $U_{\text{\'et}}\to X$, from the small \'etale
site of $U$ to the big \'etale site of $X$. In particular,
$\F\mapsto\F_U$ is exact. Both $U_{\text{\'et}}$ and $X$ are ringed
sites, and $\F\mapsto\F_U$ is also the sheaf of modules pullback of
the morphism of ringed sites $U_{\text{\'et}}\to X$. Therefore, the
functor $\F\to \F_U$ is also exact when considered as a functor from
the category of big sheaves of $\O_X$-modules to the category of small
sheaves of $\O_{U_{\text{\'et}}}$-modules.

If $\F$ and $\G$ are sheaves of $\O_X$-modules, then $\shom(\F,\G)$ is
again a sheaf of $\O_X$-modules.  In particular, for a sheaf of
$\O_X$-modules, we have the dual $\F^\vee=\shom(\F,\O_X)$. 

Note that,
in general, the natural 
homomorphism $\shom(\F,\G)_U\to\shom(\F_U,\G_U)$ is not an
isomorphism, see below~(\ref{ex1}).

\subsubsection{Coherent sheaves}

A sheaf $\F$ of $\O_X$-modules is {\em locally coherent}, if for every
$x/U$ the sheaf $\F_U$ is a coherent sheaf of
$\O_{U_\et}$-modules. (This terminology is inspired by  
\cite[\href{http://stacks.math.columbia.edu/tag/06WJ}{Tag~06WJ}]{SP}.)
It is {\em cartesian}, if all compatibility morphisms
$\alpha^\ast:f^\ast \F_U\to \F_V$ are isomorphisms of sheaves of
$\O_{V_{\text{\'et}}}$-modules. A sheaf which is both locally coherent
and cartesian is {\em coherent}.

For example, a groupoid presentation $X_1\rightrightarrows X_0$
of $X$, and a coherent sheaf $\F_0$ on $X_0$, together with an
isomorphism $s^\ast\F_0\to t^\ast \F_0$, satisfying the usual cocycle
condition on $X_2=X_1\times_{X_0}X_1$, give rise to a coherent sheaf
on $X$. 

The sheaf of sections of a vector bundle over $X$ is coherent. In
fact, the notion of {\em vector bundle }and {\em locally free coherent
  sheaf }are equivalent, and we will use them interchangeable, even
though the two categories are {\em anti-equivalent}.  The cokernel of
a homomorphism of vector bundles is coherent.  In fact, every cokernel
of a homomorphism of coherent sheaves is coherent.

If the cokernel of a homomorphism of vector bundles is locally free,
we call the homomorphism a {\bf strict }homomorphism of vector
bundles. For a strict homomorphism of vector bundles, the image and
the kernel, as well as the cokernel are locally free.

A {\em strict monomorphism }of vector bundles is a strict homomorphism
whose kernel is zero. A homomorphism of vector bundles is a strict
monomorphism/an epimorphism,  if and only if over every
geometric point of $X$, the induced linear map is
injective/surjective. A homomorphism of vector bundles, which is a
monomorphism of sheaves, is a strict monomorphism of bundles.

Let $\phi:E\to F$ be a homomorphism of vector bundles over the
algebraic stack $X$. The flattening stratification $X'\to X$ of $\cok\phi$
serves also as {\em strictening stratification }for $\phi$.  This means
that an object of $X(S)$ lifts to $X'(S)$, if and only if $\phi_S$
is strict.

In general, the kernel (in the category of big sheaves of
$\O_X$-modules) of a homomorphism of vector bundles is locally
coherent, but not coherent.

By
\cite[\href{http://stacks.math.columbia.edu/tag/06WK}{Tag~06WK}]{SP},
a sheaf of $\O_X$-modules $\F$ is coherent if and only 
if there exists a smooth covering $X_i\to X$ of $X$ by finite type
affine schemes $X_i$, such that for every $i$, the restriction $\F_i$
of $\F$ to the big \'etale site of $X_i$ is isomorphic to the cokernel
of a homomorphism of vector bundles.

\begin{prop}\label{cohprop}
Suppose that $\F$ is a coherent sheaf on the algebraic stack $X$. Then 
\begin{items}
\item for every $x/U$, we have $(\F^\vee)_U=(\F_U)^\vee$,
\item $\F^\vee$ is locally coherent,
\item $\F^\vee$ is represented by a an algebraic stack, which is of
  finite type and affine over $X$, namely $\spec_X\Sym_{\O_X}\F$,
\item the canonical homomorphism $\F\to \F^{\vee\vee}$ is an
  isomorphism of sheaves of $\O_X$-modules,
\item if $\F$ is locally free coherent, then  $\F^\vee$ is a vector bundle.
\end{items}

Moreover, the functor $\F\mapsto \F^\vee$ is a fully faithful functor from
  the category of coherent sheaves to the category of locally coherent
  sheaves of $\O_X$-modules. It maps right exact sequences of coherent
  sheaves to left exact sequences of locally coherent sheaves. 
\end{prop}
\begin{proof}
The first claim follows directly from the definitions, the main fact
being the adjunction $$\Hom_{\O_U}(\F_U,f_\ast\O_V)=\Hom_{\O_V}(f^\ast
\F_U,\O_V)=\Hom_{\O_V}(\F_V,\O_V)\,,$$ for every morphism of
$X$-schemes $f:V\to U$.  

The second claim follows from the first (see
also \cite[\href{http://stacks.math.columbia.edu/tag/06WM}{Tag
    06WM}]{SP}).

For the third claim,  see \cite[(14.2.6)]{LMB}. 

For the fourth claim, consider the sheaf of $\Gm$-equivariant
$X$-morphisms from $\Spec_X \Sym_{\O_X}\F$ to $\aaa^1_X$, denoted by
$\uHom_{\Gm}(\F^\vee,\aaa^1)$. It is equal to the sheaf of
homomorphisms of graded $\O_X$-algebras from $\O_X[t]$ to
$\Sym_{\O_X}\F$, hence equal to $\F$. But sections of $\F$ give rise
to $\O_X$-linear homomorphisms $\F^\vee\to \O_X$, not just
$\Gm$-equivariant ones.  Hence 
$$\F=\uHom_{\Gm}(\F^\vee,\aaa^1)=\shom_{\O_X}(\F^\vee,\O_X)=\F^{\vee\vee}\,.$$

The fifth claim is clear.

The `moreover' 
follows from the fact that we can reconstruct $\F$ from
$\F^\vee=\Spec_X\Sym_{\O_X}\F$, as the degree one part of the the
graded sheaf of $\O_X$-modules $\pi_\ast(\O_{\F})$, where $\pi:\F\to X$
is the projection morphism.
\end{proof}

\subsubsection{Representable  sheaves of modules}

If  $\phi:E\to F$ is a homomorphism of vector bundles
over $X$, then $\ker\phi$, constructed in the category of big sheaves,
is a representable sheaf of $\O_X$-modules.
In fact, $\ker\phi$ is equal to the fibered product of stacks
$$\xymatrix{
\ker\phi\rto\dto & X\dto^0\\
E\rto^\phi & F\rlap{\,.}}$$
Sheaves such as $\ker\phi$ belong to a class of $\O_X$-modules {\em
  dual } to coherent sheaves.

\begin{prop}\label{eqcon}
Let $\F$ be a sheaf of $\O_X$-modules.  The following are equivalent:
\begin{items}
\item there exists a coherent sheaf $\N$, such that $\F$ is isomorphic
  to $\N^\vee$,
\item there exists a smooth cover $X_i\to X$ of $X$ by finite type
  affine schemes $X_i$, such that, for every $i$,  the restriction $\F_i$ of $\F$ to
  the big \'etale site of $X_i$ is isomorphic to the kernel of a
  homomorphism of vector bundles over $X_i$.
\end{items}
\end{prop}
\begin{proof}
The fact that (i) implies (ii), follows from the results proved
above. So let us indicate how to prove that (ii) implies (i).

Let us first assume that $\F$ is isomorphic to the kernel of a
homomorphism of vector bundles $E_0\to E_1$. One checks that $\F$ is
then represented by $\Spec_X\Sym_{\O_X}\cok(E_1^\vee\to
E_0^\vee)$. Thus $\F$ is isomorphic to the dual of the coherent sheaf
$\cok(E_1^\vee\to E_0^\vee)$. 

Now suppose that $\F$ is locally isomorphic to the kernel of a
homomorphism of vector bundles. It suffices to prove that $\F^\vee$ is
coherent, and that  $\F\to
\F^{\vee\vee}$ is an isomorphism. Both claims can be checked locally,
and are true for duals of coherent sheaves.
\end{proof}
 
\begin{defn}
We call a sheaf of $\O_X$-modules {\bf locally coherent
  representable}, if any of the two equivalent conditions of
Proposition~\ref{eqcon} is satisfied. The terminology is justified by
Proposition~\ref{lgken}, below. 
\end{defn}

In other words, the category of locally coherent representable sheaves
over $X$ is the essential image of the fully faithful functor
mentioned in Proposition~\ref{cohprop}.  We therefore have an
equivalence of categories
\begin{align}\label{dualfun}
(\text{coh.\ sheaves over $X$})&\longrightarrow(\text{loc.\ 
    coh.\ repr.\ sheaves over $X$})\\
\F&\longmapsto \F^\vee\,.\nonumber
\end{align}

The following proposition summarizes facts about locally coherent
representable sheaves, which all follow easily from facts mentioned
above.  

\begin{prop}\label{lgken}
Let $\F$ be a locally coherent representable sheaf over the algebraic
stack $X$.  Then 
\begin{items}
\item the sheaf $\F$ is locally coherent,
\item the sheaf $\F^\vee$ is coherent,
\item the canonical homomorphism $\F\to\F^{\vee\vee}$ is an
  isomorphism of sheaves of $\O_X$-modules,
\item the sheaf $\F$ is representable by an algebraic stack $F\to X$,
  which is of finite type and affine over $X$,
\item in fact, $\F=\spec_X\Sym_{\O_X}\F^\vee$.
\end{items}

Moreover, the functor $\F\mapsto \F^\vee$ is an essential inverse to
the functor (\ref{dualfun}). It maps left exact sequences of locally
coherent representable sheaves to right exact sequences of coherent
sheaves. \nolinebreak$\Box$
\end{prop}

\begin{prop}\label{bdlstr}
Let $\F$ be a locally coherent representable sheaf over the finite
type algebraic stack 
$X$. There is a unique stratification $X'\to X$, with the property
that an $X$-scheme $S$ factors through $X'$, if and only if $\F|_S$
is a vector bundle. More precisely, 
$X'=\coprod_{n\geq0} X_n$, and  $X_n\to X$ is a locally closed
immersion of algebraic stacks, with the property that $S\to X$ factors
through $X_n$ if and only if $\F|_S$ is a vector bundle of rank $n$.
\end{prop}
\begin{proof}
  The sought after  stratification is the flattening stratification of
  the coherent sheaf $\F^\vee$.
\end{proof}

\begin{ex}\label{ex1}
Consider $X=\aaa^1$, with coordinate $t$, and let
$\C$ be the cokernel of the homomorphism of vector bundles
$t:\aaa^1_X\to\aaa^1_X$. It is the skyscraper sheaf of the origin,
considered as a coherent sheaf on $X$, and extended to a big sheaf
over $X$ in
the usual way. The sheaf $\C$ is an example of a coherent sheaf which
is not representable.

Let $\K$ be the kernel of
$t:\aaa^1_X\to\aaa^1_X$.  This is locally coherent representable, but
not cartesian, hence not coherent. 

Note that $\C^\vee=\K$.  This shows that $\F^\vee$ may not be
coherent, even if $\F$ is. 

Note also, that $\K^\vee=\C$, which shows that $\F^\vee$ may not
be representable, even if $\F$ is. 

Finally, note that  $(\K_X)^\vee=0^\vee=0$, but
$(\K^\vee)_X=\C_X$  is the structure sheaf of the origin in $X$, considered
as a skyscraper sheaf on $X_{\text{\'et}}$, which is not zero. This
gives an example  where $(\F^\vee)_U\not=(\F_U)^\vee$. 
\end{ex}

\begin{rmk}
Of course the category of coherent sheaves on an algebraic stack $X$
has kernels and internal homs, but they do not agree with those
in the category of  big sheaves, which we considered above. It is
therefore important to specify the 
context, when dealing with kernels or duals in the category of
coherent sheaves.  Unless specified otherwise, we will always consider sheaves
of $\O_X$-modules as big sheaves.\comment{is this actually true?}
\end{rmk}

\subsection{Linear algebraic stacks}

We will review the definition of linear algebraic stacks, and some
basic constructions.  For definitions and basic properties of fibered
categories we refer the reader to \cite[Expos\'e VI]{SGA1}. The
material here is presumably known, but we could not find a suitable
reference. 

Suppose\comment{Algebraic stacks have roman letters $X$, etc,
  but I'll let linear stacks have fraktur letters $\XX$, etc}
 $\XX\to\SSS$ is a category  over $\SSS$.  We write $\XX(S)$
for the fiber of $\XX$ over the object $S$ of $\SSS$. If $f:S'\to S$ is
a morphism in $\SSS$, and $x'\in \XX(S')$ and $x\in \XX(S)$ are
$\XX$-objects lying over $S'$ and $S$, respectively, we write
$\Hom_f(x',x)$ for the set of morphisms from $x'$ to $x$ in $\XX$,
lying over $f$.  For $S'=S$ and $f=\id_S$, we write
$\Hom_S(x',x)$. 

Recall that a morphism $\alpha:x'\to x$ lying over $f:S'\to S$ is {\em
  cartesian}, if for every object $x''$ of $\XX(S)$, composition with
$\alpha$ induces a bijection
$\Hom_S(x'',x')\longiso\Hom_f(x'',x)$. Recall further that
$\XX\to\SSS$ is a {\em fibered category}, if every composition of cartesian
morphisms is cartesian, and if for every $f:S'\to S$ in $\SSS$, and
every $x$ over $S$, there exists a cartesian morphism over $f$ with
target $x$. A {\em cartesian functor }between categories over $\SSS$
is one that preserves cartesian morphisms. 

If $\XX$ is a fibered category over $\SSS$, the subcategory of $\XX$, consisting
of the same objects and all cartesian morphisms is a category fibered
in groupoids over $\SSS$.  We denote it by $\XX_\cfg$, and call it the
{\em underlying category fibered in groupoids}.

\begin{defn} 
A category $\XX$ over $\SSS$ is an {\bf $\O$-linear category over }$\SSS$,
if for every $f:S'\to S$ in $\SSS$ and all $x'\in\XX(S')$,
$x\in \XX(S)$, the set $\Hom_f(x',x)$ is endowed with the structure of
an $\O(S')$-module,  in such a way that for every pair of morphisms
$g:S''\to S'$, $f:S'\to S$, and every triple of objects
$x''\in \XX(S'')$, $x'\in\XX(S')$, $x\in \XX(S)$, the composition
$$\Hom_f(x',x)\times\Hom_g(x'',x')\longrightarrow\Hom_{f\circ
  g}(x'',x)$$ 
is $\O(S')$-bilinear.

An {\bf $\O$-linear functor }$F:\XX\to \YY$ between $\O$-linear categories is a
functor of categories over $\SSS$, such that for every $f:S'\to S$,
and all $x'\in \XX(S')$, $x\in \XX(S)$ the map $\Hom_f(x',x)\to
\Hom_f\big(F(x'),F(x)\big)$ is $\O(S')$-linear. 
\end{defn}

Assume given an $\O$-linear fibered category $\XX$ over $\SSS$. 
Pullback in $\XX$ is $\O$-linear,
i.e., if $f:S'\to S$ is a morphism in $\SSS$, and $x,y\in \XX(S)$ are
objects with pullbacks $x',y'\in \XX(S')$, the pullback map
$f^\ast:\Hom_S(x,y)\to\Hom_{S'}(x',y')$ is $\O(S)$-linear. 

So if we fix objects $x,y\in \XX(S)$, the
presheaf $\uHom_S(x,y)$ over 
the usual small \'etale site of $S$, defined by
$\uHom_S(x,y)(T)=\Hom_T(x|_T,y|_T)$, for every \'etale $T\to S$, is a presheaf of
$\O_{S_{\text{\'et}}}$-modules. Moreover, for {\em any }morphism $f:S'\to S$ in
$\SSS$, 
we have a natural homomorphism of presheaves of $\O_{S}$-modules 
$\uHom_S(x,y)\to f_\ast\uHom_{S'}(x',y')$. Put together, the small
presheaves $\uHom_{T}(x,y)$, as $T\to S$ varies over the big \'etale site of the
scheme $S$, form a big presheaf, which we denote by $\uHom(x,y)$. 

\begin{defn}\label{defn-lin}
A {\bf linear algebraic stack }is an $\O$-linear
fibered category $\XX$ over $\SSS$, such that 
\begin{items}
\item for every object $S\in\SSS$, and every pair $x,y\in \XX(S)$, the
 presheaf $\uHom(x,y)$ on the big \'etale site of the scheme $S$ is a
 locally coherent
 representable sheaf of $\O_S$-modules,
\item the underlying category fibered in groupoids 
  $\XX_\cfg\to\SSS$  is an algebraic stack over $R$ (locally of finite
  type).
\end{items}
A {\bf morphism }of linear algebraic stacks is an $\O$-linear cartesian
functor over $\SSS$.
\end{defn}

\begin{rmk}\label{hrmk}
If $\XX$ is a linear algebraic stack, with underling algebraic stack
$X=\XX_\cfg$, there exists a locally coherent representable
sheaf $\H$ over $X\times X$, which represents the sheaf over
$X\times X$, whose set of sections over the pair $x,y\in X(S)$ is
the $\O(S)$-module $\Hom_S(x,y)$. The sheaf $\H$ is the {\em universal
sheaf of homomorphisms}.  The subsheaf $\I\subset \H$ representing
isomorphisms is naturally identified with $X$, and the projection to
$X\times X$ with the diagonal. 

Pulling back $\H$ via the diagonal to $X$, we obtain the {\em
  universal sheaf of endomorphisms }$\E\to X$, which represents the sheaf whose
set of sections over  $x\in X(S)$ is the  $\O(S)$-algebra
$\End_S(x)$. Let us emphasize that $\E\to X$ is a representable
morphism of algebraic stacks, which is at the same time a sheaf of
algebras, and a locally coherent
representable sheaf of $\O_X$-modules. 

The linear algebraic stack $\XX$ can be reconstructed from its
underlying algebraic stack $X$, and  the representable sheaf of $\O_{X\times
  X}$-algebras $\H$.  We leave it to the reader to write down axioms
for the pair $(X,\H)$, which assure that $(X,\H)$ comes from a linear
algebraic stack. 
\end{rmk}

\subsubsection{Examples}
  
\begin{ex}\label{coh}\comment{do we need any assumptions on our
    schemes $S\in\SSS$?  It seems a bit suspicious to talk about
    coherent sheaves on $X\times S$ for completely arbitrary schemes
    $S$, for example, don't we want that flat coherent sheaves are
    locally free? I looked up in the stack project: a coherent sheaf
    which is flat is locally free on {\em any }scheme}
Let $X$ be a projective $R$-scheme. The linear stack $\mf{Coh}_X$ has
as objects lying over the $R$-scheme $S$, the coherent sheaves on
$X\times S$, which are flat over $S$. For a morphism of $R$-schemes
$f:S'\to S$, and  $\F'\in\mf{Coh}_X(S')$, and
$\F\in\mf{Coh}_X(S)$, we set $\Hom_f(\F',\F)=\Hom_{\O_{X\times
    S'}}(\F',f^\ast\F)$.  A morphism $\F'
\to \F$ in $\mf{Coh}_X$ over $f$ in $\SSS$ is cartesian, if it
induces an isomorphism   $\F'\cong f^\ast\F$. 

The linear stack $\mf{Coh}_X$ is algebraic. 

To see this, suppose $\F$ and $\G$ are
coherent sheaves on $X\times S$, flat over $S$.  The fact that
$\uHom(\F,\G)$ is a locally coherent representable sheaf of
$\O_X$-modules, 
follows from the fact that 
there exists a coherent sheaf $\N$ on the big \'etale site of $S$,
such that $\uHom(\F,\G)=\N^\vee$  (see \cite[EGA III 7.7.8,
  7.7.9]{EGA}). In fact, for a morphism of schemes $T\to S$, we have
$\uHom_T(\F,\G)={\pi_T}_\ast\shom(\F_{X\times T},\G_{X\times
  T})$. The fact that pushforward does not commute with arbitrary
pullbacks means that $\uHom(\F,\G)$ is not in general cartesian, and
hence not in general coherent. On the other hand, by [ibid.], we 
have 
${\pi_T}_\ast\shom(\F_{X\times T},\G_{X\times T})=(\N_T)^\vee$, which
proves that, indeed, $\uHom(\F,\G)=\N^\vee$. 

The fact that $(\mf{Coh}_X)_\cfg$ is algebraic and locally of finite
type is proved in \cite[4.6.2.1.]{LMB}.
\end{ex}

\begin{ex}\label{vect}
As a special case of the previous example, consider the case
$X=\Spec R$. Then the linear algebraic stack $\mf{Coh}_{\Spec R}$ is
the {\em linear stack of vector bundles}, notation $\mf{Vect}$.  The
underlying algebraic 
stack $\mf{Vect}_\cfg$ 
is the disjoint union $\coprod_{n\geq 0}BGL_n$. The sheaf $\H$ over 
$$\coprod_{n\geq 0}BGL_n\times\coprod_{n\geq 0}BGL_n=\coprod_{n,m\geq0}
B(GL_n\times GL_m)$$ is given by the natural representation $M(m\times
n)$ of $GL_n\times GL_m$ over the component $B(GL_n\times GL_m)$. 
\end{ex} 

\begin{ex}\label{repQ}
A  generalization of $\mf{Vect}$ in a different direction is given by 
quiver representations. 

Let $Q$ be a quiver. The stack of representations of $Q$, notation
$\mf{Rep}_Q$, has as $\mf{Rep}_Q(S)$ the set of diagrams $(\F)$ in the shape
of $Q$ of  locally free finite rank $\O_S$-modules. For a morphism
$f:S'\to S$ of $R$-schemes we have that
$\Hom_f(\F',\F)$ is the $\O(S')$-module of homomorphisms $\F'\to f^\ast
\F$ of diagrams of locally free $\O_{S'}$-modules.
\end{ex}

\begin{ex}\label{toy}
As a toy example, let $A$ be a smooth  $R$-algebra scheme of finite type,
with smooth group scheme of units $A^\times$, also of finite type,
such that the underlying $R$-module is locally free. Then we
define the linear stack of $A^\times$-torsors to have as objects over
the $R$-scheme $S$ the right $A^\times$-torsors over $S$, and for
$f:S'\to S$ and $A^\times$-torsors $P'$ over $S'$ and $P$ over $S$, we
set $\Hom_f(P',P)= \Hom_{S'}(P',f^\ast
P)=P'\times_{A^\times}A\times_{A^\times}f^\ast P$. In this example,
the underlying algebraic stack is $BA^\times$ and we
have $\H=A^\times\backslash A/A^\times$.

The case $A=0$ is not excluded. The associated linear stack is
$\id:\SSS\to\SSS$.  All $\Hom_f(x,y)$ are singletons, endowed with
their unique module structure.  This  stack is represented by $\Spec
R$. It can also be thought of as the stack of zero-dimensional vector
bundles. 
\end{ex}

\subsubsection{Substacks}

Let $\XX$ be a linear algebraic stack with underlying algebraic
stack $X=\XX_\cfg$. If $Y\subset X$ is a locally closed algebraic
substack, there is a canonical linear algebraic stack $\YY$, with
underlying algebraic stack $\YY_\cfg= Y$.  In fact,
we can define $\YY$ to be the full subcategory of $\XX$ consisting
of objects which are in  $X$. 

In this situation, we call $\YY\to\XX$ a locally closed linear
substack of $\XX$.

\subsubsection{Fibered products}

Let $F:\XX\to\ZZ$ and $G:\YY\to \ZZ$ be cartesian morphisms of
$\O$-linear fibered categories. We define a new $\O$-linear fibered
category  $\WW$ as follows:
objects of $\WW$ over the object $T$ of $\SSS$ are triples $(x,\alpha,y)$, where
$x$ is an $\XX$-object over $T$, $y$ is a $\YY$-object over $T$, and
$\alpha$ is an isomorphism $\alpha:F(x)\to G(y)$, over $T$.  A
morphism from $(x',\alpha',y')$ to $(x,\alpha,y)$ over $T'\to T$ is a
pair of morphisms $f:x'\to x$ over $T'\to T$ and $g:y'\to y$ over
$T'\to T$, such that $\alpha\circ F(f)=G(g)\circ\alpha'$.

In other words, we can write the set of morphisms from
$(x',\alpha',y')$ to $(x,\alpha,y)$ over $\phi:T'\to T$ as the fibered
product
$$\Hom_\phi(x',x)\times_{\Hom_\phi(F(x'),G(y))}\Hom_\phi(y',y)\,,$$
and as each of the sets in this fibered product is an $\O(T')$-module,
and the maps are linear, this fibered product is also an $\O(T')$-module. 
We leave it to the reader to verify that composition is bilinear. 

Let us verify that $\WW$ is a fibered category. Suppose that
$(x,\alpha,y)$ is a triple over $T$, and $\phi:T'\to T$ a morphism in
$\SSS$. We construct a triple $(x',\alpha',y')$ over $T'$ by taking as
$x'$ a pullback of $x$ via $\phi$, and for $y'$ a pullback of $y$ via
$\phi$. Then, as $G$ is cartesian, $G(y')$ is a pullback of  $G(y)$
via $\phi$.  Hence there exists a unique morphism $\alpha':F(x')\to
G(y')$ covering $T'$, such that $\alpha\circ F(x'\to x)=G(y'\to
y)\circ \alpha'$. Then $\alpha'$ is cartesian, because cartesian
morphisms satisfy the necessary two out of three property.  Then
$\alpha'$ is invertible, because cartesian morphisms covering an
identity are invertible. The triple $(x',\alpha',y')$ comes with a
given morphism to $(x,\alpha,y)$ which covers $\phi$.  It is easily
verified that this morphism is cartesian. 

Therefore, $\WW$ is an $\O$-linear fibered category.  By construction, the two
projections $\WW\to\XX$ and $\WW\to\YY$ are cartesian.  We call $\WW$
the {\em fibered product }of $\XX$ and $\YY$ over $\ZZ$.  

Suppose $\XX$, $\YY$ and $\ZZ$ are algebraic, with underlying
algebraic stacks $X$, $Y$ and $Z$, respectively. For triples
$(x',\alpha',y')$ and $(x,\alpha,y)$ over $S$, the presheaf
$\uHom\big((x',\alpha',y'),(x,\alpha,y)\big)$ is equal to the
fibered product
$$\uHom(x',x)\times_{\uHom(Fx',Gy)}\uHom(y',y)\,,$$
and is therefore a locally coherent representable sheaf of $\O_S$-modules. We
see that $\WW$ is again a linear algebraic stack.  Moreover, the
underlying algebraic stack of $W$ is the fibered product
$X\times_Z Y$.

\subsubsection{Lack of locality} 

\begin{rmk}\label{lack}
Suppose $\XX$ and $\YY$ are linear
algebraic stacks, with underlying algebraic stacks $X$ and $Y$. We can
construct a disjoint union linear algebraic stack $\XX\amalg\YY$ whose
underlying algebraic stack is $X\amalg Y$, by declaring all
homomorphisms between objects of $\XX$ and objects of $\YY$ to be
zero. This concept of disjoint union is not useful for our purposes.
For the linear algebraic stacks we are interested in, the
underlying algebraic stack  often decomposes into a disjoint
union, even though the linear algebraic stack does not. 
An example is given by the linear stack of vector bundles $\mf{Vect}$,
Example~\ref{vect}.  

Thus linear algebraic stacks exhibit less local behaviour than
algebraic stacks, and are therefore less geometrical.  This is one of
the reasons we prefer to work with {\em algebroids}, rather than
linear algebraic stacks. 
\end{rmk}

\subsubsection{Special linear stacks}

For a linear algebraic stack $\MM$, every fiber category $\MM(S)$ is
an $R$-linear category. By putting special requirements on these
linear categories, we get stronger notions of linear algebraic stack. 

For a linear algebraic stack $\MM$, we denote the universal sheaf of
endomorphisms by $\AA\to \MM$.

\begin{defn}\label{nul}
A linear algebraic stack $\MM$ has a {\bf zero object}, if  the
$R$-linear category $\MM(R)$ admits a zero object. 
\end{defn}

If $\MM$ admits a zero object, then for every 
$R$-scheme $S$, the $R$-linear category $\MM(S)$ admits a zero object,
namely the pullback of the zero object in $\MM(\spec R)$ via the
unique morphism $S\to\spec R$. 

A zero object for $\MM$ defines a section $\spec
R\stackrel{0}{\longrightarrow} \MM$,
which is an isomorphism onto the closed substack of $\MM$
defined by the condition  $1=0$ inside $\AA$. 

If $\MM$ admits a zero object, we denote the complement of the zero
object in $\MM$ by $\MM_\ast$. It is a linear open substack of $\MM$.

\begin{defn}
The linear algebraic stack $\MM$ {\bf admits direct sums}, if for ever
$R$-scheme $S$, the $R$-linear category $\MM(S)$ admits all (finite)
direct sums. 
\end{defn}

The pullback functor $\MM(S)\to \MM(S')$, for a morphism of
$R$-schemes $S'\to S$ commutes with direct sums. 
Hence, if $\MM$ admits direct sums, there is a canonical morphism of
linear stacks
\begin{align*}
\MM\times\MM&\longrightarrow\MM\\
(x,y)&\longmapsto x\oplus y\,.
\end{align*}
See also Remark~\ref{karoubi}.


\subsection{Finite type algebras}

\begin{defn}\label{ftalg}
Let  $X$ be an algebraic stack.  By an {\bf algebra }over $X$, we mean
a sheaf of $\O_X$-algebras over $X$. If the algebra $A$ over the 
algebraic stack $X$ is an algebraic stack itself, i.e., if the
structure morphism $A\to X$ is a representable morphism of stacks,
then we say that $A$ is {\bf representable}. If $A$ is represented by a finite type
affine stack of the form $\spec_X\Sym_{\O_X}\F$, for a coherent sheaf
$\F$ over $X$, we call $A$ a {\bf   finite type algebra }over $X$.

For an automorphism $\phi$ of an algebra $A$, we denote the subalgebra
of fixed sections by $A^\phi$.  For  a section $a$ of $A$ we denote by
$A^a$ the subalgebra of sections commuting with $a$. 
\end{defn}

So the sheaf of $\O_X$-modules underlying a finite type algebra is
locally coherent representable. 

If $\XX$ is a linear algebraic stack with underlying algebraic
stack $X$, then the universal sheaf of endomorphisms  $\E\to X$ is a finite
type algebra.\comment{I'm a bit torn about what the universal stack of
  endomorphisms is: is it a linear algebraic stack? or an algebraic
  stack? But in the end, algebroids are more important, so WHO CARES?}

Note that finite type algebras need not have a coherent underlying
sheaf of $\O_X$-modules.  For example, let $X=\aaa^1$, with coordinate
$t$, and let $A\to X$ be the centralizer in $(M_{2\times 2})_X$ of the
matrix $\bigl( \begin{smallmatrix}
  t&0\\ 0&0
\end{smallmatrix} \bigr)$. In this case the big sheaf underlying $A$
is not cartesian.

\begin{rmk}
Dually, a finite type algebra $A$ over $X$ corresponds to a coherent
sheaf $M$ over $X$, which is endowed with a coalgebra structure
$(\Delta,\epsilon)$, where $\Delta:M\to M\otimes_{\O_X}M$ is an
associative comultiplication with counit $\epsilon:M\to \O_X$. For any
$X$-scheme $U$, we have $A_U=M_U^\vee$. 
\end{rmk}

\subsubsection{Inertia representation}

Whenever $A\to X$ is an algebra over the algebraic stack $X$, we have
a tautological  morphism of sheaves of groups over
$X$
\begin{equation}\label{rep}
I_X\longrightarrow \uAut(A)\,.
\end{equation}
Here $I_X$ is the inertia stack of $X$, i.e., the stack of pairs
$(x,\phi)$, where $x$ is an object of $X$, and $\phi$ an automorphism
of $x$, and $\uAut(A)$ is the sheaf of automorphisms of the sheaf of
algebras $A$ over $X$. 
To construct (\ref{rep}), consider the stack of sheaves of algebras $\mf{Alg}$
over $\SSS$, which has as objects over the scheme $S$, the sheaves of
$\O_S$-algebras on the usual (small) \'etale site of $S$.  A morphism
from the sheaf of $\O_{S'}$-algebras $A'$ over $S'$, covering the morphism of
schemes $f:S'\to S$, to the sheaf of $\O_S$-algebras $A$ over $S$, is, by
definition, an isomorphism of sheaves of $\O_{S'}$-algebras $A'\to
f^\ast A$.  The sheaf of algebras $A\to X$ gives rise to a morphism of
$\SSS$-stacks $a:X\to\mf{Alg}$. We get an induced morphism on inertia
stacks $I_X\to I_{\mf{Alg}}$, and notice that $a^\ast
I_{\mf{Alg}}=\uAut(A)$. 

With this definition, an automorphism $\phi$ of the object $x$ of the
stack $X$ is mapped to the {\em inverse }of the restriction morphism
$\phi^\ast:A(x)\to A(x)$.

\begin{lem}\label{inertiadescent}
Suppose $X$ is a gerbe over the algebraic space $Y$, and  $A\to X$
is an algebra.  Then there exists a sheaf of $\O_Y$-algebras $B$, and
an isomorphism $A\cong B|_X$ if and only if the inertia representation
$I_X\to \uAut_X(A)$ is trivial. 

Similarly, if $X$ is a connected gerbe over the Deligne-Mumford stack
$Y$, then an algebra $A$ over $X$ descends to $B$ over  $Y$, if and only if the
inertia representation restricts to a trivial homomorphism
$I_X^\circ\to\uAut_X(A)$. 

In either case,  $A$ is representable or of finite type if
and only if $B$ is. \nolinebreak $\Box$
\end{lem}

We can pull back the sheaf of algebras $A$ over $X$, via the structure
morphism $I_X\to X$, to obtain the sheaf of algebras $A|_{I_X}$. This
sheaf of algebras is endowed with a tautological automorphism, induced
from (\ref{rep}). The algebra of invariants for this automorphism we
shall denote by $A_{I_X}^\fix$. 

The following statement is somewhat tautological, and holds more
generally than for algebras.

\begin{prop}\label{tut}\comment{should be about finite type algebras,
    or algebra bundles?}
Suppose that $A$ is a representable algebra over the algebraic stack
$X$. Then the inertia stack of $A$ is naturally identified with
$A_{I_X}^\fix$. In particular, $I_A$ is a representable algebra over $I_X$. 
\end{prop}
\begin{proof}
We have a commutative diagram of algebraic stacks
$$\xymatrix{I_A\rto\dto& A\dto\\
I_X\rto & X}$$
which identifies $I_A$ with a substack of $A|_{I_X}$. The algebra
$A|_{I_X}$ is the stack of triples $(x,\phi,a)$, where $x$ is an
object of $X$, $\phi$ is an automorphism of $x$, and $a\in A(x)$ is an object
of $A$ lying over $x$. Such a triple is in $I_A$, if and only if
$\phi\in \Aut(x)$ is in the subgroup $\Aut(a)\subset\Aut(x)$. This is
equivalent to $\phi$ fixing $a$ under the action of $\Aut(x)$ on
$A(x)$. This is the claim.
\end{proof}

In fact, the fibre of $I_A$ over the object $x$ of $X$ is equal to 
$$I_A(x)=\{(\phi,a)\in \Aut(x)\times A(x)\mid
\phi^\ast(a)=a\}\,.$$
The fibre of $I_A(x)$ over $\phi\in \Aut(x)$ is the subalgebra $A(x)^\phi\subset
A(x)$, and the fibre of $I_A(x)$ over $a\in A(x)$ is the subgroup 
$\Stab_{\Aut(x)}(a)\subset \Aut(x)$.

\subsubsection{Algebra bundles}

\begin{defn}
We call a finite type algebra $A\to X$ an {\bf algebra bundle}, if the
underlying $\O_X$-module is locally free (necessarily of finite rank).
\end{defn}

When studying finite type algebras over finite type
stacks $X$, we may, after passing to a locally closed stratification
of $X$, assume that the finite type algebra is an algebra bundle.

\begin{defn}\label{rankstrat}
Let $A\to X$ be a finite type algebra over the algebraic stack
$X$. The stratification $X'\subset X$ of Proposition~\ref{bdlstr} is
characterized by the property that an $X$-scheme $S$ factors through
$X'$ if and only if $A|_S$ is an algebra bundle.  The pullback
$A'=A|_{X'}$ is an algebra bundle, and the induced stratification
$A'\to A$ of $A$ is called the {\bf rank stratification }of $A$.
\end{defn}

\begin{rmk}
By considering the representation of $A$ on itself by left
multiplication, we see that every algebra bundle is a sheaf of
subalgebras of the algebra $\End(V)$ of endomorphisms of a vector
bundle $V$ over the stack $X$. 
\end{rmk}

\subsubsection{Central idempotents}

\begin{lem}\comment{what we had before is simply wrong.  I don't think
    the center of a finite type algebra is finite type.}
The centre of an algebra bundle is a finite type algebra.
\end{lem}
\begin{proof}
The  centre of $A$ is the kernel of the $\O_X$-linear homomorphism of
vector bundles 
$A\to \uEnd_{\O_X}(A)$, given by $a\mapsto [a,\cdot\,]$. As such, it is a
locally coherent representable  sheaf.
\end{proof}

Thus, if $A\to X$ is a finite type algebra over a finite type stack,
after passing to a locally closed stratification of $X$, we may assume
that $A$ is an algebra bundle, whose centre is an algebra bundle.

If $A\to X$ is a commutative algebra bundle, then $\pi:Y=\spec_X A\to X$ is a
finite flat representable morphism, and $A=\pi_\ast\O_Y$. In fact, the
category  of commutative algebra bundles over $X$ is equivalent to the
category of finite flat representable stacks over $X$.

For a commutative finite type algebra  $A\to X$, we denote the stack of
idempotents in $A$ by $E(A)$. 

\begin{lem}\label{Estrat}
Suppose $A$ is a commutative algebra bundle over the algebraic stack
$X$. Then the structure morphism $E(A)\to X$ is affine, of finite
type, and \'etale. In particular, there exists a non-empty open
substack $U\subset X$, such that $E(A)|_U=E(A|_U)$ is
finite \'etale over $U$. 
\end{lem}
\begin{proof}
We reduce to the case where $X$ is a scheme, and then quote
Lemme~(18.5.3) of EGA~IV \cite{EGA}.
\end{proof}

By this lemma, when studying the centre of finite type algebras
over the finite type stack $X$, we may, after passing to a 
stratification
of $X$, assume that the stack of central idempotents is finite \'etale
over $X$. 

\subsubsection{Primitive idempotents}

Recall that a non-zero  idempotent $e$ is called {\em primitive}, if
whenever $e=e_1+e_2$,
for orthogonal idempotents $e_1$, $e_2$,  then necessarily 
$e_1=0$ or $e_2=0$.  

In a finite-dimensional  commutative algebra over a field, the
following is true:
\begin{items}
\item every idempotent is in a unique way (up to order of the
  summands) a sum of orthogonal primitive idempotents, this is the {\em primitive
    decomposition}, 
\item orthogonal idempotents have disjoint primitive decompositions,
\item distinct primitive idempotents are orthogonal to each other,
\item the primitive idempotents add up to 1.
\end{items}
Thus, the idempotents are in bijection with the subsets of the
(finite) set of primitive idempotents. 

Let $A\to X$ be a finite type algebra.

\begin{defn}
An idempotent local section $e:U\to A$ of  
$A\to X$ is  {\bf primitive}, if it gives rise to a primitive
idempotent in the fiber of $A$ over every  geometric point of $U$.
\end{defn}

Suppose $A=\pi_\ast\O_Y$ is a commutative algebra bundle, and $e$ an
idempotent global section.  Let $Y_1\subset Y$ be the open and closed
substack defined by the equation $e=1$. Then $e$ is primitive, if and
only if the
geometric fibres of $Y_1\to X$ are connected. As the function counting
the number of connected components of the fibres is lower
semi-continuous, the subset of $X$ where $e$ is primitive, is closed.
In general, this subset is not open.  Therefore, when studying
primitive idempotents, we assume that $E(A)\to X$ is finite \'etale.

\begin{lem}
Let $A\to X$ be a commutative algebra bundle, with finite \'etale
stack of idempotents $E(A)\to X$. There is an open and closed substack
$PE(A)\subset E(A)$, such that an idempotent local section factors through $PE(A)$
if and only if it is primitive.
\end{lem}
\begin{proof}
We may assume that $E(A)\to X$ is constant.  Then the multiplication
operation and the partially
defined addition operation on $E(A)$ are also constant. The claim
follows.
\end{proof}

\begin{defn}
Let $A\to X$ be an algebra bundle, with centre $Z\to X$. Let
$ZE(A)$ be the stack of idempotents in $Z$, in other words the stack
of central idempotents in $A$.  Assume that $ZE(A)\to X$ is finite
\'etale.
The substack of primitive idempotents in $ZE(A)$ is denoted by
$PZE(A)$, and called the stack of {\bf primitive central idempotents
}of $A$. It is finite \'etale over $X$.  The degree of $PZE(A)\to X$
is called the {\bf central rank }of $A$. 

If $X$ is connected,\comment{I got rid of normal, because we only seem
  to need connected for Galois theory to work.}  
the number of connected components of $PZE(A)$ is the {\bf split
  central rank }of $A$. More precisely, the partition of the central
rank given by the degrees of the connected components of $PZE(A)$ is
called the {\bf central type }of $A$.  (So the split central rank is
the length of the type.)
\end{defn}

\begin{rmk}
Let $X$ be connected, and let $A\to X$ be a commutative
finite type algebra, with finite \'etale stack of idempotents $E(A)\to
X$. Then there is a one-to-one correspondence between the connected
components of $PE(A)$ and the primitive idempotents in the algebra of
global sections $\Gamma(X,A)$. 
\end{rmk}

\subsubsection{The degree stratification}

Let $k$ be a field and $A$ a finite-dimensional $k$-algebra. The {\em
  rank }$r$ of $A$ is the dimension of $A$ as a $k$-vector space.  For an
element $a\in A$, we define its {\bf degree }to be the dimension of
the commutative subalgebra $k[a]\subset A$. It is equal to the degree
of the minimal polynomial of $a$, i.e., the monic generator of the
kernel of the algebra map $k[x]\to A$, defined by $x\mapsto a$. 

Now let $A$ be an algebra bundle of rank $r$ over the algebraic stack
$X$, and $a\in A(S)$ a local section of $A$ over an $X$-scheme
$S$.

\begin{defn}
If the cokernel (as a homomorphism of $\O_S$-modules) of the
morphism of $\O_S$-algebras $\O_S[x]\to A_S$, defined by $x\mapsto a$ is
flat over $\O_S$, we say that $a$ is {\bf strict}, and we call the
rank of the image of $\O_S[x]\to A$ the {\bf degree } of $a$. 
\end{defn}

If $f(x)\in\O_S[x]$ is the characteristic polynomial of $a$, the
morphism $\O_S[x]\to A$ factors through $\O_S[x]/(f)$, by the theorem
of Caley-Hamilton. Hence the cokernel of $\O_S[x]\to A$ is actually a
cokernel of a homomorphism of vector bundles, and hence coherent. The
condition that this cokernel be flat is equivalent to it being locally
free.  It implies that forming the image of $\O_S[x]\to A$ commutes
with base change, and that this image, denoted $\O_S[a]$, is also
locally free.

\begin{prop}
For every $n=1,\ldots,r$, there exists a locally closed substack
$A_n\subset A$ with the property that a
local section $a\in A(S)$ factors through $A_n(S)$ if and only if
$a$ is strict of degree $n$. The $A_n$ are pairwise disjoint and
their disjoint union 
$$A^\strat=\coprod_{n=1}^r A_n$$
maps surjectively to $A$. The section $a\in A(S)$ factors through
$A^\strat \to A$ if and only if it is strict.
\end{prop}
\begin{proof}
  Consider the tautological section $\Delta$ of the pullback of $A$ via
the structure map $A\to X$. It gives rise to a morphism of
$\O_A$-algebras $\O_A[x]\to A_A$. Then $A^\strat$ is given by the
flattening stratification of its cokernel, and $A_n\subset A^\strat$
is the component where the cokernel has rank $r-n$.
\end{proof}

We call the stratification $A^\strat\to A$ the {\bf degree
  stratification }of $A$. 

\subsubsection{Semi-simple elements}

Let $k$ be an algebraically closed field, and $A$ a finite-dimensional
$k$-algebra. Recall that an element $a\in A$ is {\em semi-simple}, if
the following equivalent conditions are satisfied:
\begin{items}
\item the map $A\to A$ given by left multiplication by $a$ is
  diagonalizable,
\item the minimal polynomial $f\in k[x]$ of $a$ is separable, i.e.,
  satisfies $(f,f')=1$,
\item the commutative subalgebra $k[a]\subset A$ is reduced, or,
  equivalently, \'etale over $k$.
\end{items}

\begin{defn}\label{ssdef}
Let $A\to X$ be an algebra bundle. A local section $a\in A(S)$, for an
$X$-scheme $S$, is called {\bf semi-simple}, if it is strict, and for
every geometric point $s\in S$, the element induced by $a$ in $A(s)$
is semi-simple.
\end{defn}

For example,  an idempotent section $e\in A(S)$ is semi-simple over
the open subset of $S$, where $e$ is neither 0 nor 1 (in a commutative
algebra bundle, this subset is also closed).\comment{maybe include a proof?} 

Assuming  $a\in A(S)$ is strict, $a$ is
semi-simple if and only if the geometric fibres of the finite flat
$S$-scheme  $\spec_S\O_S[a]$ are unramified. This
condition is equivalent to $\spec_S\O_S[a]$ being unramified, hence 
\'etale over $S$.

The semi-simple sections of $A$ form a subsheaf $A^\ss\subset A$. 

\begin{prop}\label{sost}
Let $A\to X$ be an algebra bundle over the algebraic stack $X$.  Then
$A^\ss$ is an algebraic stack with a representable  structure
morphism of finite type $A^\ss\to X$.
\end{prop}
\begin{proof}
In fact, $A^\ss\subset
A^\strat$ is the open substack defined by the condition that the
finite flat representable morphism 
$\spec_A \O_A[\Delta]\to A$ is unramified. Thus, we have a
factorization of the monomorphism $A^\ss\to A$ as
$$\xymatrix{
  A^\ss\rrto^-{\text{open immersion}}&&
  A^\strat\rrto^-{\text{stratification}}&&
  A}\,.$$
Thus, $A^\ss\subset A$ is a constructible substack. 
\end{proof}

\subsubsection{The semi-simple centre}

For a commutative finite dimensional algebra over an algebraically
closed field, we have
\begin{items}
\item the primitive idempotents are linearly independent,
\item an element is semi-simple if and only if it is a linear
  combination of primitive idempotents.
\end{items}
We need a version of this statement for algebra bundles. 

\begin{prop}\label{ssc}
Let $A\to X$ be a commutative algebra bundle whose stack of
idempotents is finite \'etale, and let $Y\to X$ be the finite
flat cover corresponding to $A$. There is a canonical finite flat
morphism of $X$-stacks $Y\to PE(A)$. Over every geometric point $x$ of
$X$, this morphism maps each point in the fibre $Y_x$ to the
characteristic function of its connected component in $Y_x$ (which is
a primitive idempotent in $A|_x$). Dually, we obtain a strict
monomorphism of algebra bundles
$$\pi_\ast \O_{PE(A)}\longrightarrow A\,,$$
where $\pi:PE(A)\to X$ is the structure map.

The induced morphism
$$(\pi_\ast\O_{PE(A)})^\strat\longrightarrow A^\strat$$ factors
through the open substack $A^\ss\subset A^\strat$, and induces a
surjective closed immersion of algebraic stacks
$(\pi_\ast\O_{PE(A)})^\strat \to A^\ss$.
\end{prop}
\begin{proof}
Consider the finite \'etale cover of primitive idempotents
$\pi:PE(A)\to X$.  We have a tautological global section $e$ of
$A|_{PE(A)}$, and $a\mapsto ae$ defines a homomorphism of
$\O_{PE(A)}$-modules $\O_{PE(A)}\to A|_{PE(A)}$. Pushing forward with
$\pi$ and composing with the trace map $\pi_\ast( A|_{PE(A)})\to A$
defines the morphism of algebra bundles over $X$
\begin{equation*}\label{ejfjqq}
\pi_\ast \O_{PE(A)}\longrightarrow A\,.
\end{equation*}
It is a strict monomorphism of vector bundles, because it is injective
over every geometric point, by Fact~(i), above. 
Dually, we obtain a morphism of $X$-stacks $Y\to PE(A)$, which is the
morphism described in the statement of the proposition. It is flat,
because $PE(A)$ is \'etale over $X$, and flatness can be checked
\'etale locally.

Passing to the degree stratification commutes with strict
monomorphisms of algebra bundles, so we have a cartesian diagram of
$X$-stacks
$$\xymatrix{
  (\pi_\ast\O_{PE(A)})^\strat\rto\dto & A^\strat\dto\\
  \pi_\ast\O_{PE(A)}\rto & A\rlap{\,,}}$$
which shows that $(\pi_\ast \O_{PE(A)})^\strat\to A^\strat$ is a
closed immersion. The facts that this closed immersion factors through
$A^\ss\subset A^\strat$, and is surjective onto $A^\ss$ can be checked
over the  geometric points of $X$, where it follows from Fact~(ii),
above.
\end{proof}

\subsubsection{Permanence of rank and split rank}

\begin{prop}\label{rankineq}
Let $A\hookrightarrow A'$ be a monomorphism of commutative finite type
algebras with finite \'etale stacks of idempotents over the 
connected stack $X$. Denote the ranks of $A$ and $A'$ by $n$ and 
$n'$, and the split ranks by
$k$ and $k'$, respectively. Then $n\leq n'$ and $k\leq k'$. Moreover,
\begin{items}
\item if $A'$ admits a semi-simple global section which does not factor
  through the fiber $A|_x$ for {\em all }points of $x$, then $n<n'$,
\item if $A'$ admits an idempotent global section, which is not in
  $A$, then $k<k'$. 
\end{items}
\end{prop}
\begin{proof}
The monomorphism $A\hookrightarrow A'$ induces an open and closed
embedding of finite \'etale 
$X$-stacks $E(A)\hookrightarrow E(A')$. Every idempotent $e$ in $A$ can be
decomposed uniquely into a sum of orthogonal primitive idempotents in
$A'$. Let us call this the {\em primitive decomposition }of $e$ in
$A'$.  Consider the correspondence $Q\subset PE(A)\times_X PE(A')$
defined by
\begin{align*}
(e,e')\in Q\quad&\Longleftrightarrow\quad\text{$e'$ partakes in the
    primitive decomposition of $e$  in $A'$}\\
  &\Longleftrightarrow\quad ee'=e'\,.
\end{align*}
One shows that $Q$ is a finite \'etale cover of $X$ locally in the
\'etale topology of $X$, reducing to the case where both $E(A)$ and
$E(A')$ are trivial covers. By properties of the primitive
decomposition, the projection $Q\to PE(A)$ is surjective, and the
projection $Q\to PE(A')$ is injective. Thus we have
$$n=\deg PE(A)\leq \deg Q\leq \deg PE(A')=n'\,.$$
If $n=n'$, then both $Q\to PE(A)$ and $Q\to PE(A')$ are isomorphisms,
showing that $PE(A)=PE(A')$.  By Proposition~\ref{ssc}, this implies
that the morphism $A^\ss\to (A')^\ss$ is surjective. This
proves~(i).

We can repeat the argument for the algebras of global sections
$\Gamma(X,A)\hookrightarrow \Gamma(X,A')$. We deduce that $k\leq k'$,
and if $k=k'$, every primitive idempotent in $\Gamma(X,A)$ remains
primitive in $\Gamma(X,A')$, and every primitive idempotent of
$\Gamma(X,A')$ is in $\Gamma(X,A)$.  We deduce that $\Gamma(X,A)$ and
$\Gamma(X,A')$ have the same idempotents, which proves~(ii).
\end{proof}

\subsubsection{Families of idempotents}

\begin{defn}\label{famid}
For a finite type algebra $A\to X$, we denote by $E_n(A)\to X$ the
stack of $n$-tuples of non-zero idempotents in $A$, which are pairwise
orthogonal, and add up to unity. We call sections of $E_n(A)$ also
{\em complete sets of orthogonal idempotents}. 
\end{defn}

Note that the family members of the sections of $E_n(A)$ need not be
central.

The stack $E_n(A)$ is algebraic, and of finite type over $X$. 

For $n=0$, the stack $E_0(A)$ is empty, unless $A=0$, in which case it
is identified with $X$.  For $n=1$, the stack $E_1(A)$ contains exactly the
unit in $A$ (so is identified with $X$), unless $A=0$, in which case
$E_1(A)$ is empty. 

\subsubsection{Group of units}

Let $A\to X$ be a finite type algebra over the algebraic stack
$X$. The  subsheaf of units $A^\times\subset A$ is defined by 
$$A^\times(x)=\{a\in A(x)\mid\exists b\in A(x)\colon ab=ba=1\}\,,$$
for every object $x$ of $X$. We can see that $A^\times$ is a
(relative) affine group scheme over $X$, by writing it as the fibered
product
$$\xymatrix{
A^\times\rrto\dto_{u\mapsto(u,u^{-1})} && X\dto^{(1,1)}\\
A\times A\rrto^{(a,b)\mapsto (ab,ba)}&& A\times A\rlap{\,.}}$$ 

if $A$ is an algebra bundle, the subsheaf $A^\times\subset A$ is
represented by an open substack, because in this case, a local section
$a$ is invertible if and only if the determinants of left and right
multiplication by $a$ on $A$ do not vanish. 
We conclude that if $A$ is an algebra bundle, $A^\times$ is smooth
over $X$ with geometrically connected fibers. For the general case,
this implies that $A^\times$ is an affine group scheme with smooth and
geometrically connected fibers over $X$. 

A similar argument using the determinant proves that if $A\to B$ is a
strict monomorphism of algebra bundles, we have $A^\times=A\cap
B^\times$. Also, if $A$ is an algebra bundle, any morphism $A\to B$ to
another algebra bundle is determined by its restriction to
$A^\times$. 

\subsection{Algebroids}

\begin{defn}\label{defnalgebroid}
A {\bf algebroid }is a triple $(X,A,\iota)$, where $X$ is an
algebraic stack, $A$ is a finite type algebra  over $X$,
and $\iota:A^\times\to I_X$ is a homomorphism of sheaves of groups over
$X$, which identifies $A^\times$ with an open substack of
$I_X$. Moreover, we require that the diagram of groups over $X$
\begin{equation}\label{algebroidproerpty}
\vcenter{\xymatrix{
A^\times\rto^\iota\drto& I_X\dto\\
&\uAut(A)}}
\end{equation}
commutes.  
Here, the  map $A^\times\to \uAut_X(A)$ associates to a
unit $u$ of $A$ the inner automorphism $x\mapsto uxu^{-1}$.  The
vertical 
map $I_X\to\uAut_X(A)$ is the inertia
representation~(\ref{rep}).
If $\iota$ is an isomorphism, we call $(X,A,\iota)$ a {\bf
  strict algebroid}. 
\end{defn}

We will usually  abbreviate the triple $(X,A,\iota)$ to $X$, and write
$A_X$ for $A$, if we need to specify the algebra. We call the
commutativity of (\ref{algebroidproerpty}) the {\bf algebroid
  property}. 

For an explanation of the terminology, see Remark~\ref{explainterm}.

The condition that  $\iota$ is an open immersion implies that
$A^\times$ represents the subsheaf
$I^\circ_X\subset I_X$ of connected components of the
identity. (Over a field, $A^\times$ is connected, and 
$A^\times\to I_X$ being an open immersion implies that $I_X/A^\times$ is
\'etale. These conditions characterize the connected component over a field.)
In particular, if $X$ is a connected gerbe over a Deligne-Mumford
stack $S$, then the relative inertia $I_{X/S}$ is necessarily equal to
$I_X^\circ$, and hence $A^\times$ is identified with $I_{X/S}$.

If $A$ is an algebra bundle, $A^\times$, and therefore also
$I_X^\circ$, is an affine smooth
(relative) group scheme over $X$. Hence $X$ admits a {\em coarse
Deligne-Mumford stack}  $X\to \ol X$, which  is uniquely determined by
being a Deligne-Mumford stack, and $X\to \ol X$ being a connected
gerbe.
Moreover,  $A^\times$ is then identified with the relative inertia
group $I_{X/\ol X}\subset I_X$, and we have a short exact sequence of
relative group schemes
$$\xymatrix{
1\rto & A^\times\rto^\iota & I_X\rto & I_{\ol X}|_X\rto & 1}$$
over $X$. 
In case $(X,A)$ is a strict  algebroid, $I^\circ_X=I_X$, and $\ol X$ is an
algebraic space, in fact the {\em coarse moduli space }of $X$.

In many cases, the algebroid property is automatic:

\begin{prop}\label{algebroidproperty}
Consider a triple $(X,A,\iota)$, where $X$ is an
algebraic stack, $A$ is an algebra bundle  over $X$,
and $\iota:A^\times\to I_X$ is a homomorphism of sheaves of groups over
$X$, which identifies $A^\times$ with an open substack of
$I_X$. Then (\ref{algebroidproerpty}) commutes, so that $(X,A,\iota)$
is an algebroid. 
\end{prop}
\begin{proof}
 As
$\iota:A^\times\to I_X$ is defined over $X$, the homomorphism $\iota$
is equivariant with respect to the inertia action.
The proof now combines the facts that $\iota$ is a monomorphism, that
the inertia action on $I_X$ is the inner action, and 
the fact that a morphism of algebra bundles is determined by its
restriction to units.

In fact, let $a\in A^\times$.  To show that
${}^a(\cdot)=\iota(a)\cdot$, as automorphism of $A$, it suffices to
show that for all $b\in A^\times$, we have ${}^ab=\iota(a)\cdot b$.
We can check this after applying $\iota$, so it suffices that
$\iota({}^ab)=\iota(\iota(a)\cdot b)$, or
${}^{\iota(a)}\iota(b)=\iota(a)\cdot\iota(b)$, which is true.
\end{proof}

\begin{ex}{\bf (algebroid underlying a linear stack)} \label{linex}
Let $\XX$ be a linear
algebraic stack with underlying algebraic stack 
$X$, and let $A\to X$ be the universal sheaf of endomorphisms of 
Remark~\ref{hrmk}. Then automorphisms are invertible endomorphisms, so
we use for $\iota$ the tautological identification $A^\times=I_X$.

The inertia representation being the inverse of the pullback
action, it is, indeed, given by (left) inner automorphisms. 

We call $(X,A)$ the {\em algebroid underlying }the linear algebraic
stack $\XX$. It is a strict algebroid. 
\end{ex}

\begin{ex}\label{vectalg}
Consider the linear stack of vector bundles $\mf{Vect}$,
Example~\ref{vect}. The underlying algebroid consists of the disjoint
union of the quotient stacks $\GL_n\backslash M_{n\times n}$, given by
the adjoint representations, for $n\geq0$. Thus, in
passing from the linear stack to the underlying algebroid, we discard
all $M_{m\times n}$, for $m\not=n$, and for $m=n$, restrict the
left-right bi-action of $\GL_n$ on $M_{n\times n}$ to the (left only) adjoint
action. 
Thus we remove exactly the information which we consider non-local,
see Remark~\ref{lack}.
\end{ex}

\begin{ex}{\bf (classifying algebroid)}\label{trivialgerbe}
Let $A\to X$ be an algebra bundle over a
Deligne-Mumford stack $X$. Let $A^\times$ act on $A$ from the
left by inner automorphisms. Then $A^\times\backslash A$ is an algebra
bundle over the relative classifying stack $Y=B_XA^\times$. We have an exact
sequence of group schemes over $Y$
$$\xymatrix{
1\rto & I_{Y/X}\rto & I_Y\rto & I_X|_Y\rto & 1\,,}$$
where $I_{Y/X}=A^\times\backslash A^\times$. As $I_X\to X$ is unramified,
$I_{Y/X}\to I_Y$ is an open embedding, and so $(Y,A^\times\backslash
A)$ is an algebroid. If $X$ is a space, $(Y,A^\times\backslash A)$ is
a strict algebroid.
\end{ex}
 
\begin{rmk}{\bf (algebroids which are trivial gerbes are classifying
    algebroids) }\label{gerberemark}
Let $(X,A)$ be an algebroid, such that $A$ is an algebra bundle over
$X$, and therefore  $X$ is a connected gerbe over the Deligne-Mumford
stack $S$,  with $A^\times=I_{X/S}$.
Suppose the gerbe $X\to S$ admits a section $x:S\to X$.  Via
$x$, we pull back $A$ to a bundle of algebras $C$ over $S$. We 
claim that $(X,A)$ is canonically isomorphic to $(B
C^\times,C^\times\backslash C)$.

In fact, because $X$ is a gerbe over $S$, the section $x:S\to X$ is a
universal principal $x^\ast I_{X/S}$-bundle. 
The pullback diagram
$$\xymatrix{ C\dto\rto & S\dto^x\\ A\rto & X\rlap{}}$$ shows that $C$ is
an $x^\ast I_{X/S}$-bundle over $A$. Hence, $A= C/x^\ast I_{X/S}$. Via the
isomorphism $\iota:C^\times\to x^\ast I_{X/S}$, the action of $x^\ast I_{X/S}$ on
$C$ is identified with the action by left inner automorphisms.  This
follows from the algebroid property of $(X,A)$, and proves the claim.
\end{rmk}

\begin{rmk}\label{explainterm}\comment{think again}
If $X\to\Spec R$ is a gerbe, any strict algebroid over $X$ can be promoted to
a linear algebraic stack, whose underlying algebraic stack is
$X$.

More generally, there exists a notion of {\em relative }linear
algebraic stack, where the base $R$ is replaced by an arbitrary scheme
(or algebraic space). Then every strict algebroid $(X,A)$ where $X$ is a
gerbe over a space $S$ becomes naturally a linear algebraic stack {\em
  over }$S$. These types of linear algebraic stacks occur naturally in
the theory of deformation quantization, where they were introduced by
Kontsevich under the name of `stack of algebroids', see
\cite{DefQuant}. In [loc.\ cit.], one can also find a description
of these stacks of algebroids in terms of cocycles (compare also
\cite{AgnPol}).

Thus our notion of algebroid is a natural generalization of
Kontsevich's notion of stack of algebroids.  This justifies our
terminology.

Loosely speaking, algebroids are stacks which are linear over their
coarse Deligne-Mumford stack.
\end{rmk}

\begin{ex}[schemes as algebroids]
Every Deligne-Mumford stack $Z$ is an algebroid via the definition $A_Z=0_Z$. There
is no natural way to enhance the algebroid $(Z,0_Z)$ to a linear
algebraic stack, unless $Z=\spec R$ is the final scheme.  This
exhibits one way in which algebroids are more 
flexible than linear algebraic stacks.
\end{ex}

\begin{ex}[algebroids over quotient stacks]
Let $X=G\backslash Y$ be a quotient stack. A finite type algebra $A$ over $X$ is
given by a finite type algebra $B$ over $Y$, together with a lift of the
$G$-action on $X$ to an action on $B$ by algebra automorphisms. The inertia
stack of $X$ is naturally identified with
$G\backslash\Stab_GY$.
Thus, $(X,A)$ becomes an
algebroid, if we specify a $G$-equivariant open embedding of $Y$-group
schemes $\iota:B^\times\to\Stab_GY$.  If $B$ is not an algebra bundle, we
also need to require that ${}^ub=\iota(u)\cdot b$, for all $u\in
B^\times$, $b\in B$. 
\end{ex}

\subsubsection{Morphisms of algebroids}

\begin{defn}\label{define}
We call a morphism of algebraic stacks $f:X\to Y$ {\bf inert}, if the
diagram of  stacks
$$\xymatrix{
I_X^\circ\rto\dto & I_Y^\circ\dto\\
X\rto & Y}$$
is cartesian. 
If $I_X\to I_Y|_X$ is an isomorphism,  we call $f$ {\bf strictly inert}.
\end{defn}

The connected component $I_Y^\circ$ is not necessarily an algebraic
stack, but if it is, then so is $I_X^\circ$, if $X\to Y$ is inert. We
will only apply this concept when $Y$ is an algebroid, so that
$I_Y^\circ$ is representable over $Y$. 

The basic facts about inert morphisms are:
\begin{items}
\item Every inert morphism of algebraic stacks is Deligne-Mumford
  representable, because its relative inertia group scheme is
  unramified.
\item Every base change of an inert morphism of algebraic stacks is
inert.
\item  Every monomorphism of algebraic stacks (in particular every locally
  closed immersion and every stratification) is (strictly) inert. 
\item Being inert is
local in the \'etale topology of the target.
\item Every morphism of Deligne-Mumford stacks is inert. 
\end{items}

Base changes of morphisms of Deligne-Mumford stacks are, in fact,  the only
inert morphisms, at least up to stratifications:

\begin{prop}\label{inertprop}
Suppose $X\to Y$ is an inert morphism of algebraic stacks, and suppose
$I_Y^\circ$ (and hence also $I_X^\circ$) is smooth and representable
over $X$, so that we have
coarse Deligne-Mumford stacks  $\ol
X$, $\ol Y$, and an induced morphism  $\ol X\to \ol
Y$. Then the diagram 
$$\xymatrix{
X\rto \dto & Y\dto\\
\ol X\rto & \ol Y}$$
is cartesian.
\end{prop}
\begin{proof}
To prove that the morphism of gerbes $X\to Y|_{\ol X}$ over $\ol X$ is
an isomorphism, we may pass to an \'etale cover  $\ol X'$ of $X$ and assume
that the gerbe $X$ is trivial.  Then  $X=B_{\ol X}G$,
where $G$ is the pullback of $I_X^\circ$ to $\ol X$ via a trivializing
section. Moreover, $Y|_{\ol X}=B_{\ol X} H$, where $H$ is the pullback
of $I_Y^\circ|_X$ to $\ol X$ via the same trivializing section.  Since
$I_X^\circ\to I_Y^\circ|_X$ is an isomorphism, so is $G\to H$, and hence $X\to
Y|_{\ol X}$. 
\end{proof}

\begin{defn}
A {\bf morphism of algebroids }$X\to Y$ is a pair $(f,\phi)$, where
$f:X\to Y$ is a morphism of algebraic stacks, and $\phi:A_X\to A_Y$ is
a morphism of algebras over $f$, such that the diagram
\begin{equation}\label{quasidiga}
\vcenter{\xymatrix{
A_X^\times\rto^\phi\dto_{\iota} & A_Y^\times\dto^\iota\\
I_X\rto^{I_f} & I_Y}}
\end{equation}
commutes. 

The morphism $(f,\phi)$ is a {\bf representable }morphism of
algebroids if $\phi:A_X\to A_Y|_X$ is a monomorphism of sheaves of
algebras over $X$. (This implies that $f$ is Deligne-Mumford representable.)

The morphism $(f,\phi)$ is {\bf inert}, if $\phi:A_X\to A_Y|_X$ is an
isomorphism of finite type algebras. (This implies that $f$ is inert.)
\end{defn} 

There is a natural notion of 2-morphism of algebroid, which makes
algebroids into a 2-category. 

\begin{rmk}[inert morphisms in the case of algebra bundles]\label{imitabc}
Suppose $(X,A_X)\to (Y,A_Y)$ is a morphism of algebroids, where
$A_X$ and $A_Y$ are algebra bundles. If 
$X\to Y$ is an inert morphism of algebraic stacks, then we
automatically have $A_X^\times=A_Y^\times|_X$ and hence $A_X=A_Y|_X$,
and so $(X,A_X)\to (Y,A_Y)$ is an inert morphism of algebroids. 
\end{rmk}

\begin{rmk}[strict algebroids and representable
    morphisms]\label{repsimple}
Suppose $(X,A)$ is a strict algebroid, and $f:Y\to X$  a representable
morphism of algebraic stacks. If $B\subset A|_Y$ is a finite type
subalgebra,  such that $\iota(B^\times)= I_Y\cap
\iota(A^\times|_Y)$ inside $I_X|_Y$, then $(Y,B)$ is
a strict algebroid with a 
representable morphism $(Y,B)\to (X,A)$. (The algebroid condition for $(Y,B)$ is
automatic.)

Every strict algebroid over $f:Y\to X$ and $(X,A)$ comes about in this way.
\end{rmk}

\begin{rmk}[pullbacks]
Suppose $(X,A_X)$ is an algebroid, and $Y\to X$ an inert
morphism of algebraic stacks.  In this case, $Y$ admits a unique
algebroid $A_Y$, endowed with an inert morphism of
algebroids $(Y,A_Y)\to (X,A_X)$. In fact, $A_Y=A_X|_Y$. 

If $(X,A_X)$ is a strict algebroid, then the morphism $Y\to X$ is
necessarily strictly inert, and $(Y,A_Y)$ is necessarily a strict
algebroid.  We call $(Y,A_Y)\to (X,A_X)$ is a {\em strictly inert
  morphism }of strict algebroids. 
\end{rmk}

\begin{defn}\label{bundledef}
We call a morphism of algebroids $(X,A_X)\to (Y,A_Y)$ 
\begin{items}
\item a  vector bundle, 
\item a principal homogeneous $G$-bundle, for an algebraic group
$G$, 
\item a fibre bundle with group $G$ and fibre $F$,
\item a locally closed immersion,
\item a stratification,
\end{items}
 if it is inert, and the underlying morphism of algebraic
stacks $X\to Y$ has the indicated property.
\end{defn}

\begin{rmk}[fibered products]
Fibered products of algebroids exist, and commute with fibered
products of underlying stacks and underlying algebras. Fibered
products of strict algebroids are strict algebroids. The algebroid underlying a
fibered product of linear algebraic stacks is equal to the fibered
product of the underlying algebroids. 
\end{rmk}

\begin{prop}\label{strem}
Suppose  $(X,A)\to (Y,B)$ is a morphism of algebroids
  where $X\to Y$ is a stratification of algebraic stacks, then there exists a
  stratification of algebroids $(X',A')\to (X,A)$, such that the
  composition $(X',A')\to (Y,B)$ is also a stratification of
  algebroids. 

If $(Y,B)$ is strict, then $(X',A')\to (X,A)$ and
$(X',A')\to(Y,B)$ are stratifications of strict algebroids.
\end{prop}
\begin{proof}
Passing to the rank
stratifications of $A$ and $B$, we obtain a commutative diagram of
algebroids
$$\xymatrix{
(X',A')\rto\dto & (Y',B')\dto\\
(X,A)\rto  & (Y,B)\rlap{\,.}}$$
The upper horizontal morphism is an inert morphism of
algebroids, because $A'$ and $B'$ are algebra bundles, by
Remark~\ref{imitabc}. The claim follows.
\end{proof}

\subsubsection{Algebroid Inertia}

\begin{rmk}[Inertia]
Let $(X,A)$ be an algebroid.  Then $(I_X,I_A)$ is another
algebroid, which we call the {\em algebroid inertia }of
$(X,A)$. In fact, $I_A=(A|_{I_X})^\fix$, the subalgebra of $A|_{I_X}$
of elements invariant under the tautological automorphism induced by
the inertia action of $I_X$ on $A$.  The subgroup of units is
$(A^\times|_{I_X})^\fix$, and we have a cartesian diagram
$$\xymatrix{
(A^\times|_{I_X})^\fix\rto\dto& A^\times\dto^\iota\\
I_{I_X}\rto & I_X\rlap{\,.}}$$
If $(X,A)$ is a strict algebroid, then so is $(I_X,I_A)$.
\end{rmk}

\begin{rmk}[Induced algebroid structure on the algebra]\label{lg}
Let $(X,A)$ be an algebroid. Let $(A|_A)^\fix$ be the subalgebra
of the pullback $A|_A$ of elements commuting with the tautological
section of $A|_A$. (This is equal to the space of commuting pairs in
$A\times_XA$). The subgroup of units is $(A^\times|_A)^\fix$, and we
have a cartesian diagram
$$\xymatrix{
(A^\times|_A)^\fix\rto \dto & A^\times\dto^\iota\\
I_A\rto & I_X\,.}$$
This proves that $\big(A,(A|_A)^\fix\big)$ is an algebroid over $A$.

In fact, we have a commutative diagram of algebroids
$$\xymatrix{
(A^\times,(A|_{A^\times})^\fix)\rto^-\iota\dto & (I_X,I_A)\dto\\
(A,(A|_A)^\fix)\rto& (X,A)\rlap{\,.}}$$

If $(X,A)$ is a strict algebroid,  this is a diagram of strict algebroids.
\end{rmk}

\begin{rmk}[semi-simple algebroid inertia]\label{ssai}
There is also a semi-simple version of the algebroid inertia.
To define it, let $(X,A)$ be an algebroid, and consider the
diagram of algebroids 
\begin{equation}\label{dobulk}
\xymatrix{
A^{\times,\ss}\rto\dto & A^\times\rto^\iota\dto & I_X\dto\\
A^\ss\rto &A\rto & X\rlap{\,.}}
\end{equation}
The square on the right was constructed in Remark~\ref{lg}. 
The morphism $A^\ss\to A$ is the composition of a stratification and
an open immersion (see Proposition~\ref{sost}), in particular it is
inert. Hence we get an induced algebroid structure on $A^\ss$.
Requiring the square on the left to be a cartesian diagram of
algebroids defines the algebraic stack $A^{\times,\ss}$ and the
algebroid structure over it. 

If $(X,A)$ is a strict algebroid, all objects in Diagram (\ref{dobulk}) are
strict, and $\iota$ is an isomorphism. We are then justified in
defining $I_X^\ss=A^{\times,\ss}$, and calling it the {\em semi-simple
  algebroid inertia }of $X$. In the general case we define
$I_X^{\circ,\ss}=A^{\times,\ss}$, and also call it the {\em
  semi-simple algebroid inertia}, by a slight abuse of language. 
\end{rmk}

\subsubsection{Idempotents and algebroids}

\begin{lem}\label{equaler}
Let $a\in A^{\times,\ss}$ be a semi-simple invertible global section of an algebra
bundle $A\to X$. Let $\tilde A^{\times,\ss}=A^{\times,\ss}/\Gm$ be the quotient of
$A^{\times,\ss}$ by the subgroup of scalars. Consider the action of
the group $A^\times$ on $A^{\times,\ss}\to \tilde
A^{\times,\ss}$ by conjugation.  The induced group homomorphism
$$\Stab_{A^\times}(a)\longrightarrow \Stab_{A^\times}[a]\,,$$
where $[a]$ is the class of $a$ in $\tilde A^{\times,\ss}$, is an open
immersion.
\end{lem}
\begin{proof}
Let $Y_a$ be the relative spectrum of $\O_X[a]$ over $X$. 
The epimorphism of commutative algebras $\O_X[x]\to \O_X[a]$ gives rise
to a closed immersion $\phi:Y_a\to (\Gm)_X$, because $a$ is
invertible. We get an induced proper morphism
\begin{align}
Y_a\times_XY_a&\longrightarrow (\Gm)_X\label{yty}\\
(\lambda,\mu)&\longmapsto \phi(\lambda)/\phi(\mu)\,.\nonumber
\end{align}
As $Y_a\to X$ is unramified, the diagonal $Y_a\to Y_a\times_X Y_a$ is an
open immersion, so the complement, denoted $(Y_a\times Y_a)^{\not=}$,
is closed in $Y_a\times Y_a$, and hence proper over $X$. Hence the
image of $(Y_a\times Y_a)^{\not=}$ in $(\Gm)_X$, denoted by $Z$, is
closed. The complement of $Z$ in $(\Gm)_X$ is hence an open
neighborhood $U$ of the identity section. 

We have a cartesian diagram
$$\xymatrix{
\Stab_{A^\times}(a)\rto\dto & U\dto\\
\Stab_{A^\times}[a]\rto\dto & (\Gm)_X\dto\\
A^\times\rto_{b\mapsto bab^{-1}a^{-1}} & A^\times\rlap{\,.}}$$
The lower square is cartesian by the definition of
$\Stab_{A^\times}[a]$. The fact that the upper square is cartesian
follows from the fact that the image of $\Stab_{A^\times}[a]$ in
  $(\Gm)_X$ is contained in the image of (\ref{yty}). The latter claim
follows from the fact that if ${}^{b}a=\lambda a$, for a scalar
$\lambda\in(\Gm)_X$, then $\lambda$ is necessarily a quotient of
eigenvalues of $a$. 
\end{proof}

\begin{rmk}\label{era}
Let $(X,A)$ be an algebroid, and let $Y=E_k(A)$ be the stack of
complete labelled sets of $k$ orthogonal idempotents in $A$. Let us
write $B=(A|_Y)^{e_1,\ldots,e_k}$, for the subalgebra of $A|_Y$, consisting of
elements commuting with each of the $k$ tautological idempotents in
$A|_Y$. The homomorphism $\iota:A^\times\to I_X$ restricts to a homomorphism
$B^\times\to I_Y$, and defines an algebroid structure over $Y$.
The algebra $B$ is  endowed with a canonical complete set of
orthogonal {\em central }idempotents, and hence decomposes as
a product $B=B_1\times\ldots\times B_k$. 
\end{rmk}

Let $(X,A)$ be an algebroid, where $A$ is an algebra bundle, and
let $e_1,\ldots,e_k$ be a complete set 
of orthogonal central idempotents in $A$, decomposing $A$ into a
product of algebra bundles
$A=A_1\times\ldots\times A_k$. 
We get an induced decomposition 
$$A^{\times,\ss}=A_1^{\times,\ss}\times\ldots\times
A_k^{\times,\ss}\subset A\,,$$
and  the algebroid structure on
$A^{\times,\ss}$ is the fibered product over $X$ of the
algebroid structures on the $A_i^{\times,\ss}$, $i=1,\ldots,k$. 

We obtain an 
embedding of algebras $\alpha:\aaa_Y^k\to A$.  Via $\alpha$, the torus
$\Gm^k$ acts on $A$ by left multiplication.  
The action of $\Gm^k$ on $A$ preserves the semi-simple
units
 and the restricted action of $\Gm^k$ on
$A^{\times,\ss}$ is fiberwise free, so the quotient
$$\tilde
A^{\times,\ss}=A^{\times,\ss}/\Gm^k=A_1^{\times,\ss}/\Gm\times
\ldots\times A_k^{\times,\ss}/\Gm$$ 
is representable over $X$. 

\begin{prop}\label{bundlequack}
We claim that $\tilde A^{\times,\ss}$ has a canonical algebroid
structure, and we get an induced commutative diagram
$$\xymatrix{
A^{\times,\ss}\rto\drto  & \tilde A^{\times,\ss}\dto\\
&X}$$ 
of algebroids, where the horizontal map is inert, hence a
principal $\Gm^k$-bundle of algebroids.
\end{prop}
\begin{proof}
Because inert morphisms of algebroids are stable under
composition and pullback, this claim reduces to the case $k=1$, and
$e_1=1$, which we will now consider. 

Let us denote the tautological section of $A^\times$ over
$A^{\times,\ss}$ by $a$.  Then the algebra over $A^{\times,\ss}$ is
given by the centralizer algebra $(A|_{A^{\times,\ss}})^a$. This
algebra descends to the quotient $\tilde A^{\times,\ss}$, because
the centralizer of an algebra element does not depend on its
equivalence class. Let us denote this descended algebra by 
$(A|_{\tilde   A^{\times,\ss}})^{[a]}$. The units in this algebra are
identified with 
$$(A^\times|_{\tilde
  A^{\times,\ss}})^{[a]}=\Stab_{(A^\times|_{\tilde
    A^{\times,\ss}})}(a)\,,$$
which is an open subgroup of 
$$\Stab_{(A^\times|_{\tilde
    A^{\times,\ss}})}[a]\,,$$
by Lemma~\ref{equaler}.  We also have a cartesian diagram
$$\xymatrix{
\Stab_{(A^\times|_{\tilde A^{\times,\ss}})}[a]\rto \dto &
A^{\times}|_{\tilde A^{\times,\ss}}\dto^\iota\\
I_{\tilde A^{\times,\ss}}\rto & I_X|_{\tilde
  A^{\times,\ss}}\rlap{\,,}}$$
because the inertia stack $I_{\tilde A^{\times,\ss}}$ can be identified as 
$$I_{\tilde A^{\times,\ss}}=\{([a],\phi)\in \tilde A^{\times,\ss}\times_X I_X\mid
{}^\phi[a]=[a]\}\,.$$
This proves that $\Stab_{(A^\times|_{\tilde A^{\times,\ss}})}[a]$ is
an open subgroup of $I_{\tilde A^{\times,\ss}}$. Composing our two
open immersions, we obtain an open immersion of groups over $\tilde
A^{\times,\ss}$
from $(A^\times|_{\tilde A^{\times,\ss}})^{[a]}$ to $I_{\tilde
  A^{\times,\ss}}$, endowing $\tilde A^{\times,\ss}$ with the
structure of an algebroid, as required. 

We get an induced morphism of algebroids $A^{\times,\ss}\to
\tilde A^{\times,\ss}$, which is inert, by construction.
\end{proof}

\begin{rmk}
Even if $(X,A)$ is strict, $\tilde A^{\times,\ss}$ is not
necessarily strict. It is this construction, in fact, which
makes it impossible for us to restrict attention to strict
algebroids. 
\end{rmk}

\begin{rmk}\label{karoubi}
Let $(\MM,\AA)$ be the algebroid underlying a linear algebraic stack.
Suppose that $\MM$ admits direct sums.  Let $\MM_\ast$ be the
complement of the zero object in $\MM$.  We obtain a canonical 
morphism of algebroids
\begin{align}\label{surprize}
\underbrace{\MM_\ast\times\ldots\times \MM_\ast}
_{\text{$n$ times}}&\longrightarrow E_n(\AA)\\\nonumber
(x_1,\ldots,x_n)&\longmapsto (x_1\oplus\ldots\oplus x_n;\pi_1\ldots,\pi_n)\,,
\end{align}
where $\pi_1,\ldots,\pi_n$ are the projectors corresponding to the
factors $x_1,\ldots,x_n$ of $x_1\oplus\ldots\oplus x_n$. Over every
$R$-scheme $S$, this morphism is fully faithful.  The underlying
morphism of algebraic stacks (\ref{surprize}) is a monomorphism. 

If we require all fibers $\MM(S)$ to be Karoubian, i.e., we require
all idempotents to admit the corresponding direct summands,
(\ref{surprize}) is an isomorphism of algebroids.

We call a linear algebraic stack $\MM$ {\bf Karoubian}, if it admits
direct sums and all fibers are Karoubian. 
\end{rmk}

\subsubsection{Algebroid representations}

\begin{defn}
Let $X$ be an algebroid.  A {\bf representation }of $X$ is a
morphism of algebroids $\rho:X\to \mf{Vect}$, to the algebroid underlying the
linear stack $\mf{Vect}$ of vector bundles. If $\rho$ factors through
vector bundles of rank $n$, i.e., defines a morphism $X\to
B\GL_n$, with its natural algebroid structure (see
Example~\ref{vectalg}), we say that $\rho$ has rank $n$. 

If the algebroid morphism $X\to B\GL_n$ is representable, we call the
representation $\rho$ {\bf faithful}. 
\end{defn}

To give a representation of the algebroid $(X,A)$ is equivalent to specifying a
vector bundle $V$ over $X$, together with a morphism of algebras $A\to
\End(V)$, such that the induced morphism on unit groups $A^\times\to
GL(V)$ makes the diagram
$$\xymatrix{
A_X^\times\rto^{\iota}\drto& I_X\dto\\
&GL(V)}$$
commute, where $I_X\to GL(V)$ is the inertia representation given by
the vector bundle $V/X$.

The representation $V$ of rank $n$ is faithful if and only if $A\to \End(V)$ is a
monomorphism of algebras over $X$. If this is the case, the underlying
morphism of stacks $X\to B\GL_n$ is Deligne-Mumford representable, and the
$\GL_n$-bundle of frames $Y=\uIsom(V,\O^n)$ is represented by a
Deligne-Mumford stack.

So a faithful representation identifies $X$ as  a quotient stack
$X=\GL_n\backslash Y$, where $Y$ is a Deligne-Mumford stack. The
algebroid structure on $X$ is then given by 
a $\GL_n$-invariant subalgebra $B\hookrightarrow M_{n\times n}\times Y$, such
that the subgroup $B^\times \hookrightarrow \GL_n\times Y$ is equal to the
subgroup $\Stab_{\GL_n}Y\hookrightarrow {\GL_n}\times Y$.  

\begin{rmk}\label{tautex}
Every algebroid $(X,A)$, with $A$ an algebra bundle, admits  the
tautological adjoint representation, given 
by the adjoint representation of the algebra $A$ on itself. By
contrast, the representation of $A$ on itself by left multiplication
is not an algebroid representation, unless $A=0$.
\end{rmk}

\begin{prop}\label{faith}
Every algebroid $(X,A)$ admits a stratification  $X'\subset X$,
such that the restricted algebroid $(X',A|_{X'})$ admits a faithful
representation.
\end{prop}
\begin{proof}
Without loss of generality, assume that $A$ is an algebra
bundle with smooth unit group $A^\times$, and that $X$ is a connected
isotrivial  gerbe over the Deligne-Mumford stack $S$.  Thus $X\to S$
admits a section over  a finite \'etale
cover $S'\to S$.

To begin with, we consider the case where $S'=S$, i.e., the case where
$X$ is a trivial gerbe over $S$. By Remark~\ref{gerberemark}, we can
assume that $A$ descends to $S$, and that we are dealing with the
algebroid $(BA^\times,A^\times\backslash A)$. Then we can consider the
representation of $A$ on itself by left multiplication
$\ell:A\to \End(A)$. It restricts to a representation of $S$-group
schemes $\ell:A^\times\to GL(A)$. We get an induced morphism of
algebraic stacks $BA^\times\to BGL(A)$, which is covered by the
morphism of algebras $A^\times\backslash A\to
BGL(A)\backslash \End(A)$. Since $\ell:A\to \End(A)$ is injective,
this gives the required faithful representation of
$(BA^\times,A^\times\backslash A)$. 

(Note that this construction does not contradict Remark~\ref{tautex}.
The vector bundle over $X$ defined by the left representation of $A^\times$
on itself is different from the vector bundle underlying the algebroid
$A$ over $X$, which is
given by the adjoint representation of $A^\times$ on itself.)

Now consider the general case. The pullback $(X',A')$ of $(X,A)$ to
$S'$ is again an algebroid, as $X'\to X$ is inert. Since
$X'\to S'$ is a trivial gerbe, $(X',A')$ admits a faithful
representation $A'\to \End(V)$, on a vector bundle $V$ over
$X'$. Let $\pi:X'\to X$ be the projection.  Then $\pi_\ast V$ is a
faithful representation of $A$. In fact, by adjunction, the embedding
$\pi^\ast A\to \End (V)$ gives rise to an embedding
$A\to\pi_\ast\End(V)\to\End(\pi_\ast V)$. 
\end{proof}

\subsubsection{Clear algebroids}

Suppose that $(X,A)$ is an algebroid, and that $X$ is a
connected
gerbe over the Deligne-Mumford stack $\ol X$. Then the centre $Z(A)$
descends to a commutative finite 
type algebra over $\ol X$, by Lemma~\ref{inertiadescent}.

\begin{defn}\label{clear}
We call an algebroid $(X,A)$ {\bf clear}, if 
\begin{items}
\item $A$ and $Z(A)$ are algebra bundles over $X$,
\item $X$ is a connected isotrivial gerbe over $\ol X$,
\item the Deligne-Mumford stack $\ol X$ is connected,
\item $ZE(A)\to X$ is finite \'etale.
\end{items}
\end{defn}

For a clear algebroid, $ZE(A)$ and $PZE(A)$ descend to  finite \'etale $\ol
X$-schemes. The definitions of {\em central rank}, {\em split central
  rank}, and {\em central  type }apply to clear algebroids. 

For every algebroid $(X,A)$, over  a finite type algebraic
stack $X$, there exists a stratification of $X$, such that the restricted
algebroids over the pieces of the stratification are all clear. This
follows from Proposition~\ref{bdlstr} and Lemma~\ref{Estrat}. 

\section{The spectrum of semi-simple inertia}\label{spectrum}

Let $K(\DM)$ be the $\qq$-vector space on (isomorphism classes of)
finite type Deligne-Mumford stacks, modulo scissor relations and
bundle relations, i.e., equations of the form $[Y]=[F\times X]$,
whenever $Y\to X$ is a fibre bundle with special structure group and
fibre $F$. The product over $\spec R$ makes $K(\DM)$ a $\qq$-algebra.
We write $q$ for the class of the affine line in $K(\DM)$.

Let $\MM$ be a linear algebraic stack, and $\AA\to\MM$  its universal
endomorphism algebra.  Recall that $(\MM, \AA)$ is an algebroid (c.f.
Example~\ref{linex}).

\subsubsection{Stack functions}

\begin{defn}
A {\bf stack function }is a representable  morphism of algebroids $(X,A)\to
(\MM,\AA)$, such that $X$ is of finite type.

The $K$-module of $\MM$, notation $K(\MM)$, is the free
$\qq$-vector space on (isomorphism classes of) stack functions,
modulo the scissor and bundle relations  relative to $(\MM,\AA)$. 
The class in $K(\MM)$
defined by a stack function $X\to\MM$ will be denoted $[X\to\MM]$. 

A {\bf scissor relation relative $\MM$ }is
$$[X\to\MM]=[Z\to X\to\MM]+[X\setminus Z\to X\to\MM]\,,$$
for any closed immersion of algebroids  $Z\hookrightarrow X$, and any
stack function $X\to \MM$. The substacks $Z$ and $X\setminus Z$ are
endowed with their respective pullback algebroids. 

A {\bf bundle relation relative $\MM$ }is 
$$[Y\to X\to \MM]=[F\times X\to X\to\MM]\,,$$
for any fibre bundle $Y\to X$ of algebroids with special structure group and fibre
$F$, see Definition~\ref{bundledef}.
\end{defn}

 There is  an action of $K(\DM)$ on
  $K(\MM)$, given by
$$[Z]\cdot[X\to\MM]=[Z\times X\to X\to\MM]\,.$$ This action makes
 $K(\MM)$ into a $K(\DM)$-module.

The additive zero in $K(\MM)$ is given by the empty algebroid
$$0=[\varnothing\to\MM]\,.$$

If $\MM$ admits a zero object (Definition~\ref{nul}), we denote the
corresponding stack 
function by $1=[\spec R\stackrel{0}{\longrightarrow} \MM]$. We can use
it to embed $K(\DM)$ into $K(\MM)$ 
via $[X]\mapsto [X]\cdot 1=[X\to\spec
  R\stackrel{0}{\longrightarrow}\MM]$.
We will always assume that $\MM$ admits a zero object. 

\hide{

For $\MM=\Spec R$, we write $K(\mf{Alg})$ for $K(\Spec R)$, to
forestall confusion.  Thus $K(\mf{Alg})$ is the $K(\Var)$-module of
finite type algebroids modulo scissor relations. In fact,
$K(\mf{Alg})$ is a $K(\Var)$-algebra in a canonical way.\comment{This
  doesn't make sense with the representable definition}
}

\unsure{
Recall Remark~\ref{repsimple}, which says that to prove that a given
$(X,A)\to(\MM,\AA)$ is a stack function, we do not need to check the
algebroid condition on $(X,A)$.

}

\subsubsection{The filtration by split central rank}

We call a stack function $X\to\MM$  {\bf clear}, if $X$ is a clear
algebroid (Definition~\ref{clear}). 

\begin{defn}
  We introduce the 
{\bf filtration by split central rank }$K^{\geq k}(\MM)$ on $K(\MM)$,
by declaring $K^{\geq 
  k}(\MM)$ to be generated as a $\qq$-vector space by clear stack functions
$[X\to\MM]$, 
such that $A_X$ admits $k$ orthogonal central non-zero idempotents (globally).
\end{defn}

Alternatively, $K^{\geq k}(\MM)$  is generated by $[X\to\MM]$, where
$X$ is a clear algebroid such that 
$PZE(A_X)$ has at least $k$ components. 

Each filtered piece $K^{\geq k}(\MM)$ is stable under scalar
multiplication by $K(\DM)$. 
Let us introduce the abbreviation
$$K^{\geq k}/K^{>k}(\MM)=K^{\geq k}(\MM)/K^{>k}(\MM)\,.$$

\begin{rmk}
Trying to define a direct sum decomposition of $K(\MM)$ by split
central rank would not work, because a clear algebroid  $X$ of split central
rank $k$ may very well admit a closed substack $Z\subset X$ whose
restricted algebroid is again clear, but of
split central rank larger than $k$. Similarly, the bundle relations
do not respect split central rank. 
\end{rmk}

The zero ring has no non-zero central idempotents, but any non-zero ring has at
least one.  Therefore, $K(\DM)\subset K(\MM)$ is a complement for
$K^{>0}(\MM)$ in $K(\MM)=K^{\geq0}(\MM)$, i.e., $K(\MM)=K(\DM)\oplus
K^{>0}(\MM)$. In particular, we have
$$K^{\geq0}(\MM)/K^{>0}(\MM)=K(\DM)\,.$$

\unsure{
\subsubsection{Refinement by central type}\comment{I think this is doubtful, at least
    I can't prove it, because I don't know that the intersection of
    the scissor relations and the stuff generated by clear algebroids
    of central rank at least $k$, is generated by scissor relations
    among clear algebroids of central rank at least $k$}

\begin{prop}
The subquotient $K^{\geq k}/K^{>k}(\MM)$ decomposes into a direct
sum 
\begin{equation}
K^{\geq
  k}/K^{>k}(\MM)=\bigoplus_{|\lambda|=k}K^{\lambda}(\MM)'\,,
\end{equation}
where the sum is over all partitions with $k$ parts, and
$K^\lambda(\MM)'$ is generated by all clear   
stack functions of central type  $\lambda$. (The prime symbol is added
to the notation to distinguish these spaces from others considered below.)
\end{prop}
\begin{proof}
Consider a clear algebroid $X$ of central type $\lambda$, with
$|\lambda|=k$.  For every 
stratification of $X$, all pieces of the stratification have central types which are
refinements of $\lambda$.  All pieces whose central types are proper
refinements of $\lambda$ are zero in $K^{\geq 
  k}(\MM)/K^{>k}(\MM)$. 
\end{proof}

}

\subsection{The idempotent operators $E_r$}

Let $E_r$ denote the operator on $K(\MM)$ which maps a stack function
$[X\to\MM]$ to $[E_r(X)\to X\to\MM]$, where $E_r(X)=E_r(A_X)$ is the
stack of $r$-tuples of non-zero orthogonal idempotents adding to unity in
$A_X$, see Definition~\ref{famid}. The algebroid structure on $E_r(X)$
is described in Remark~\ref{era}. 

The operators $E_r$ are well-defined, because applying $E_r$ to a
stratification or a fiber bundle of algebroids gives rise to another
inert morphism of algebroids of the same type. 
The operators $E_r:K(\MM)\to K(\MM)$ are $K(\DM)$-linear.

This definition applies also to $r=0$. The stack $E_0(X)$ is
empty if $A_X\not=0$, and $E_0(X)=X$, if $X$ is a Deligne-Mumford stack. Hence $E_0$
is diagonalizable, and has eigenvalues $0$ and $1$.  The kernel
($0$-eigenspaces) is $K^{>0}(\MM)\subset K(\MM)$, the image
($1$-eigenspaces) is denoted by $K^0(\MM)\subset K(\MM)$, and is
generated by all stack functions $[X\to\MM]$, where $X$ is a
Deligne-Mumford stack. In fact, $K^0(\MM)=K(\DM)\subset K(\MM)$. 

For $r=1$, the operator $E_1$ vanishes on stack functions $[X\to\MM]$,
where $X$ is a Deligne-Mumford stack, and acts as
identity on stack functions for which $A_X\not=0$. Hence, $E_1$ is also
diagonalizable with eigenvalues~$0$ and~$1$.  The kernel of $E_1$ is
$K^0(\MM)$, and the image is $K^{>0}(\MM)$.  Hence $E_0$ and $E_1$ are
complementary idempotent operators on $K(\MM)$, i.e., they are
orthogonal to each other and add up to the identity.

Recall the Stirling number of the second kind, $S(k,r)$, which is
defined in such a way that $r!\,S(k,r)$ is the number of
surjections from $\ul k$ to $\ul r$. Here, and elsewhere, We write
$\ul n=\{1,\ldots,n\}$.  

\begin{thm}\label{Ethm}
The operators $E_r$, for all $r\geq0$, preserve the
filtration $K^{\geq k}(\MM)$ by split central rank.  On the subquotient
$K^{\geq k}/K^{>k}(\MM)$, 
the operator $E_r$ acts as multiplication by $r!\,S(k,r)$.
\end{thm}

\begin{proof}\comment{A good thing to check: does this proof work in
    the context of linear algebraic stacks?}
Consider a clear algebroid $(X,A)$ with a morphism $X\to\MM$
defining the stack function $[X\to\MM]$ in $K(\MM)$. Let $n$ be the
central rank of $X$, and $k$  the split
central rank of $X$.
The filtered piece $K^{\geq k}(\MM)$ is generated by such $[X\to\MM]$. 

Denote by $X\to \ol X$ the coarse Deligne-Mumford stack of $X$. By
assumption, both $X$ and $\tilde X$ are connected and hence 
admit Grothendieck style Galois theories (see
\cite[Section~4]{Noohi}). 

Let $\tilde  X\to  X$ be a connected Galois cover with Galois group
$\Gamma$, which  trivializes $PZE(A)\to  X$. As $PZE(A)$ descends
to $\ol X$, this Galois cover can be constructed as a pullback from
$\ol X$. Therefore, the morphism $\tilde X\to
X$ is inert and hence $\tilde X$ inherits, via pullback,  the structure
of an algebroid, and hence $[\tilde X\to X\to \MM]$ is a stack
function. 

Recall that the degree of the cover $PZE(A)\to X$ is $n$, and the
number of components of $PZE(A)$ is $k$. 

By
labelling the components of the 
pullback of $PZE(A)$ to $\tilde X$, we obtain an action of $\Gamma$
on the set $\ul n=\{1,\ldots,n\}$ and an isomorphism of finite \'etale
covers of $X$
\begin{align*}
\tilde X\times_\Gamma\ul n&\longiso PZE(A)\\
[x,\nu]&\longmapsto e_{[x,\nu]}\,.
\end{align*}
Both source and target of this isomorphism support natural algebroids
and the isomorphism preserves them.
The number of orbits of $\Gamma$ on $\ul n$ is $k$. 

Then we also have an isomorphism
\begin{align*}
\tilde X\times_\Gamma\Epi(\ul n, \ul r)&\longiso ZE_r(A)\\
[x,\phi]&\longmapsto
\Big(\sum_{\phi(\nu)=\rho}e_{[x,\nu]}\Big)_{\rho=1,\ldots,r}\,,
\end{align*}
where  $ZE_r$ denotes the stack of labelled complete sets of
$r$ orthogonal central idempotents.  Again, both stacks involved are
in fact algebroids, and this isomorphism is an isomorphism of
algebroids.

Hence, we may calculate as follows (all stacks involved are endowed
with their natural algebroid structures):
\begin{align*}
ZE_r[X\to\MM]
&=[\tilde X\times_\Gamma\Epi(\ul n,\ul r)\to\MM]\\
&=[\tilde X\times_\Gamma\coprod_{\phi\in\Epi(\ul n,\ul
    r)/\Gamma}\Gamma/\Stab_{\Gamma}\phi\to\MM]\\ 
&=\sum_{\phi\in\Epi(\ul n,\ul r)/\Gamma}[\tilde
  X/\Stab_\Gamma\phi\to\MM]\\
&=\sum_{\phi\in\Epi(\ul n,\ul r)^\Gamma}[X\to\MM]+ 
\sum_{\stack{\phi\in\Epi(\ul n,\ul
    r)/\Gamma}{\Stab_\Gamma\phi\not=\Gamma}}[\tilde
  X/\Stab_\Gamma\phi\to\MM] \,.
\end{align*}
Now, we have $\Epi(\ul n,\ul r)^\Gamma=\Epi(\ul n/\Gamma,\ul r)$, and
hence
$$\#\Epi(\ul n, \ul r)^\Gamma=r!\,S(k,r)\,.$$
Thus, we conclude,
$$ZE_r[X\to\MM]=r!\,S(k,r)\,[X\to\MM]+
\sum_{\stack{\phi\in\Epi(\ul n,\ul
    r)/\Gamma}{\Stab_\Gamma\phi\not=\Gamma}}[\tilde
  X/\Stab_\Gamma\phi\to\MM]\,. $$
For any proper subgroup $\Gamma'\subset\Gamma$, the quotient
$X'=\tilde X/\Gamma'$ is an intermediate cover $\tilde X\to X'\to X$,
such that $X'\not=X$. The pullback of $PZE(A)$ to $X'$ has more than
$k$ components, because the number of orbits of $\Gamma'$ on $\ul n$
is larger than $k$.  Thus we have proved the theorem for $ZE_r$,
instead of $E_r$. 

Now observe that $ZE_r(A)\subset E_r(A)$ is a closed
substack, because $ZE_r(A)\to X$ is proper and $E_r(A)\to X$ is
separated. So we can write 
$$E_r[X\to\MM]=ZE_r[X\to\MM]+[NZE_r(A)\to X\to \MM]\,,$$
where $NZE_r(A)$ is the complement of $ZE_r(A)$ in $E_r(A)$. 
To prove that $[NZE_r(A)\to\MM]\in K^{>k}(\MM)$, let $Y\hookrightarrow NZE_r(A)$
be a locally closed embedding, such that the algebroid
$\big(E_r(A),A^\fix\big)|_Y$ is clear. 

Consider the embedding of
algebras $A^\fix|_Y\hookrightarrow A|_Y$.  It induces an embedding of commutative
algebras 
$Z(A|_Y)\hookrightarrow 
Z(A^\fix|_Y)$, because $Z(A|_Y)\subset A^\fix|_Y$. 
The algebra $A|_Y$ comes with $r$ tautological idempotent sections, all of which
are contained in $Z(A^\fix|_Y)$, but at least
one of which is not contained in  $Z(A|_Y)$. So by
Proposition~\ref{rankineq}~(ii), the split central rank of $A^\fix|_Y$ is
strictly larger than the split central rank of $A|_Y$.  The latter is
at 
least as big as $k$, the split central rank of $A$, because the split
central rank cannot decrease under base extension. This shows that
$[Y\to\MM]\in K^{>k}(\MM)$ and finishes the proof.
\end{proof}

\begin{cor}\label{simult}
The operators $E_r$, for $r\geq0$ are simultaneously
diagonalizable. The common eigenspaces form a family $K^k(\MM)$ of
subspaces of $K(\MM)$ indexed by non-negative integers
$k\geq0$, and 
\begin{equation}\label{idret}
K(\MM)=\bigoplus_{k\geq0}K^k(\MM)\,.
\end{equation}
Moreover, for every $r\geq0$,
$$K^{\geq r}(\MM)=\bigoplus_{k\geq r} K^{k}(\MM)\,.$$
Let $\pi_k$ denote the projection onto $K^k(\MM)$. We
have
$$E_r\pi_k=r!\,S(k,r)\pi_k\,,$$
for all $r\geq0$, $k\geq0$. 
\end{cor}
\begin{proof}
First remark that for given $r$, the numbers $r!\,S(k,r)$ form a
monotone increasing sequence of integers. 

Then note that the operators $E_r$ pairwise commute: the composition
$E_r\circ E_{r'}$ associates to an algebroid $(X,A)$ the stack of
pairs $(e,e')$, where both $e$ and $e'$ are complete families of
non-zero orthogonal idempotents in $A$, the length of $e$ being $r$
and the length of $e'$ being $r'$, and the members of $e$ commuting
with the members of $e'$. 

Finally, let us prove that, for every $k$ and  every $r$, the
$\qq$-vector space $K^{\geq k}(\MM)$ is a union of
finite-dimensional subspaces invariant by $E_r$. 

For this, define $K(\MM)_{\leq N}$ to be generated as $\qq$-vector
space by stack functions $[X\to\MM]$, where $X$ is a clear
algebroid, such that  the rank of the vector bundle underlying the algebra
$A_X\to X$ is bounded above by $N$. This is an ascending filtration of
$K(\MM)$, which is preserved by $E_r$. Set 
$$K^{\geq k}(\MM)\cap
K(\MM)_{\leq N}=K^{\geq k}(\MM)_{\leq N}\,.$$

Suppose  $x=[X\to\MM]$ is a stack function with $X$ a clear
algebroid of split central rank $k$, and let $N$ be the rank of the 
vector 
bundle underlying $A_X$. Note that  $k\leq N$, because for a commutative
algebra, the number
of primitive idempotents is bounded by the  rank of the underlying
vector bundle.  We deduce that for $k>N$, we have $K^{\geq
  k}(\MM)_{\leq N}=0$. 

On the other hand, 
Theorem~\ref{Ethm} implies by induction that
$$E^i_r(x)\in \qq\,x+\qq\,E_r(x)+\ldots+\qq\,E^{i-1}_r(x)+K^{\geq k+i}(\MM)\,.$$
Applying this for $i=N-k+1$, we see that 
$$E_r\big(E^{N-k}_r(x)\big)\in
\qq\,x+\qq\,E_r(x)+\ldots+\qq\,E^{N-k}_r(x)\,,$$
and hence that $\qq\,x+\qq\,E_r(x)+\ldots+\qq\,E^{N-k}_r(x)$ is
invariant under $E_r$. 

This proves that any $x\in K^{\geq k}(\MM)$ is contained in a
finite-dimensional subspace invariant under $E_r$.  Standard
techniques from finite-dimensional linear algebra over $\qq$ now imply
the result. 
\end{proof}

\begin{rmk}
The proof of Theorem~\ref{Ethm} and its corollary show that the
central versions $ZE_r$ of the $E_r$ are also diagonalizable.  On the
other hand, the
$ZE_r$ do not commute with each other, and so  are less
useful. 
\end{rmk}

\begin{cor}\label{kercor}
For $r\geq1$, we have
$$\ker E_r=\bigoplus_{k<r}K^k(\MM)\,.$$
In particular, for any $x\in K(\MM)$, we have $E^rx=0$, for
$r\gg0$. 
\end{cor}

\begin{cor}\label{forforpi}
For every $k\geq0$, we have
$$\pi_k=\sum_{r=k}^\infty \frac{s(r,k)}{r!}\,E_r\,,$$
where the $s(n,k)$ are the Stirling numbers of the first kind. In
particular, $\pi_0=E_0$, and  
$$\pi_1=\sum_{r=1}^\infty \frac{(-1)^{r+1}}{r}\,E_r\,.$$
\end{cor}
\begin{proof}
We have 
$$\id=\sum_{\ell\geq0}\pi_\ell\,,$$
and hence
$$E_r=\sum_{\ell\geq0}E_r\pi_\ell=\sum_{\ell\geq0}r!\,S(\ell,r)\,\pi_\ell\,,$$
and therefore
\begin{multline*}
\sum_{r\geq0}\frac{s(r,k)}{r!}\,E_r
=\sum_{r\geq0}\frac{s(r,k)}{r!}\sum_{\ell\geq0}r!\,S(\ell,r)\,\pi_\ell\\
=\sum_{\ell\geq0}\Big(\sum_{r\geq0}S(\ell,r)s(r,k)\Big)\,\pi_\ell
=\sum_{\ell\geq0}\delta_{\ell,k}\,\pi_\ell=\pi_k\,,
\end{multline*}
by the inverse relationship between the Stirling numbers of the first
and second kind.
\end{proof}

\begin{rmk}
The Stirling numbers of the first kind appear in the Taylor expansions
of the powers of the logarithm:
$$\sum_{r=k}^\infty \frac{s(r,k)}{r!}\,t^r=\frac{1}{k!}\log(1+t)^k\,.$$
\end{rmk}

\begin{defn}\label{formalop}
Let $t$ be a formal variable. We define the operator
$$\pi_t:K(\MM)[t]\longrightarrow K(\MM)[t]$$
by the formula
$$\pi_t(\xi)=\sum_k\pi_k(\xi)\,t^k\,.$$
and extending $K(\Var)[t]$-linearly. We can write, formally,
$$\pi_t=\sum_{k}\pi_k\,t^k\,.$$
\end{defn}

\begin{rmk}\label{convenient}
We have the following convenient formula:
$$\pi_t=\sum_n \binom{t}{n}E_n\,.$$
It follows from Corollary~\ref{forforpi}, using the identity
$$\sum_{k}\frac{s(n,k)}{n!}t^k=\binom{t}{n}\,.$$
\end{rmk}

\begin{ex}
\label{bgl2ex}
The universal rank 2 vector bundle $GL_2\backslash \aaa^2\to BGL_2$,
and its classifying morphism to $\mf{Vect}$ define a Hall algebra
element $[BGL_2\to\mf{Vect}]\in K(\mf{Vect})$, which we will
abbreviate to $[BGL_2]$. To decompose $[BGL_2]$ into its pieces
according to (\ref{idret}), we consider the action of $E_2$, as we
have $E_r[BGL_2]=0$, for all $r>2$.  In fact,
$$E_2[BGL_2]=[BT]\,,\qquad\text{and}\qquad E_2[BT]=2[BT]\,,$$
where $T$ is a maximal torus in $GL_2$. Thus $\qq[BGL_2]+\qq[BT]$ is a
subspace of $K(\mf{Vect})$ invariant under $E_2$, and the matrix of
$E_2$ acting on this subspace is 
\begin{equation}\label{matr}
\begin{pmatrix}0&0\\1&2\end{pmatrix}\,.
\end{equation}
This matrix is lower triangular, with different numbers on the
diagonal, hence diagonalizable over $\qq$. In fact, the diagonal
entries are $2\,S(1,2)=0$ and $2\,S(2,2)=2$. 
Diagonalizing (\ref{matr})
gives the eigenvectors 
\begin{items}
\item $v_1=[BGL_2]-\frac{1}{2}[BT]$ with eigenvalue $0$,
\item $v_2=\frac{1}{2}[BT]$ with eigenvalue $2$.
\end{items}
Therefore, we have $v_1\in K^1(\mf{Vect})$ and $v_2\in
K^2(\mf{Vect})$, and since $[BGL_2]=v_1+v_2$, we have found the
required decomposition of $[BGL_2]$. 
\end{ex}

\subsection{The spectrum of semisimple Inertia}

The connected semi-simple inertia operator on $ K(\MM)$ is the
$\qq$-linear endomorphism 
\begin{align*}
I^{\circ,\ss}: K(\MM)&\longrightarrow  K(\MM)\\
[X\to\MM]&\longmapsto [I_X^{\circ,\ss}\to X\to\MM]\,.
\end{align*}
Here $I_X^{\circ,\ss}=A_X^{\times,\ss}$ denotes the
semisimple algebroid inertia of the algebroid $X$,
see Remark~\ref{ssai}. Note that $I^{\circ,\ss}$ respects
the scissor and bundle
relations defining $\tilde K(\MM)$, and is linear over $K(\DM)$, 
because passing to connected inertia commutes with inert
pullbacks. 

Note that $I^{\circ,\ss}$ commutes with $E_r$, for every $r$. Both
compositions $E_r\circ I^{\circ,ss}$ and $I^{\circ,\ss}\circ E_r$
associate to an algebroid $(X,A)$ the stack of pairs $(a,e)$, where $a$ is a
semi-simple unit in $A$, and $e$ a labelled complete set of $r$
orthogonal idempotents in $A$, all commuting with $a$.
In particular, $I^{\circ,\ss}$ preserves the filtration of $K(\MM)$ by split
central rank. 

The composition $I^{\circ,\ss}\circ E_r$ is divisible by $(q-1)^r$:

\begin{prop}
For every $k\geq0$, there exists a $K(\DM)$-linear operator 
$$\tilde I_k^{\circ,\ss}: K(\MM)\longrightarrow  K(\MM)\,,$$
such that 
$$I^{\circ,\ss}\circ E_k=(q-1)^k\, \tilde I_k^{\circ,\ss}\,.$$
\end{prop}
\begin{proof}
Let $X$ be an algebroid.  The algebra $A_{E_k X}$ is endowed with a
tautological complete set of orthogonal central idempotents, and so we
can apply the construction of Proposition~\ref{bundlequack}, to obtain
a principal $\Gm^k$-bundle of algebroids 
$$A_{E_k X}^{\times,\ss}\longrightarrow \tilde A_{E_k X}^{\times,\ss}\,.$$
The assignment
$$[X]\longmapsto [\tilde A_{E_k X}^{\times,\ss}]$$
extends to a well-defined $K(\DM)$-linear operator $
K(\MM)\to K(\MM)$, which we shall denote by $\tilde
I_k^{\circ,\ss}$.  

We record that for an algebroid $X$, the stack  $\tilde I^\ss_k X$
is the stack of $(k+2)$-tuples $(x,e_1,\ldots,e_k,[a])$, where $x$ is an
object of $X$, and $e_1,\ldots,e_k$ form a complete set of orthogonal
idempotents in $A_x$, and $[a]$ is an equivalence class of
semi-simple units in $A_x^{e_1,\ldots,e_r}$, where $a\sim \sum_{i=1}^r \lambda_i
e_i\,a$, for $\lambda_i\in \O_X|_x$. 

The equation
$$I^{\circ,\ss}\circ E_k=(q-1)^k\,\tilde I^{\circ,\ss}_k\,,$$
follows from the fact that 
$$[A_{E_k X}^{\times,\ss}]=(q-1)^k[\tilde A_{E_k X}^{\times,\ss}]\,,$$
which holds because of the bundle relations in $K(\MM)$.
\end{proof}

\begin{cor}
The map which $I^{\circ,\ss}$ induces on the subquotient $ K^{\geq
  k}/ K^{>k}(\MM)$ is divisible by $(q-1)^k$:
$$I^{\circ,\ss}|_{ K^{\geq k}/ K^{>k}(\MM)}=\frac{1}{k!}(q-1)^k\,\tilde
I^{\circ,\ss}_k|_{ K^{\geq k}/ K^{>k}(\MM)}\,.$$
\end{cor}
\begin{proof}
This is because on $ K^{\geq k}/ K^{>k}(\MM)$, the operator $E_k$ acts as
multiplication by $k!$.
\end{proof}

We will use as scalars the localization of $\qq[q]$ at the maximal
ideal $(q-1)$, denoted $\qq[q]_{(q-1)}$, thus inverting all rational polynomials
in $q$, which do not vanish at $q=1$.  We extend scalars on $K(\MM)$
as well, and consider
$$K(\MM)_{(q-1)}= \qq[q]_{(q-1)}\cdot K(\MM)\subset
K(\MM)(q)=K(\MM)\otimes_{\qq[q]}\qq(q)\,.$$
Note that this definition ensures that $K(\MM)_{(q-1)}$ is
$(q-1)$-torsion free. The direct sum decomposition (\ref{idret})
extends to $K(\MM)_{(q-1)}$, and the operator $I^{\circ,\ss}$ extends to a
$\qq[q]_{(q-1)}$-linear operator  
$$I^{\circ,\ss}: K(\MM)_{(q-1)}\longrightarrow  K(\MM)_{(q-1)}\,.$$

For a partition $\lambda\vdash n$, we define 
$$\Q_\lambda=\prod_{i\in\lambda}(q^i-1)\,.$$
This is a polynomial in $q$, of degree $n$, which vanishes to order
$|\lambda|$ at $q=1$. 
We also define
$$\tilde\Q_\lambda=k!\prod_{i\in\lambda}\frac{q^i-1}{q-1}\,.$$
This is a polynomial in $q$, which is invertible in $\qq[q]_{(q-1)}$. 

\begin{thm}
The operator 
$$\tilde I^{\circ,\ss}_k: K^{\geq k}/ K^{>k}(\MM)_{(q-1)}\longrightarrow
 K^{\geq k}/ K^{>k}(\MM)_{(q-1)}$$
is diagonalizable. Its eigenvalue spectrum consists of all $\tilde
\Q_\lambda$, for partitions $\lambda$ of length $|\lambda|=k$. 
\end{thm}
\begin{proof}
We will fix $k$, and work throughout in the subquotient $ K^{\geq
  k}/ K^{>k}(\MM)_{(q-1)}$, restricting all operators tacitly to this
subquotient.  Note that, as a $\qq[q]_{(q-1)}$-module, $K^{\geq
  k}/K^{>k}(\MM)_{(q-1)}$ is isomorphic to $K^k(\MM)_{(q-1)}$, and is
hence $(q-1)$-torsion free. 

We order partitions of length $k$ by {\em divisibility}. If $\lambda$
and $\mu$ are partitions with $|\lambda|=k$ and $|\mu|=k$, we write
$\lambda\mid\mu$, if there exists a permutation $\sigma$ of $\ul k$,
such that $\lambda_i\mid \mu_{\sigma(i)}$, for all
$i=1,\ldots,k$. This is a partial ordering on the partitions of length
$k$.  We write
\begin{equation}\label{ordiv}
K^{\geq\lambda}(\MM)_{(q-1)}
\end{equation}
for the $\qq[q]_{(q-1)}$-subspace generated by clear stack functions
of central type divisible by $\lambda$.

We will prove
\begin{items}
\item the operator $\tilde I^{\circ,\ss}_k$ preserves the filtration
  (\ref{ordiv}) by
  divisibility of partitions,
\item on the quotient $\tilde K^{\geq\lambda}(\MM)_{(q-1)}/\tilde
  K^{>\lambda}(\MM)_{(q-1)}$, 
  the operator $\tilde I^{\circ,\ss}_k$ acts as multiplication by $\tilde \Q_\lambda$. 
\item the operator $\tilde I^{\circ,\ss}_k$ is locally finite.
\end{items}
These facts will imply the claims concerning diagonalizability of $\tilde
I^{\circ,\ss}_k$. This is because for a lower triangular matrix
with distinct diagonal entries over a discrete valuation ring to be
diagonalizable, it suffices that the differences between the diagonal
entries are units. The latter condition is 
satisfied, because if $\lambda\mid\mu$, then $\tilde
Q_{\mu}-\tilde Q_\lambda$ does not vanish at $q=1$. (This argument
does not apply directly, because our eigenvalues are not linearly
ordered, but only partially.  Nevertheless, the conclusion remains
true in this larger generality.)\comment{Can we say this better? or with more
  detail?}

Let us fix a partition $\lambda$ of length $k$, and consider a clear
stack function $X\to \MM$ of central type $\lambda$, with
algebra $A\to X$. Abbreviate the
induced element of $ K^{\geq k}/ K^{>k}(\MM)$ by $[X]$. Denote the
central rank of $X$ by $n$, so that $\lambda\vdash n$. As $A$ has
$k$ central idempotents, $E_kX\to X$ has $k!$ canonical sections, each 
given by a labelling $\sigma$ of these $k$ idempotents. Denote the
images of these sections by $\{X_\sigma\}$. By the proof of
Theorem~\ref{Ethm}, the algebroid $E_kX$ can be stratified as
$$E_kX=\bigsqcup_{\sigma} X_\sigma\sqcup\bigsqcup_\tau Y_\tau\,,$$
where the $Y_\tau$ are clear algebroids of split central rank larger
than $k$.  The part of $\tilde I_k^{\circ,\ss} X$ lying over $Y_\tau$ then
also has split central rank larger than $k$.  Hence, when calculating
$\tilde I_k^{\circ,\ss}[X]$, we can discard all $Y_\tau$.  Every $X_\sigma$ is
isomorphic to $X$, and so we will fix a labelling $\sigma$, and
replace $X_\sigma$ by $X$ in the following arguments, remembering to
multiply the final result by $k!$. 

We need to consider $A^{\times,\ss}$  and its quotient $\tilde
A^{\times,\ss}$.
We write 
\begin{equation}\label{needtocons}
A^{\times,\ss}=Z^{\times,\ss}\sqcup NZ^{\times,\ss}\,,
\end{equation}
where
$Z\subset A$ is the center of $A$ (which is a strict subbundle and
hence a closed substack) and
$NZ$ is its complement.  We start by examining 
$Z^{\times,\ss}$, and its quotient $\tilde
Z^{\times,\ss}=Z^{\times,\ss}/\Gm^k$. Note that $Z^{\times,\ss}$ and
$\tilde Z^{\times,\ss}$ are pullbacks from the coarse Deligne-Mumford
stack $\ol X$ of $X$, and hence are inert over $X$, and their
algebroid structures are hence the canonical algebroid
structures as inert $X$-stacks.  

As in the 
proof of Theorem~\ref{Ethm}, let $\tilde X\to X$ be a connected Galois
cover with Galois group $\Gamma$, acting on the set $\ul n$, such that 
\begin{equation*}
\tilde X\times_\Gamma\ul n\longiso PZE(A)\,.
\end{equation*}
We get  induced isomorphisms
$$\tilde X\times_\Gamma \aaa^n\longiso \pi_\ast \O_{PZE(A)}\,,$$
and
$$\tilde X\times_\Gamma \Gm^n\longiso (\pi_\ast
\O_{PZE(A)})^\times\,.$$
By Proposition~\ref{ssc}, we have a surjective closed immersion
$$(\pi_\ast\O_{PZE(A)})^{\times,\strat}\longrightarrow
Z^{\times,\ss}\,.$$
It follows that we have a surjective closed immersion
$$(\tilde X\times_\Gamma\Gm^n)^\strat\longrightarrow Z^{\times,\ss}\,,$$
and by passing to the quotient another surjective closed immersion
$$(\tilde X\times_\Gamma \Gm^n/\Gm^k)^\strat\longrightarrow
\tilde Z^{\times,\ss}\,.$$
So in $ K^{\geq k}/ K^{>k}(\MM)$, we can replace $\tilde Z^{\times,\ss}$ by
$\tilde X\times_\Gamma \Gm^n/\Gm^k$. 

Our labelling $\sigma$ of the central idempotents in $A_X$ corresponds
to a labelling of the orbits of $\Gamma$ on $\ul n$, and a labelling
of the parts of $\lambda\vdash n$. Let us denote
these orbits by $I_1,\ldots,I_k$, and $\lambda$ by
$(\lambda_1,\ldots,\lambda_k)$, such that $\lambda_i=|I_i|$, for
$i=1,\ldots,k$.  We write $\pp(\aaa^\lambda)$ for 
the product of projective spaces
$\pp(\aaa^{\lambda_1})\times\ldots\times\pp(\aaa^{\lambda_k})$. Moreover,  for a sequence
of subsets $J_i\subset I_i$ we write
$\pp^\ast(J_1,\ldots,J_k)\subset\pp(\aaa^\lambda)$ for the locally
closed subspace defined by the entries in $J_1\cup\ldots\cup J_k$ being
non-zero, and all others to be zero.  We have
$$\pp(\aaa^\lambda)=\bigsqcup_{\stack{(J_1,\ldots,J_k)\in}{
    \P(I_1)\times\ldots\times\P(I_k)}}\pp^\ast(J_1,\ldots,J_k)\,,$$ 
where the disjoint union is over all sequences of subsets $J_i\subset
I_i$. 

The group $\Gamma$ acts linearly on $\pp(\aaa^\lambda)$, respecting
this stratification (although not the individual strata),
and we have
$$\tilde X\times_\Gamma\Gm^n/\Gm^k=\tilde
X\times_\Gamma\pp^\ast(I_1,\ldots,I_k)\,.$$
Moreover,
\begin{align*} 
\tilde X\times_\Gamma \pp(\aaa^\lambda)
&=\tilde
X\times_\Gamma\bigsqcup_{\P(I_1)\times\ldots\times\P(I_k)}
  \pp^\ast(J_1,\ldots,J_k)\\ 
&=\bigsqcup_{\P(I_1)\times\ldots\times\P(I_k)/\Gamma}
\tilde
X\times_{\Stab(J_1,\ldots,J_n)}\pp^\ast(J_1,\ldots,J_k)\\
&=\tilde X\times_\Gamma\pp^\ast(I_1,\ldots,I_k)\,\sqcup\\
&\phantom{mmmm}
\bigsqcup_{\stack{\P(I_1)\times\ldots\times\P(I_k)/\Gamma}{\Stab(J_1,\ldots,J_n)\subsetneq\Gamma}}
\tilde
X\times_{\Stab(J_1,\ldots,J_n)}\pp^\ast(J_1,\ldots,J_k)\,.
\end{align*}
Every subgroup $\Gamma'\subsetneq\Gamma$, which is the stabilizer of a
sequence $(J_1,\ldots,J_k)$, has more than $k$ orbits on $\ul n$.  As
in the proof of Theorem~\ref{Ethm}, this implies that $\tilde
X/{\Gamma'}$ is in $ K^{>k}(\MM)$.  The same is then true for $\tilde
X\times_{\Gamma'}\pp^\ast(J_1,\ldots,J_k)$, as the projection
$\tilde X\times_{\Gamma'}\pp^\ast(J_1,\ldots,J_k)\to \tilde
X/{\Gamma'}$ is inert (being the
pullback of a corresponding morphism of course Deligne-Mumford stacks).  We
deduce that in $ K^{\geq k}/ K^{>k}(\MM)$, we have
\begin{align}
[\tilde Z^{\times,\ss}]&=[\tilde X\times_\Gamma\Gm^n/\Gm^k]=[\tilde
X\times_\Gamma\pp^\ast(I_1,\ldots,I_k)]
=[\tilde X\times_\Gamma\pp(\aaa^\lambda)]\nonumber\\
&=[\pp(\aaa^\lambda)]\,[X]=\frac{1}{k!}\tilde\Q_\lambda(q)\,[X]\,.\label{easyformula}
\end{align}
In the last step, we used the bundle relations in $\tilde K(\MM)$. 
The bundle $\tilde
X\times_\Gamma\pp(\aaa^\lambda)$ is a product of  projective bundles
associated to vector 
bundles, whose structure groups are special (as they are general
linear groups).

Now consider a locally closed embedding $Y\hookrightarrow
NZ^{\times,ss}/\Gm^k$, such that $Y$ is a clear algebroid.
Over $Y$, we then consider the inclusion of commutative algebra bundles 
$ZA_X|_Y\hookrightarrow ZA_Y$. By Proposition~\ref{rankineq}, the split rank of
$ZA_Y$ (which is the split central rank of $Y$), is at least as large
as the split rank of $ZA_X|_Y$ which, in turn, is at least as large as
the split rank of $ZA_X$, which is $k$. Since we are working modulo
$ K^{>k}(\MM)$, we may assume that the 
split central rank of $Y$ is $k$, and hence that the split rank of $ZA_Y$
and of $ZA_X|_Y$ are both equal to $k$. Consider the correspondence
$Q$ which we used in the proof of 
Proposition~\ref{rankineq}:
$$\xymatrix{
Q{\phantom{.}}\ar@{->>}[d]\ar@{^{(}->}[r] & PZE(A_Y)\\
PZE(A_X|_Y)\rlap{\,.}}$$

All three stacks in this diagram are representable finite \'etale
covers of $Y$. By assumption, both $PZE(A_Y)$ and $PZE(A_X|_Y)$ have
$k$ connected components.  This implies that the horizontal inclusion
in the diagram is an isomorphism, and that we have a surjective
representable finite \'etale cover
\begin{equation}\label{fecov}
\xymatrix{PZE(A_Y)\ar@{->>}[r]&PZE(A_X|_Y)}\,.
\end{equation}
Since $PZE(A_X|_Y)$ and $PZE(A_X)$ have the same number of components,
the degrees of these components are equal as well, which means that the
central type of $A_X|_Y$ is equal to the central type of $A_X$, which
is $\lambda$. The existence of (\ref{fecov}) then implies that
$\lambda$ divides the central type of $A_Y$. 

The surjection (\ref{fecov}) is not an isomorphism, because otherwise,
by Proposition~\ref{rankineq}, we would have a surjection
$ZA^{\times,\ss}_Y\to ZA^{\times,\ss}_X|_Y$, 
but this would force the tautological section class of $ZA^{\times,\ss}_Y$, given by
the structure map $Y\to A_X^{\times,\ss}/\Gm^k$ to be central in $A_X$
(at least pointwise), which it is not.

This shows that the central type of $Y$ {\em strictly }divides
$\lambda$. 

We have thus completed the proof of (i), and (ii), above. For the
local finiteness  of $\tilde I_k^{\circ,\ss}$, proceed as in the proof of
Corollary~\ref{simult}. Every time we apply $\tilde I^{\circ,\ss}_k$, we
produce only clear algebroids whose central type is a multiple of
$\lambda$, but as we can bound the central rank by the rank, which
does not increase by applying $\tilde I^{\circ,\ss}_k$, after
finitely many steps,  this process stops. 
\end{proof}

\begin{cor}
The operator $I^{\circ,\ss}: K(\MM)_{(q-1)}\to 
K(\MM)_{(q-1)}$ is diagonalizable. Its 
eigenvalue spectrum consists of the  $\Q_{\lambda}\in\qq[q]$, for all
partitions $\lambda$. Denote the eigenspace corresponding to the
eigenvalue $\Q_{\lambda}$ by $ K^\lambda(\MM)_{(q-1)}$. We have
$$ K^k(\MM)_{(q-1)}=\bigoplus_{|\lambda|=k}  K^\lambda(\MM)_{(q-1)}\,.$$
\end{cor}

\begin{ex}
Consider, as in Example~\ref{bgl2ex}, the stack function of rank 2
vector bundles. It defines an element $[BGL_2]$ of $ K^{\geq
  1}/ K^{>1}(\mf{Vect})$, which we are going to decompose into its
eigencomponents with respect to the operator $\tilde I^{\circ,\ss}_1$. 

The stack function
$[BGL_2]$ is clear, its central rank is~1. The decomposition
(\ref{needtocons}) is given in this case as
$$
I^{\circ,\ss}_{BGL_2}=\Delta/GL_2\sqcup T^\ast/N\,,
$$
where $\Delta$ is the central torus of $GL_2$, and
$T^\ast=T\setminus\Delta$. Also, $N$ is the normalizer of $T$ in
$GL_2$. 
We get the corresponding decomposition
\begin{align*}
\tilde I_1^{\circ,\ss} BGL_2&=(\Delta/\Gm)/GL_2\sqcup (T^\ast/\Gm)/N\\
&=BGL_2\sqcup \tilde T^\ast/N\,,
\end{align*}
where $\tilde T^\ast=T^\ast/\Gm$, and we have 
$$\tilde I_1^{\circ,\ss}[BGL_2]=[BGL_2]+[\tilde T^\ast/N]\,.$$
Note that $\tilde T^\ast/N$ is not
a strict algebroid.  In fact, let $T'\subset T^\ast$ be the closed subscheme
consisting of elements of trace zero, and write 
$T^\ast=T'\sqcup T^{\ast\ast}$. Then (at least if $2\in R^\ast$)
we have $T'/\Gm=\spec R$, and $[\tilde T^\ast/N]=[BN]+[\tilde
  T^{\ast\ast}/N]$, and $BN$ is not a strict algebroid, as $N$ is not
connected. But $\tilde T^\ast/N$ is a clear algebroid.\comment{as
  long as that allows for discontinuity of $I_X$} 
Its connected inertia stack is $\tilde T^\ast\times T/N$, and its
associated coarse Deligne-Mumford stack is $\tilde T^\ast/\zz_2$,
which is only generically a scheme. The central rank of 
$\tilde T^\ast/N$ is 2, and the split central rank 1.

Now we consider $\tilde I_1^{\circ,\ss}[\tilde T^\ast/N]$.  We start by noting
that  all connected inertia of $\tilde T^\ast/N$ is central. Hence, modulo
$ K^{>1}(\mf{Vect})$, we have
$$\tilde I_1^{\circ,\ss}[\tilde T^\ast/N]=(q+1)[\tilde T^\ast/N]\,,$$
by (\ref{easyformula}). 

We see that $\qq[q]_{(q-1)}[BGL_2]+\qq[q]_{(q-1)}[\tilde T^\ast/N]$ is invariant
under $\tilde I^{\circ,\ss}_1$, and the matrix of 
$\tilde I^{\circ,\ss}_1$ on this subspace is 
$$\begin{pmatrix}
1 & 0 \\
1 & q+1
\end{pmatrix}
$$
This matrix is lower triangular, and the differences between the
scalars on the diagonal are all invertible in $\qq[q]_{(q-1)}$.
Therefore it is diagonalizable over $\qq[q]_{(q-1)}$.
Diagonalizing, we get the following eigenvectors modulo
$ K^{>1}(\mf{Vect})_{(q-1)}$:
\begin{items}
\item $\ol v_{(1)}=[BGL_2]-\frac{1}{q}[\tilde T^\ast/N]$,
\item $\ol v_{(2)}=\frac{1}{q}[\tilde T^\ast/N]$.
\end{items}
To get the actual eigenvectors, we project into
$ K^1(\mf{Vect})_{(q-1)}$.  We have
\begin{items}
\item $\pi_1[BGL_2]=[BGL_2]-\frac{1}{2}[BT]$,
\item $\pi_1[\tilde T^\ast/N]=[\tilde T^\ast/N]-\frac{1}{2}(q-2)[BT]$,
\end{items}
and hence
\begin{items}
\item $v_{(1)}=[BGL_2]-\frac{1}{q}[\tilde T^\ast/N]-\frac{1}{q}[BT]$,
\item $v_{(2)}=\frac{1}{q}[\tilde T^\ast/N]-\frac{q-2}{2q}[BT]$.
\end{items}
If we add 
\begin{items}\setcounter{enumi}{2}
\item $v_{(1,1)}=\frac{1}{2}[BGL_2]$,
\end{items}
we get the spectral decomposition 
$[BGL_2]=v_{(1)}+v_{(2)}+v_{(1,1)}$ of $[BGL_2]$,  with respect to the
operator $I^{\circ,\ss}$.  This is, of course, the same as the spectral
decomposition with respect to $I^\ss$, which we computed in the
introduction (after applying the bundle relations). 
\end{ex}

\begin{rmk}
If we are willing to invert $(q-1)$, we can prove the
diagonalizability of $I^{\circ,\ss}=I^\ss$ entirely within the context
of strict algebroids. In fact, we can generalize the  calculation in
the introduction to  accomplish this. 
\end{rmk}

\subsection{Graded structure of multiplication}

We will now assume that $\MM$ admits all direct sums.  Then we can
define a commutative product on $K(\MM)$ by
$$[X\to \MM]\cdot [Y\to\MM]=[X\times Y\to \MM\times
  \MM\stackrel{\oplus}{\longrightarrow} \MM]\,.$$
With this product $K(\MM)$ becomes a commutative $K(\DM)$-algebra
with unit $1=[\spec R\stackrel{0}{\longrightarrow}\MM]$.

\begin{prop}
For $x,y\in K(\MM)$, we have 
$$I^{\circ,\ss}(x\cdot y)=I^{\circ,\ss}(x)\cdot I^{\circ,\ss}(y)\,.$$
\end{prop}
\begin{proof}
This follows immediately from the fact that, for any two algebroids
$X$, $Y$, we have  $A^{\times,\ss}_{X\times
Y}=A^{\times,\ss}_X\times A^{\times,\ss}_Y$, as algebroids over $X\times Y$.
\end{proof}

Denote the disjoint union\comment{is this standard terminology?} of
two partitions $\lambda$ and $\mu$ by $\lambda+\mu$. 

\begin{cor}
We have  $K^{\lambda}(\MM)_{(q-1)}\cdot K^{\mu}(\MM)_{(q-1)}\subset
K^{\lambda+\mu}(\MM)_{(q-1)}$, and hence also $K^k(\MM)_{(q-1)}\cdot
K^\ell(\MM)_{(q-1)}\subset K^{k+\ell}(\MM)_{(q-1)}$. 
\end{cor}

So the $\qq[q]_{(q-1)}$-module
$$K(\MM)_{(q-1)}=\bigoplus_{k\geq0} K^{k}(\MM)_{(q-1)}$$
is a graded $\qq[q]_{(q-1)}$-algebra, with respect to the commutative product
on $K(\MM)_{(q-1)}$.  We will prove next, that this
fact is true for $K(\MM)$ itself.

\begin{prop}
For any $x,y\in K(\MM)$ and any $p\geq0$, we have
$$E_p(x\cdot y)=\sum_{n,m}{p\brack n,m}E_n(x)\cdot
E_m(y)\,.$$
Here ${p\brack n,m}$ is the number of ways the set $\ul p$ can be
written as the union of a subset of order $n$, and a subset of order
$m$. 
\end{prop}
\begin{proof}
Consider stack functions $X\to \MM$ and $Y\to \MM$. Then $E_p(X\times
Y)$ is the stack of pairs $(e,f)$, where  $e=(e_\rho)_{\rho\in \ul p}$
is a complete set of  orthogonal idempotents in $A_X$, and
$f=(f_\rho)_{\rho\in \ul p}$ is a complete set  of  orthogonal
idempotents in $A_Y$, such that for every $\rho=1,\ldots,p$, at least
one of the two idempotents   $e_\rho$, $f_\rho$ is non-zero.

For every pair of strictly monotone maps $\ul n\hookrightarrow \ul p$,
$\ul m\hookrightarrow \ul p$, whose images cover $\ul p$, we get a
morphism of stack functions $E_n(X)\times E_m(Y)\to E_p(X\times Y)$,
by mapping a pair of complete sets of orthogonal idempotents
$(e',f')$, where $e'=(e'_\nu)_{\nu\in \ul n}$ and $f'=(f'_\mu)_{\mu\in
  \ul m}$, to the pair $(e,f)$, defined by 
$$e_\rho=\sum_{\nu\mapsto \rho}e'_\nu\qquad\text{and}\qquad
f_\rho=\sum_{\mu\mapsto \rho}f'_\mu\,.
$$
(As the maps $\ul n\to \ul p$ and $\ul m\to \ul p$ are injective, all
these sums have either zero or one summand.)

Each of the morphisms $E_n(X)\times E_m(X)\to E_p(X\times Y)$ is an
isomorphism onto a locally closed substack, because the locus of
vanishing for an idempotent is closed.
Moreover, the images of these morphisms are disjoint, and from a
cover. There are ${p\brack n,m}$ of them.
\end{proof}

\begin{cor}
If $x\in K^k(\MM)$ and $y\in K^\ell(\MM)$, then $x\cdot y\in
K^{k+\ell}(\MM)$. 
\end{cor}
\begin{proof}
We have
\begin{align*}
\pi_t(x\cdot y)&=
\sum_{p}\binom{t}{p}E_p(x\cdot y)\\
&=\sum_p\binom{t}{p}\sum_{n,m}{p\brack n,m}E_n(x)\cdot E_m(y)\\
&=\sum_{n,m}\Big(\sum_p{p\brack n,m}\binom{t}{p}\Big)E_n(x)\cdot
E_m(y)\\
&=\sum_{n,m}\binom{t}{n}\binom{t}{m}E_n(x)\cdot E_m(x)\\
&=\pi_t(x)\cdot \pi_t(y)\,.
\end{align*}
The step from Line~3 to Line~4 uses Proposition~\ref{combinatorics}, below.
\end{proof}

\subsubsection{A combinatorial lemma}

Let $p\geq0$. 

For a non-negative integer $n$, and a $p$-tuple of non-negative
integers $\lambda=(\lambda_1,\ldots,\lambda_p)$, we define 
\begin{equation}\label{defnbrack}
{n\brack \lambda}={n\brack \lambda_1,\ldots,\lambda_p}
\end{equation}
to be the number of indexed covers of $\ul n$ by subsets
$S_1,\ldots,S_p$ of
cardinalities $\lambda_1,\ldots,\lambda_p$. The non-negative integer
${n\brack\lambda}$ vanishes, unless $\lambda_\rho\leq n$, for all
$\rho=1,\ldots,p$, and $n\leq |\lambda|$, where
$|\lambda|=\sum_{\rho}\lambda_\rho$. 

We have the following useful combinatorial property:
\begin{prop}\label{combinatorics}
For every $p$-tuple of non-negative integers
$\lambda=(\lambda_1,\ldots,\lambda_p)$, we have
$$\binom{t}{\lambda_1}\ldots\binom{t}{\lambda_p}=\sum_n{n\brack\lambda}\binom{t}{n}\,.$$ 
\end{prop}
\begin{proof}
Let $x_1,\ldots,x_p$ be formal variables. We will prove that
$$\sum_{\lambda}\binom{t}{\lambda_1}\ldots\binom{t}{\lambda_p}x_1^{\lambda_1}\ldots
x_p^{\lambda_p}=
\sum_{\lambda}\sum_n{n\brack\lambda}\binom{t}{n}x_1^{\lambda_1}\ldots
x_p^{\lambda_p}\,,$$
by proving that both sides of this equation are equal to 
$$\prod_{i=1}^p
(1+x_i)^t\,.$$
On the one hand, we have
\begin{align*}
\prod_{i=1}^p(1+x_i)^t&=
\prod_{i=1}^p\sum_{n}\binom{t}{n}x_i^n\\
&=\sum_{\lambda_1,\ldots,\lambda_p}\binom{t}{\lambda_1}\ldots\binom{t}{\lambda_p}x_1^{\lambda_1}\ldots 
x_p^{\lambda_p}\,.
\end{align*}
On the other hand, we have
\begin{align*}
\prod_{i=1}^p(1+x_i)^t&=
\Big(1+\prod_{i=1}^p(1+x_i)-1\Big)^t\\
&=\sum_n\binom{t}{n}\Big(\prod_{i=1}^p(1+x_i)-1\Big)^n\\
&=\sum_n\binom{t}{n}\sum_j(-1)^j\binom{n}{j}\prod_{i=1}^p(1+x_i)^{n-j}\\
&=\sum_n\binom{t}{n}\sum_j(-1)^j\binom{n}{j}
\prod_{i=1}^p\sum_\ell\binom{n-j}{\ell}x_i^{\ell}\\ 
&=\sum_n\binom{t}{n}\sum_j(-1)^j\binom{n}{j}
\sum_{\lambda_1,\ldots,\lambda_p}\binom{n-j}{\lambda_i}x_1^{\lambda_1}\ldots
x_p^{\lambda_p}\\ 
&=\sum_n\binom{t}{n}
\sum_{\lambda}\Big(
\sum_j(-1)^j\binom{n}{j}\binom{n-j}{\lambda_i}\Big)x^\lambda\\
&=\sum_n\binom{t}{n}
\sum_{\lambda}{n\brack \lambda_1,\ldots,\lambda_p}x^\lambda\\
&=\sum_\lambda\sum_n{n\brack\lambda}\binom{t}{n}x^\lambda\,.
\end{align*}
Here we have used the obvious inclusion-exclusion property satisfied
by the covering numbers. 
\end{proof}

\section{The order filtration}\label{sec-filo}

\subsubsection{The Hall algebra}

Let $\MM$ be a linear algebraic stack admitting direct sums and direct
summands, i.e., assume that $\MM$ is Karoubian (Remark~\ref{karoubi}).
  To define the Hall product, we need an additional structure
on $\MM$. This is a linear algebraic substack $\MM^{(2)}$ of the stack
of all sequences $M'\to M\to M''$ in $\MM$, such that for every
$R$-scheme $S$, the fibre $\MM^{(2)}(S)$ defines the structure of an
exact category on $\MM(S)$.  The stack $\MM^{(2)}$ comes with a
diagram of morphisms of linear algebraic stacks 
$$\xymatrix{
\MM^{(2)}\rto^b\dto_{a_1\times a_2} & \MM\\
\MM\times\MM\rlap{\,,}}$$
where $a_1,a_2,b:\MM^{(2)}\to\MM$ are the projections of the sequence
$M'\to M\to M''$ onto the objects $M', M'', M$, respectively. We
require further, that  
  the morphism $a_1\times a_2:\MM^{(2)}\to\MM\times\MM$ is of finite type, and
  the morphism $b:\MM^{(2)}\to\MM$ is representable.

We call such an $\MM$ an {\bf exact linear algebraic stack}.

\begin{ex}
The linear stacks $\mf{Coh}_X$, $\mf{Vect}$, and $\mf{Rep}_Q$ of
Examples \ref{coh}, \ref{vect}, and~\ref{repQ} satisfy these 
axioms.  For $\mf{Coh}_X$, see
\cite[Section~4.1]{Bridgeland-Hall-Algebras}.  

In each case, the exact structure is given by all short exact
sequences.  Note that the categories $\MM(S)$ are not abelian, as the
cokernel of a homomorphisms of flat sheaves is not necessarily
flat.
\end{ex}

Throughout the following discussion we fix an exact linear algebraic
stack $\MM$, and let $\AA\to\MM$ be its universal
endomorphism algebra, as in Section~\ref{spectrum}.

We have the following structures on  $K(\MM)$. 
\begin{enumerate}
\item {\bf Module structure}. The action of $K(\DM)$ on
  $K(\MM)$, given by $[Z]\cdot[X\to\MM]=[Z\times X\to X\to\MM]$, which
  turns $K(\MM)$ into a $K(\DM)$-module.
\item  {\bf Multiplication}.  The commutative multiplication given by
  $$[X\to\MM]\cdot[Y\to\MM]=[X\times Y\to
  \MM\times\MM\stackrel{\oplus}{\longrightarrow} \MM]\,.$$
\item {\bf Hall product}. The Hall product of the stack functions
  $[X\to\MM]$ and $[Y\to\MM]$, which is defined by first constructing the
  fibered product
$$\xymatrix{
X\ast Y\dto\rto & \MM^{(2)}\dto^{a_1\times a_2}\\
X\times Y\rto & \MM\times \MM}$$
and then setting 
$$[X\to\MM]\ast[Y\to\MM]=[X\ast Y\longrightarrow
  \MM^{(2)}\stackrel{b}{\longrightarrow}\MM]\,.$$ 
\end{enumerate}

The multiplication is associative and commutative, the Hall product is
associative. The unit with respect to both multiplications is given by 
the $0$-object of $\MM$:
$$1=[\spec R\stackrel{0}{\longrightarrow}\MM]\,.$$

From now on we
will refer to $K(\MM)$ as the {\bf Hall algebra }of $\MM$.

\subsection{Filtered structure of the Hall algebra}

\begin{defn}
For $n\geq0$, we define 
$$K^{\leq n}(\MM)=\ker E_{n+1}=\bigoplus_{k\leq n}K^k(\MM)\,.$$
This is an
ascending filtration on $K(\MM)$, called the {\bf filtration by the
  order of vanishing of inertia at $q=1$}, or simply the {\bf order
  filtration }of $K(\MM)$.
\end{defn}

This is a slight abuse of language, because only the space obtained by
extension of scalars
$K^{\leq n}(\MM)_{(q-1)}$ is the direct
sum of all eigenspaces of $I^{\circ,\ss}$ whose corresponding eigenvalues
$\Q\in\qq[q]$ have order of vanishing at $q=1$ less than or equal to
$n$. 

\begin{thm}\label{main}
Suppose that  $\xi\in K^{\leq
  n}(\MM)$ and $\chi\in K^{\leq m}(\MM)$, then $\xi\ast\eta\in
K^{\leq n+m}(\MM)$. 
Moreover, we have $$\xi\ast \chi\equiv \xi\cdot \chi\mod K^{<n+m}\,.$$
\end{thm}

To prove this theorem we will prove the following lemma.

\begin{lem}\label{mainlemma}
For any two stack functions $\xi,\chi\in K(\MM)$, and for any integer
$p\geq0$, we have 
\begin{equation}\label{mainlem1}
{\frac{1}{p!}E_p\big(\pi_t(\xi)\ast\pi_t(\chi)\big)\equiv
\sum_{i+j=p}\pi_i(\xi)\,\pi_j(\chi)\,t^p\mod
t^{p+1}}\,,\end{equation}
as an equation in $K(\MM)[t]$.
\end{lem}

Before proving the lemma, let us indicate how the lemma implies the
theorem.  For this, suppose that  $\xi\in K^{\leq k}(\MM)$ and $\chi\in
K^{\leq \ell}(\MM)$.  Then  the 
  degree of $\pi_t(\xi)$ in $t$ is at most $k$ and the degree of
  $\pi_t(\chi)$ is at most $\ell$. So the degree of
  $E_p\big(\pi_t(\xi)\ast\pi_t(\chi)\big)$ is at most $k+\ell$. 
So we see that if  $p>k+\ell$, then
$E_p\big(\pi_t(\xi)\ast\pi_t(\chi)\big)=0$, which implies that
$\xi\ast\chi\in K^{\leq k+\ell}(\MM)$, by Corollary~\ref{kercor}.

Now set  $p=k+\ell$. The left hand side of (\ref{mainlem1}) has degree
at most $k+\ell$, the the right hand side has degree exactly
$k+\ell$, which implies that both sides are homogeneous of degree
$k+\ell$, and we have
$$\frac{1}{(k+\ell)!}E_{k+\ell}\big(\pi_t(\xi)\ast\pi_t(\chi)\big)=
\pi_k(\xi)\,\pi_\ell(\chi)\,t^{k+\ell}\,.$$
Now notice that if $x\in K^{\leq n}(\MM)$, we have
$\pi_n(x)=\frac{1}{n!}E_n(x)$. Hence we can rewrite our equation as
$$\pi_{k+\ell}\big(\pi_t(\xi)\ast\pi_t(\chi)\big)=
\pi_k(\xi)t^k\,\pi_\ell(\chi)t^{\ell}\,.$$  
This proves the theorem.

\subsubsection{Analysis of $E_p(E_n\ast E_m)$}

Suppose $\xi=(X\to\MM)$ and $\chi=(Y\to\MM)$ are stack functions. 
The stack function $\xi\ast\chi$ is defined by the cartesian diagram:
$$\xymatrix{
X\ast Y\dto\rto & \MM^{(2)}\dto\rto & \MM\\
X\times Y\rto & \MM\times \MM}$$
Explicitly, $X\ast Y$ is the stack of triples $(x,M,y)$, 
\begin{equation}\label{Udia}\vcenter{
\xymatrix{
x\dto & & y\dto\\
M'\rto & M\rto & M''}
}\end{equation}
where $x$
and $y$ are objects of $X$ and $Y$, respectively, $M$ is an object
of $\MM^{(2)}$, i.e., a short exact sequence $M'\to M\to M''$ of
objects in $\MM$, and $x\to M'$ and $y\to M''$ are isomorphisms from
the images of $x$ and $y$ in $\MM$ to $M'$ and $M''$,
respectively. (We omit these isomorphisms from the triple to simplify
the notation.)

The stack function $E_n(\xi)\ast E_m(\chi)$ is defined by the enlarged
diagram:
$$\xymatrix{
E_n(X)\ast E_m(Y)\rto\dto &
X\ast Y\dto\rto & \MM^{(2)}\dto\rto & \MM\\
E_n(X)\times E_m(Y)\rto & 
X\times Y\rto & \MM\times \MM}$$
Explicitly, $E_n(X)\ast E_m(Y)$ is the stack of 5-tuples
$\big(x,(e_\nu),M,y,(f_\mu)\big)$, where $(x,M,y)$ represents a
diagram~(\ref{Udia}), and $(e_\nu)=(e_1,\ldots,e_n)$ is a complete set
of non-zero orthogonal idempotents in $A(x)$, and $(f_\mu)=(f_1,\ldots,f_m)$ is
a complete set of non-zero orthogonal idempotents in $A(y)$.

Finally, the stack
$E_p\big(E_n(X)\ast E_m(Y)\big)$ is 
the stack of objects of $E_n(X)\ast E_m(Y)$, endowed with a
complete set of $p$ non-zero labelled idempotents. Explicitly, it
consists of 6-tuples
\begin{equation}\label{6tup}
\big(x,(e_{\nu,\rho}),M,(g_\rho),y,(f_{\mu,\rho})\big)\,,
\end{equation}
where $(x,M,y)$ is as in (\ref{Udia}), and $(g_\rho)_{\rho\in \ul p}$
is a complete set of non-zero orthogonal idempotent endomorphisms of the short
exact sequence $M'\to M\to M''$. Moreover, 
$(e_{\rho,\nu})_{\rho\in\ul p,\nu\in\ul n}$ is a $pn$-tuple of
orthogonal idempotents in $A(x)$, and
$(f_{\rho,\mu})_{\rho\in\ul p,\mu\in \ul m}$ is a $pm$-tuple of
orthogonal idempotents in $A(y)$, such that for every $\rho=1,\ldots,
p$ we have $\sum_{\nu=1}^n e_{\rho,\nu}=g_\rho|_{M'}$ and
$\sum_{\mu=1}^m f_{\rho,\mu}=g_\rho|_{M''}$. Finally, we require 
for all $\nu=1,\ldots, n$ that $e_\nu=\sum_{\rho=1}^p e_{\rho,\nu}\not=0$ and
for all $\mu=1,\ldots,m$ that $f_\mu=\sum_{\rho=1}^p f_{\rho,\mu}\not=0$. 

\subsubsection{Decomposing $E_p(E_n\ast E_m)$}

Given $p$-tuples of non-negative integers
$\phi=(\phi_1,\ldots,\phi_p)$ and $\psi=(\psi_1\ldots,\psi_p)$, we
define a new stack function $(X\ast Y)_{\phi,\psi}\to\MM$, denoted
$(\xi\ast\chi)_{\phi,\psi}$, as follows. 

Let $(X\ast Y)_{\phi,\psi}$  be the algebraic stack of 6-tuples
\begin{equation}\label{tuple}
\big(x,(e_{\rho}),M,(g_\rho),
y,(f_{\rho})\big)\,,
\end{equation}
where  $(x,M,y)$ is as in (\ref{Udia}), and
$(g_\rho)_{\rho=1,\ldots,p}$ is a complete set of non-zero orthogonal
idempotent endomorphisms of the short exact $M$. 
Moreover, for every $\rho=1,\ldots,p$, we require that
$e_\rho=(e_1,\ldots,e_{\phi_\rho})$ and
$f_\rho=(f_1,\ldots,f_{\psi_\rho})$ are 
families of non-zero orthogonal idempotents for $x$ and $y$,
respectively, such that 
for all $\rho=1,\ldots,p$,
\begin{equation}\label{gef}
g_\rho|_{M'}=\sum_{\omega=1}^{\phi_\rho}e_\omega\qquad\text{ and}\qquad
g_\rho|_{M''}=\sum_{\eta=1}^{\psi_\rho}f_\eta\,.
\end{equation}
It follows that the union of $e_1,\ldots,e_p$ is a complete set of
orthogonal idempotents for $x$, and the union of $f_1,\ldots,f_p$ is a complete set
of orthogonal idempotents for $y$. 

There is a natural algebroid structure on $(X\ast Y)_{\phi,\psi}$. The
morphism to $\MM$ given by mapping the 
$6$-tuple (\ref{tuple}) to the middle object $b(M)$ of the short
exact sequence $M$, makes $(X\ast Y)_{\phi,\psi}$ into a stack
function.

Note that if for some $\rho=1,\ldots,p$ both integers $\phi_\rho$ and
$\psi_\rho$ vanish, then $(X\ast Y)_{\phi,\psi}=\varnothing$, because all
$g_\rho$ are required to be non-zero. 

Let us write $|\phi|=\sum_\rho\phi_\rho$ and
$|\psi|=\sum_\rho\psi_\rho$. Let us assume that for every
$\rho=1,\ldots,p$, at least one of the two integers $\phi_\rho$,
$\psi_\rho$ is non-zero. Then  we have a morphism
\begin{equation}\label{mor}
E_{|\phi|}(X)\times E_{|\psi|}(Y)\longrightarrow (X\ast Y)_{\phi,\psi}
\end{equation}
which maps a quadruple $\big(x,(e_\omega),y,(f_\eta)\big)$ to the
6-tuple  (\ref{tuple}) where $M=M'\oplus M''$, with $M'$ denoting  the image
of $x$ in $\MM$, and $M''$ the image of $y$ in $\MM$.  To
define~(\ref{mor}), we break up the complete family of orthogonal idempotents
$e_1,\ldots,e_{|\phi|}$ for $x$ into $p$ subfamilies, where the
$\rho$-th subfamily has $\phi_\rho$ members.  Similarly, we break up
$f_1,\ldots,f_{|\psi|}$ into $p$ subfamilies whose sizes are
$\psi_1,\ldots,\psi_p$. Then the family of 
idempotents $(g_\rho)$ on $M$ is defined by formulas~(\ref{gef}). 
Note that we need to make our assumption on the $p$-tuples $\phi$,
$\psi$, in order for every
family member $g_\rho$ to be  non-zero.

\begin{lem}\label{exact}
If for every $\rho=1,\ldots,p$ exactly one of the two integers
$\phi_\rho$, $\psi_\rho$ is non-zero,
(\ref{mor}) is an isomorphism. Hence we have the equality
$$(\xi\ast\chi)_{\phi,\psi}=E_{|\phi|}(\xi)E_{|\psi|}(\chi)$$
for stack functions.
\end{lem}
\begin{proof}\comment{we might want to check that all language is
    appropriate for exact categories, and that everything works in
    exact categories.}
Given an object (\ref{tuple}) of $(X\ast Y)_{\phi,\psi}$, the short
exact sequence $M$ is split into a direct sum of $p$ short exact
sequences. Each one of these sequences is canonically split, because
either the subobject or the quotient object vanishes, by the
assumption on $\phi$ and $\psi$. Therefore the sequence $M$ is split,
canonically, too. 
\end{proof}

\newcommand{\hooklongrightarrow}{\lhook\joinrel\longrightarrow}

Now suppose given strictly monotone maps
$$\Phi_\rho:\ul{\phi_\rho}\hooklongrightarrow\ul n\qquad\text{and}\qquad
\Psi_\rho:\ul{\psi_\rho}\hooklongrightarrow\ul m\,,$$
for all $\rho=1,\ldots,p$, such that the images of the $\Phi_\rho$
cover $\ul n$, and the images of the $\Psi_\rho$ cover $\ul m$. 
The choice of these  injections determines a morphism of algebraic stacks
\begin{equation}\label{morfu}
(X\ast Y )_{\phi,\psi}\longrightarrow 
E_p\big(E_n(X)\ast E_m(Y)\big)\,,
\end{equation}
by mapping the $6$-tuple (\ref{tuple}) to the 6-tuple
(\ref{6tup}) by defining 
$$e_{\nu,\rho}=\sum_{\Phi_\rho(\omega)=\nu}(e_\rho)_{\omega}
\qquad\text{and}\qquad
f_{\mu,\rho}=\sum_{\Psi(\eta)=\mu}(f_\rho)_\eta\,.
$$
By our assumptions, these sums are either empty or consist of a single
summand, so the $e_{\nu,\rho}$ and the $f_{\mu,\rho}$ are obtained
from the $(e_\rho)_\omega$ and the $(f_\rho)_{\eta}$ essentially by
relabelling.

Note that the requirements $\bigcup_\rho\Phi_\rho(\ul{\phi_\rho})=\ul n$
and $\bigcup_\rho\Psi_\rho(\ul{\psi_\rho})=\ul m$ are needed to assure
that $\sum_\rho e_{\nu,\rho}$ and $\sum_{\rho}f_{\mu,\rho}$  are
non-zero, for all $\nu=1,\ldots, n$, and 
$\mu=1,\ldots m$. 

\begin{lem}
The morphism (\ref{morfu}) gives rise to a morphism of stack
functions $(\xi\ast\chi)_{\phi,\psi}\to E_p\big(E_n(\xi)\ast
E_m(\chi)\big)$, which is both an open and a closed immersion. 

If we
change 
any of $\phi$, $\psi$, or $\Phi$, $\Psi$,
we get a morphism with disjoint
image.  The images of all morphisms (\ref{morfu}) cover
$E_p\big(E_n(X)\ast E_m(Y)\big)$. 
\end{lem}
\begin{proof}
This follows from the fact that the source and target of (\ref{morfu})
only differ in the way the idempotents in $A_x$ and $A_y$ are indexed.
\end{proof}

\begin{cor}
Using the notation introduced in (\ref{defnbrack}), we have 
the following equation in  $K(\MM)$:
$$
E_p\big(E_n(\xi)\ast E_m(\chi)\big)=
\sum_{\phi,\psi}\sum_{\Phi,\Psi}
(\xi\ast \eta)_{\phi,\psi}
=
\sum_{\phi,\psi}
{n\brack \phi}{m\brack \psi}
(\xi\ast \eta)_{\phi,\psi}\,,
$$
where $\phi$, $\psi$ run over all $p$-tuples of non-negative integers. 
\end{cor}

For example, consider $\chi=1$, and $m=0$. If any of the $\psi_\rho$
is non-zero, $(X\ast Y)_{\phi,\psi}$ is empty. Hence
$$E_pE_n(\xi)=\sum_{\phi_1,\ldots,\phi_p>0}{n\brack \phi}E_{|\phi|}(\xi)\,,$$
where the sum is over all $p$-tuples of positive integers.

\subsubsection{Proof of the main lemma}

Using Proposition~\ref{combinatorics}, we can now calculate as follows:
\begin{align}
E_p\big(\pi_t(\xi)\ast\pi_t(\chi)\big) &=\nonumber
E_p\Big(
\sum_{n}\binom{t}{n}E_n(\xi)
\ast
\sum_m\binom{t}{m}E_m(\chi)\Big)\\
&=\nonumber
\sum_{n,m}\binom{t}{n}\binom{t}{m}\sum_{\phi,\psi}{n\brack\phi}{m\brack\psi}
(\xi\ast\chi)_{\phi,\psi} \\
&=\nonumber
\sum_{\phi,\psi}
\Big(\sum_n\binom{t}{n}{n\brack\phi}\Big)
\Big(\sum_m\binom{t}{m}{m\brack\psi}\Big)
(\xi\ast\chi)_{\phi,\psi} \\
&=
\sum_{\phi,\psi}
\binom{t}{\phi_1}\ldots\binom{t}{\phi_p}
\binom{t}{\psi_1}\ldots\binom{t}{\psi_p}
(\xi\ast\chi)_{\phi,\psi}\,.\label{m16}
\end{align}
For example, if $\chi=1$, we get
\begin{equation}\label{t17}
E_p\pi_t(\xi)=\sum_{\phi_1,\ldots,\phi_p>0}\binom{t}{\phi_1}\ldots\binom{t}{\phi_p}
E_{|\phi|}(\xi)\,.\end{equation}

The lowest order term in (\ref{m16}) has degree $p$, since
for $(\xi\ast\chi)_{\phi,\psi}$ not to vanish, we need for every
$\rho=1,\ldots,p$ at least one of $\phi_\rho$, $\psi_\rho$ to be
non-zero. 

Modulo $(t^{p+1})$, only terms corresponding to  pairs $(\phi,\psi)$,
with the property that for every $\rho=1,\ldots,p$ exactly one of
$\phi_\rho$, $\psi_\rho$ is non-zero, contribute to~(\ref{m16}). 
These are exactly the terms to which Lemma~\ref{exact} applies, and we
deduce, that modulo $t^{p+1}$, we have:
$$E_p\big(\pi_t(\xi)\ast\pi_t(\chi)\big) 
\equiv\sum_{\phi,\psi}
\binom{t}{\phi_1}\ldots\binom{t}{\phi_p}
\binom{t}{\psi_1}\ldots\binom{t}{\psi_p}
E_{|\phi|}(\xi)\,E_{|\psi|}(\chi)\,,$$
where the sum is over all $(\phi,\psi)$, where the supports of $\phi$
and $\psi$ form a partition of $\ul p$. By grouping terms
corresponding to  partitions of the same size together, we can rewrite
this as
$$\sum_{i+j=p}\binom{p}{i}
\sum_{\phi_1,\ldots,\phi_i>0}\binom{t}{\phi_1}\ldots\binom{t}{\phi_i}
E_{|\phi|}(\xi)
\sum_{\psi_1,\ldots,\psi_j>0}\binom{t}{\psi_1}\ldots\binom{t}{\psi_j}
E_{|\psi|}(\chi)\,,$$
which is equal to 
$$\sum_{i+j=p}\binom{p}{i} E_i\pi_t(\xi)\, E_j\pi_t(\chi)\,,$$
by (\ref{t17}). 
Modulo $t^{p+1}$, this term is congruent to 
$$\sum_{i+j=p}\binom{p}{i} E_i\pi_i(\xi)\,E_j\pi_j(\chi)=
p!\sum_{i+j=p}\pi_i(\xi)\,\pi_j(\chi)\,t^p\,.$$
We conclude that
$$
\frac{1}{p!}E_p\big(\pi_t(\xi)\ast\pi_t(\chi)\big) \equiv\\
\sum_{i+j=p}\pi_i(\xi)\,\pi_j(\chi)\,t^p\mod
t^{p+1}\,,
$$
which proves  Lemma~\ref{mainlemma}.

\subsection{The semi-classical Hall algebra}\label{semiclass}
\newcommand{\kK}{\mathscr{K}}

By Theorem~\ref{main}, the submodule 
$$\kK(\MM)=\bigoplus_{n\geq0}t^n K^{\leq n}(\MM)$$
of $K(\MM)[ t]$
 is a $K(\DM)[t]$-subalgebra with respect to the
Hall product. The algebra $\kK(\MM)$ is a one-parameter flat family of
algebras. The special fibre at $t=0$ is canonically isomorphic to the
graded algebra associated to the filtered algebra
$\big(K(\MM),\ast\big)$.  The quotient map $\kK\to \kK/t\kK$ is identified
with the map $\sum_n x_nt^n\mapsto \sum_n\pi_n(x_n)$.

The graded algebra associated to the filtered algebra
$\big(K(\MM),\ast\big)$, is canonically isomorphic to the
commutative graded algebra $\big(K(\MM),\,\cdot\,\big)$, by
Theorem~\ref{main}.  The  special fibre inherits therefore a Poisson
bracket, which encodes the Hall product to second order.  This Poisson
bracket has degree $-1$, and is given by the formula
\begin{equation}\label{poiss}
\{x,y\}=\pi_{k+\ell-1}(x\ast y-y\ast x)\,,\qquad\text{ for $x\in
  K^k(\MM)$, $y\in K^\ell(\MM)$}\,.
\end{equation}

\begin{cor}
The graded $K(\DM)$-algebra $\big(K(\MM),\,\cdot\,\big)$ is endowed with a
Poisson bracket of degree $-1$, given by (\ref{poiss}). 
\end{cor}

\begin{cor}
In particular, $K^1(\MM)$ is a Lie algebra with respect to the Poisson
bracket~(\ref{poiss}).  In fact, for $x,y\in K^1(\MM)$, we have that
$x\ast y-y\ast x\in K^1(\MM)$, so in this case, the Poisson bracket is
equal to the Lie bracket.  Thus, $K^1(\MM)$ is a
Lie algebra over the ring of scalars $K(\DM)$. 
\end{cor}
\begin{proof}
Equation (\ref{m16}) for $p=0$, together with Lemma~\ref{exact} says
$$E_0\big(\pi_t(x)\ast\pi_t(y)\big)=E_0(x)E_0(y)\,.$$
This proves that $E_0(x)=0$ or $E_0(y)=0$ implies that $E_0(x\ast
y)=0$. 
\end{proof}

\begin{defn}
We call  $K^1(\MM)$ the Lie algebra of {\bf virtually
  indecomposable }stack functions. We will usually write $K^\vir(\MM)$
for $K^1(\MM)$.  
\end{defn}

This terminology is used in analogy with that of
\cite{JoyceII}. In Section~\ref{smallish}, we check that our notion of
virtually indecomposable agrees with that of [ibid.] in a special
case. 

\subsection{Epsilon functions}\label{seceps}

We will prove that replacing direct sum decompositions by filtrations,
in the formula 
$$\pi_k=\sum_{n\geq k}\frac{s(n,k)}{n!}E_n\,,$$
will give rise to an operator mapping $K(\MM)$ into $K^{\leq
  k}(\MM)$. In particular, we will be able to construct virtually
indecomposable stack functions as `Hall algebra logarithms'.

Fix an algebraic substack  $\NN\hookrightarrow \MM$, with the
following properties:
\comment{Do we really gain anything by not simply taking $\NN=\MM$?
  Of course it would restrict our choice of $\MM$, but it would make
  the treatment simpler. We could require $\MM$ to be an exact linear
  category with the finiteness condition.}
\begin{items}
\item $\NN$ avoids the image of  $\Spec
R\stackrel{0}{\longrightarrow}\MM$.  
\item $\NN$ is closed under direct sums and direct summands, i.e., it
  is Karoubian (Remark~\ref{karoubi}) if we add $\spec
  R\stackrel{0}{\longrightarrow} \MM$ to it,
\item for every positive integer $n$, the morphism
  $b|_{\NN^{(n)}}:\NN^{(n)}\to\MM$, illustrated in the diagram
\begin{equation}\label{nn}
\vcenter{\xymatrix{
\NN^{(n)}\rto\dto & \MM^{(n)}\rto^b\dto^{a_1\times\ldots\times a_n} & \MM\\
\NN^n\rto & \MM^n\rlap{\,,}}}
\end{equation}
where the square is cartesian,
is of finite
  type.\comment{probably suffices to require this for $n=2$}
\item the disjoint union over all these morphisms
  $\coprod_{n>0}\NN^{(n)}\to \MM$ is still of finite type. This means
  that if $X\to \MM$ is a morphism with  $X$ of finite type, there exists
  an $N>0$, such that for all $n\geq N$, the image of
  $b|_{\NN^{(n)}}:\NN^{(n)}\to \MM$ does not intersect the image of
  $X$ in $\MM$. 
\end{items}

\begin{ex}
If   $\MM$ is the stack of coherent sheaves on a projective curve,
then the substack of non-zero semi-stable vector bundles of fixed slope is an
example of a substack $\NN$ satisfying our conditions. More generally,
we can take for $\NN$ the stack of all vector bundles whose
Harder-Narasimhan slopes are contained in a fixed interval. 
\end{ex}

\begin{ex}
If $\MM$ is the stack of representations of a quiver $Q$, then we can
take $\NN=\MM_\ast$.
\end{ex}

Consider  an arbitrary  stack function $X\to\MM$, and denote by $F_nX$,
for $n\geq1$, the stack
$$F_nX=\NN^{(n)}\times_\MM X\,.$$
It fits into the cartesian diagram
$$\xymatrix{ F_nX\rto \dto &X^{(n)}\rto\dto& X\dto^M\\ \NN^{(n)}\rto &
  \MM^{(n)}\rto^b& \MM\rlap{\,.}}$$ 
Note that $F_nX$ is of finite type,
by our assumption on $\NN$, and also representable over $\MM$ (as an
algebroid), because 
$b$ is. Therefore, $F_n X$ is another stack function.

The objects of $F_nX$ are pairs 
$(x,F)$, 
where $x$ is an object of $X$, and $F=(F_1\to\ldots \to F_n)$ is a
flag in $F_n=M$, where $M$ is the image of $x$ in $\MM$, 
such that all subquotients 
$ F_{\nu}/F_{\nu-1}$, for 
$\nu=1,\ldots,n$, are in $\NN$. 

We now consider $E_k(F_nX)$, for $k\geq0$.  This is the stack of
triples
$$\big(x,(e_\kappa),F\big)\,,$$
where the pair $(x,F)$ is an object of $F_nX$, and
$(e_\kappa)=(e_1,\ldots,e_k)$ is a complete set  of  non-zero orthogonal
idempotents in $A(X)$, such that 
for every $\kappa=1,\ldots,k$ the  endomorphism of $M$ induced
  by $e_\kappa$
  respects the flag $F$.
For every $\nu=1,\ldots,n$, we therefore get an induced idempotent
operator 
$$f_{\kappa,\nu}=e_\kappa|_{F_\nu/F_{\nu-1}}\,.$$
These idempotents  have the properties 
\begin{items}
\item  $\sum_{\kappa} f_{\kappa,\nu}=1$, for
all $\nu$,
\item for every $\kappa$, at least one of the
$f_{\kappa,\nu}$ does not vanish.\comment{here, for example, we use
  representability of the stack function $X$}
\end{items}

The stack  $E_k(F_nX)$ decomposes into a disjoint union of substacks
according to which of the idempotents $(f_{\kappa,\nu})$ vanish. 

To make this decomposition precise, consider a sequence of positive
integers 
$(\alpha_1,\ldots,\alpha_k)$. Define $F_\alpha X$ to be the stack of
triples 
$$\big(x,(e_\kappa),(F_\kappa)\big)\,.$$
Here $x$ is an object of $X$, with image $M$ in $\MM$, and
$(e_\kappa)$ is a complete set of 
orthogonal non-zero idempotents for  $x$, which decomposes $M$ into a
direct sum $M=\bigoplus_\kappa M_\kappa$.  
Moreover, $F_\kappa$ is a flag of length $\alpha_\kappa$ on
$M_\kappa$, with subquotients in $\NN$, for every $\kappa=1,\ldots,k$.

For every $k$-tuple of strictly monotone maps $\Phi_\kappa:\ul
{\alpha_\kappa}\hookrightarrow \ul n$,  we define a morphism
\begin{equation}\label{alp}
F_\alpha X\longrightarrow E_k(F_nX)\,,
\end{equation}
by defining the flag $F$ on $M$ in terms of the $k$-tuple of flags
$(F_\kappa)$ by 
$$F_\nu=\bigoplus_{\kappa}\sum_{\Phi_\kappa(\rho)\leq\nu}F_\rho\,.$$
Note that the sum for fixed $\kappa$ is not really a sum, it is just
the largest of the subobjects $F_\rho$ of $M_\kappa$ making up  the
flag $F_1\to\ldots\to 
F_{\alpha_\kappa}$, such that  $\Phi_\kappa(\rho)\leq \nu$. 

\begin{lem} \comment{we need, for sure, that $\NN$ is Karoubian,
    here.  Make sure it makes sense that $\NN$ does not contain $0$.}
The morphism (\ref{alp}) given by $(\Phi_\kappa)_{\kappa\in\ul k}$ is
an isomorphism onto the locus in 
$E_k(F_nX)$, defined  by $f_{\kappa,\nu}\not=0$ if and only if
$\nu\in\Phi_\kappa(\ul{\alpha_{\kappa}})$, for all
$\kappa=1,\ldots,k$. \nolinebreak$\Box$
\end{lem}

\begin{cor}  
If $\xi$ denotes the element of $K(\MM)$ defined by $X\to \MM$, we have
$$ E_k(F_n\xi)= \sum_{\alpha}{n\brack\alpha}F_\alpha(\xi)\,,$$ 
where the sum is taken over all $k$-tuples of positive integers.

If we set $F_0(\xi)=1$, and $F_{\varnothing}(\xi)=1$, this equality
also holds for $n=0$. 
\end{cor}

\begin{defn}
Define, for every $\xi\in K(\MM)$,
$$\epsilon_t(\xi)=\sum_{n\geq 0}\binom{t}{n} F_n(\xi)\,,$$
where for $n=0$, we set $F_0(\xi)=1$. This definition is justified,
because 
by our assumptions on $\NN$, this sum is actually finite. 

Expanding in powers of $t$ defines the $\epsilon_k(\xi)$, for $k\geq0$:
$$\epsilon_t(\xi)=\sum_{k\geq0}\epsilon_k(\xi)t^k\,.$$
For example, $\epsilon_0=1$, and 
$$\epsilon_1(\xi)=\sum_{n>0}\frac{(-1)^{n+1}}{n}F_n(\xi)\,.$$
In general,
$$\epsilon_k(\xi)=\sum_{n\geq k}\frac{s(n,k)}{n!} F_n(\xi)\,.$$
\end{defn}

\begin{cor}
For every $k\geq0$, we have $\epsilon_k(\xi)\in K^{\leq k}(\MM)$.
Hence
$\epsilon_t(\xi)\in \kK(\MM)$.  In particular, $\epsilon_1(\xi)$ is
virtually indecomposable, for all $\xi\in K(\MM)$. 
\end{cor}
\begin{proof}
It suffices to prove that $E_k\big(\epsilon_t(\xi)\big)\equiv 0\mod
(t^k)$, for all $k$.  In fact, 
\begin{align*}
E_k\big(\epsilon_t(\xi)\big)&=
\sum_{n\geq0} \binom{t}{n}E_kF_n(\xi)\\
&=\sum_{n\geq0}\binom{t}{n}
\sum_{\alpha_1,\ldots,\alpha_k>0}{n\brack\alpha}F_\alpha(\xi)\\  
&=\sum_{\alpha_1,\ldots,\alpha_k>0}
\binom{t}{\alpha_1}\ldots\binom{t}{\alpha_k}F_\alpha(\xi) 
\end{align*}
is, indeed, divisible by $t^k$, if  all $\alpha_1,\ldots,\alpha_k$ are
positive. 
\end{proof}

\begin{rmk}
The operator $F_n:K(\MM)\to K(\MM)$ respects strict algebroids. The
same is true for all $\epsilon_k$.
\end{rmk}

\subsubsection{Epsilons as logarithms}\label{ealog}

Suppose there exists an  abelian group
$\Gamma$, and a decomposition 
of $\MM$ (as an algebroid, not a linear stack) into a disjoint union
$$\MM=\coprod_{\gamma\in \Gamma}\MM_\gamma\,.$$
We require that if $\EE_{\gamma,\beta}$ is 
defined by the cartesian diagram 
$$\xymatrix{
\EE_{\gamma,\beta}\rto\dto & \MM^{(2)}\dto^{a_1\times a_2}\\
\MM_\gamma\times\MM_\beta\rto & \MM\times \MM}$$
then the composition
$$\EE_{\gamma,\beta}\longrightarrow
  \MM^{(2)}\stackrel{b}{\longrightarrow}\MM$$
factors through $\MM_{\gamma+\beta}\subset\MM$. 

We call such $\Gamma$ a {\bf grading group }for $\MM$. 

The grading group $\Gamma$ decomposes $K(\MM)$ into a direct
sum
\begin{equation}\label{dirsumd}
K(\MM)=\bigoplus_{\gamma\in\Gamma}K(\MM)_\gamma\,,
\end{equation}
where $K(\MM)_\gamma$ is the submodule of $K(\MM)$ generated by stack
functions $X\to \MM$ which factor through $\MM_\gamma$. The Hall
product, as well as the commutative product, are graded with respect
to~(\ref{dirsumd}). For $x\in K(\MM)$, we denote the projection of $x$
into the component $K(\MM)_\gamma$ by $x_\gamma$.

Let $\NN\subset \MM$ be a linear algebraic substack avoiding
$\spec R\stackrel{0}{\longrightarrow}$, with the properties
\begin{items}
\item every intersection $\NN_\gamma=\NN\cap \MM_\gamma$ is of finite
  type,
\item $\NN$ is closed under direct summands and extensions in $\MM$,
  the latter meaning that if
  $\NN^{(2)}$ is defined as in (\ref{nn}), then the composition
  $\NN^{(2)}\to\MM^{(2)}\stackrel{b}{\longrightarrow} \MM$ factors
  through $\NN\subset\MM$,
\item there is a submonoid  $\Gamma_{+}\subset\Gamma$,  such that
  $\NN_\gamma\not=\varnothing$ implies that $\gamma\in
  \Gamma_+\setminus\{0\}$.   The monoid $\Gamma_+$ is required to have
  the property that every 
  $\gamma\in\Gamma_+$ admits only finitely many decompositions
  $\gamma=\alpha+\beta$, such that both $\alpha,\beta\in\Gamma_+$. We
  will further assume that $\Gamma_+$ has the property that the
  intersection of all cofinite ideals is empty.
\end{items}
If these  axioms hold, $\NN$ satisfies the
finiteness conditions above, so that the
$\epsilon_k(\xi)$ are
defined, for all $\xi\in K(\MM)$.

\begin{rmk}
If $\MM$ is the stack of coherent sheaves on a projective curve, then
we can take $\Gamma=\zz^2$, and define $\MM_{(d,r)}$, for
$(d,r)\in\zz^2$ to be the stack of sheaves of rank $r$ and degree
$d$. Suppose $\NN$ is the stack of bundles whose Harder-Narasimhan
slopes are contained in the interval $(a,b)\subset \mb{R}$.  Then we
can take 
$$\Gamma_+=\{(0,0)\}\cup\{(d,r)\in\zz^2\mid \text{$r>0$ and
  $a<\sfrac{d}{r}<b$}\}\,,$$
and the above requirements will be satisfied. 
\end{rmk}

\begin{rmk}
If $\MM$ is the stack of representations of a quiver $Q$, we can take
$\Gamma=\zz^{Q_0}$, where $Q_0$ is the set of vertices of $Q$, and
then set $\MM_\gamma$, for $\gamma\in\Gamma$, equal to the stack of
representations with dimension vector $\gamma$. If we take $\NN=\MM_\ast$,
we can take $\Gamma_+=\zz^{Q_0}_{\geq0}$.
\end{rmk}

Let us also define
$$K(\MM)_S=\bigoplus_{\gamma\in S}K(\MM)_\gamma\subset
K(\MM)\,,$$
for any cofinite ideal $S\subset \Gamma_+$. For every such $S$, the group
$K(\MM)_S$ is an ideal (with respect to both multiplications) in the ring
$K(\MM)_+=K(\MM)_{\Gamma_+}$, and we may
complete $K(\MM)_+$  with respect to this collection of ideals, to
obtain $\hat K(\MM)_+$. The morphism $K(\MM)_+\to \hat K(\MM)_+$ is
injective and both multiplications extend to $\hat K(\MM)_+$.

In $\hat K(\MM)_+$ the sum
$$[\NN]=\sum_{\gamma\in \Gamma_+}[\NN_\gamma\to \MM]$$
converges.

The idempotent operators commute with the $\Gamma$-grading, and so
everything defined in terms of them does, too.

\begin{prop}
In $\hat K(\MM)_+[[t]]$, we have
$$\epsilon_t[\NN]=\sum_{n\geq0}\binom{t}{n}[\NN]^{\ast n}\,.$$
Hence we can write 
$$\epsilon_t[\NN]=(1+[\NN])^{\ast
  t}=\exp_\ast\big(t\log_\ast(1+[\NN])\big)\,,$$
where exponential and logarithm are defined by their power series
using the Hall product. 
In particular,
$$\epsilon_1[\NN]=\log_\ast(1+[\NN])\,,$$
and 
$$\epsilon_k[\NN]=\frac{1}{k!}\log(1+[\NN])^{\ast
  k}=\frac{1}{k!}(\epsilon_1[\NN])^{\ast k}\,.$$
\end{prop}
\begin{proof}
Using the formula
$$F_n[\NN]=[\NN^{(n)}\stackrel{b}{\longrightarrow}\MM]=
\underbrace{[\NN]\ast\ldots\ast[\NN]}_{\text{$n$ times}}\,,$$
the result follows.
\end{proof}

\begin{rmk}\label{cytgo}
Compare the two formulas
\begin{align*}
\epsilon_t[\NN]&=(1+[\NN])^{\ast t}\,,\\
\pi_t[\NN]&=(1+[\NN])^t\,.
\end{align*}
\end{rmk}

\begin{rmk}
Let us write $\hat\K(\MM)_+$ for the subspace of $\hat K(\MM)_+[[t]]$,
defined by requiring the coefficient of $t^k$ to be contained in
$\hat K^{\leq k}(\MM)_+$, for all $k$.  Then 
$$
\epsilon_t[\NN]=(1+[\NN])^{\ast t}\in \hat \K(\MM)_+\,.
$$
\end{rmk}
\begin{rmk}
Setting $t=1$, we also get that $1+[\NN]=\exp_\ast(\epsilon_1[\NN])$. 
One should think of $1+[\NN]$ as {\em group-like}.
\end{rmk}

\subsubsection{Hopf algebra}

We make a brief remark, without striving for generality.

Let us fix $\NN\subset \MM$ and $\Gamma_+\subset\Gamma$ as
before. Assume for simplicity that $\Gamma$ is free. In particular,
$\alpha+\beta=0$, for $\alpha,\beta\in \Gamma_+$, implies
$\alpha=\beta=0$.

For
$0\not=\gamma\in\Gamma_+$, abbreviate the element $[\NN_\gamma\to\MM]\in
K(\MM)$ by $[\gamma]$. 

For a finite sequence $(\gamma)=\gamma_1,\ldots,\gamma_n$ of
non-zero elements of $\Gamma_+$, write
$$[(\gamma)]=[\gamma_1,\ldots,\gamma_n]=[\gamma_1]\ast\ldots\ast[\gamma_n]\,.$$
In particular, for $n=0$, we have $[\varnothing]=1$. 

In many cases of interest, the Hall algebra elements $[(\gamma)]$, as
$(\gamma)$ runs over all finite sequences of non-zero elements of
$\Gamma_+$ are linearly independent over $\qq$.  Let us assume that
this is the case.  Then the $\qq$-span of all $[(\gamma)]$ is a
$\qq$-subalgebra of $K(\MM)$, which is free on the generators
$[\gamma]$, for $\gamma\in \Gamma_+\setminus\{0\}$, as a unitary
$\qq$-algebra.  Let us denote this algebra by $U$. Let us further
assume that the morphism induced by the commutative product $U\otimes
U\to K(\MM)$ is injective.  (Again, this will hold in many cases of
interest.)\comment{say for which cases}

We will now define a comultiplication $\Delta$ on $U$, making a $U$ a
cocommutative Hopf algebra over $\qq$. 

To define $\Delta$, it is convenient to extend the notation
$[(\gamma)]$ to finite sequences of elements of $\Gamma_+$, which may
be zero. This is done by setting $[0]=1$.  Thus $[(\gamma)]$ is
unchanged by `crossing off its zeros'.  We then define
\begin{align*}
\Delta:U&\longrightarrow U\otimes U\\
[(\gamma)]&\longmapsto \sum_{(\alpha)+(\beta)=(\gamma)}[(\alpha)]\otimes[(\beta)]\,,
\end{align*}
where the sum is over all pairs of sequences of the same length as
$\gamma$, but allowing zeros. 

This defines on $U$ the structure of a cocommutative Hopf-algebra. 

\begin{lem}
The diagram
$$\xymatrix{
U^\ast\ar[rrr]^{x\longmapsto \Delta(x)-1\otimes x-x\otimes 1}\ar[drrr]_{E_2}
&&&U\otimes U\dto^{\text{comm.\ mult.}}  \\
&
&&K(\MM)}$$
is commutative, where $U^\ast\subset U$ is the augmentation ideal.
\end{lem}

It follows that the virtual indecomposables in $U$ are equal to the
primitives with respect to the Hopf algebra structure:
$$U^\prim=U^\vir\,.$$
As $U$ is a cocommutative Hopf algebra, it is isomorphic to the
universal enveloping algebra of $U^\prim$, by the Cartier-Gabriel
theorem \cite[Theorem~3.8.2]{CartierHopf}.

The Lie algebra $U^\vir$ is free, as a Lie algebra over $\qq$, on the
elements
$$\epsilon[\gamma]=\sum_{n>0}\frac{(-1)^{n+1}}{n}
\sum_{\stack{\gamma=\gamma_1+\ldots+\gamma_n}{\gamma_1,\ldots,\gamma_n>0}}
    [\gamma_1]\ast\ldots\ast[\gamma_n]\,,$$
for $\gamma\in \Gamma_+\setminus\{0\}$. 

\begin{rmk}
Suppose $\NN=\MM=\mf{Vect}$ is the stack of vector bundles. 
We take $\Gamma=\zz$ and $\Gamma_+=\zz_{\geq0}$. Then the $[(\gamma)]$
are, indeed, linearly independent over 
$\qq$, at least if our ground ring $R$ is a field. Moreover,
$U\otimes U\to K(\MM)$ is injective.  The Hopf 
algebra we obtain is the Hopf algebra of {\em non-commutative
  symmetric functions}, see \cite[Example~4.1~(F)]{CartierHopf}.
\end{rmk}

\begin{rmk}
It is doubtful that it is possible to extend the coproduct 
to all of $K(\MM)$, in such a way that  $K^\vir(\MM)=K(\MM)^\prim$. By the above
considerations, we consider the family of operators $(E_n)$ as a
substitute, which allows us to prove at least some of the result one
would expect in a cocommutative Hopf algebra.  In particular, we find
it unlikely that, in general, $K(\MM)$ would be isomorphic to the
universal enveloping algebra of $K^\vir(\MM)$.
\end{rmk}

\section{Integration}
\renewcommand{\St}{\mathop{\rm St}}

The integral of a stack function $(X,A)\to (\MM,\AA)$ does three
things: it forgets the structure map to $(\MM,\AA)$, it forgets the
algebroid structure, mapping $(X,A)$ to $X$, and it introduces the
bundle relations in $K(\St)$, for non-inert morphisms of algebraic stacks.

\subsubsection{The vector bundle relations}

Let $K(\St)$ be the Grothendieck $K(\DM)$-algebra of algebraic stacks
(finite type, with affine diagonal), modulo the scissor and the bundle
relations. A {\em bundle relation }is any equation of the form
$$[Y]=[F\times X]\,,$$
for a fibre bundle $Y\to X$ of algebraic stacks with special structure
group and fibre $F$.

It is well-known, that\comment{check the intermediate one}
\begin{align*}
K(\St)&=K(\DM)[\sfrac{1}{q}][\sfrac{1}{q^n-1}]_{n\geq1}\\
      &=K(\Var)[\sfrac{1}{q}][\sfrac{1}{q^n-1}]_{n\geq1}\,.
\end{align*}
We prefer the latter expression in terms of $K(\Var)$. 

Note that the (connected, semi-simple) inertia operator does not
preserve non-inert bundle relations. Therefore, in $K(\St)$, we cannot talk about
$I^{\circ,\ss}[X]$, only about $[I^{\circ,\ss}_X]$.

\subsubsection{Regular motivic weights}

\begin{defn}
We call an element of $K(\St)$ {\bf regular}, if it can be written
with a denominator which does not vanish at $q=1$. Thus the
subalgebra of regular motivic weights $K(\St)_\reg\subset
K(\St)$ is by definition the image of the morphism of
$K(\Var)$-algebras
$$
K(\Var)[\sfrac{1}{q}][\sfrac{1}{q^n+\ldots+1}]_{n\geq1}
\longrightarrow
K(\Var)[\sfrac{1}{q}][\sfrac{1}{q^n-1}]_{n\geq1}=K(\St)\,.$$
\end{defn}

The image of $K(\DM)$ in $K(\St)$ is contained in
$K(\St)_\reg$. Hence we can also think of $K(\St)_\reg$ as a
$K(\DM)$-algebra. This follows from the following Lemma.

\begin{lem}\label{qfinreg}
Every finite type stack with quasi-finite stabilizer has regular
motivic weight in $K(\St)$. 
\end{lem}
\begin{proof}
Let $Z$ be a stack with quasi-finite stabilizer (meaning that its
inertia $I_Z$ is quasi-finite over $Z$).  By stratifying $Z$, if
necessary, we may assume that the inertia stack of $Z$ is in fact
finite. By \cite{Kresch}, Proposition~3.5.7, $Z$ is stratified by
global quotient stacks, so we may assume that $Z=Y/\GL_m$, for an
algebraic space $Y$, such that $\GL_m$ acts on $Y$ with finite
stabilizer. The maximal torus $T\subset \GL_m$ then also acts with
finite stabilizer on $Y$.  The flattening stratification
$\coprod Y_i\to Y$ of the
stabilizer $\Stab_TY$ is then $T$-equivariant, so that $T$ acts on
each $Y_i$, and by passing to open and closed subspaces of the $Y_i$,
we may assume that the action of $T$ on $Y_i$ has constant
stabilizer (see \cite[I, 5.4]{Oesterle}).  Then $T$ acts on
$Y_i$ freely through a quotient $T_i$ by 
a finite subgroup.  We conclude:
\begin{equation}\label{lstem}
[Z]=\frac{1}{[\GL_m]}\sum_i [Y_i]=
\frac{1}{[\GL_m]}\sum_i[T_i][Y_i/T_i]=\frac{[T]}{[\GL_m]}\sum_i[Y_i/T_i]\,,
\end{equation}
because each quotient $T_i$ of $T$ is isomorphic to $T$. The last term
in (\ref{lstem}) is regular, because each $Y_i/T_i$ is an algebraic
space, and 
$$\frac{[T]}{[\GL_m]}=q^{-\frac{1}{2}m(m-1)}\prod_{i=1}^m\frac{1}{q^i+\ldots+1}$$
does not vanish at $q=1$.
\end{proof}
 
\subsubsection{The integral}

Mapping a stack function $(X,A)\to (\MM,\AA)$ to the class $[X]\in
K(\St)$ gives rise to  a well-defined homomorphism $K(\MM)\to
K(\St)$ of $K(\DM)$-modules. We
denote this homomorphism by
$$\int:K(\MM)\to K(\St)\,.$$

\subsection{The No Poles theorem}

\begin{thm}
The composition $\int\circ I^{\circ,\ss}$ factors through the algebra of
regular motivic weights:
\begin{equation}
\vcenter{\xymatrix{
K(\MM)\rto^{I^{\circ,\ss}}\dto_{\int\circ I^{\circ,\ss}} & K(\MM)\dto^{\int\phantom{T}}\\
K(\St)_\reg\ar@{^{(}->}[r] & K(\St)\rlap{\,.}}}
\end{equation}
\end{thm}
\begin{proof}
It suffices to prove that $[A^{\times,\ss}]\in K(\St)$ is regular, for
every clear algebroid $(X,A)$, admitting a faithful representation.
This will suffice, by Proposition~\ref{faith}. 
If $(X,A)$ is such an algebroid, there exists a Deligne-Mumford stack $Y$,
with a left $\GL_n$-action, 
together with a strict $\GL_n$-equivariant algebra subbundle
$B\hookrightarrow M_{n\times   n}|_Y$, such that
$B^\times=\Stab_{\GL_n}Y$, and $(X,A)=(\GL_n\backslash Y, \GL_n\backslash B)$.

Let $D_n\subset M_{n\times n}$ be the diagonal subalgebra, and
$T=D_n^\times$ the standard maximal torus of $\GL_n$. As $\GL_n$ acts
on $B\subset M_{n\times n}|_Y$, so does the torus $T$. 
We will now stratify $B$ by the stabilizer with respect
to the action of  $T$. For this stratification to be canonical, we
need $\Stab_T B$ to be the units in a finite type algebra over $B$. 

In fact, such an algebra $C\subset D_n|_B$ is given as the
intersection of $D_n|_B$ with $(B|_B)^\fix$ inside $M_{n\times
  n}|_B$. Here $(B|_B)^\fix$ is the centralizer of the tautological
section of $B|_B$,  or, under the
identification $B|_B=B\times_YB$, the stack of commuting pairs. 
Thus, a section $(u,b,y)\in D_n|_B$ is in $C$, if and only if $u\in
Z_{B(y)}(b)$. 

We have, indeed, an equality 
$$C^\times=\Stab_TB$$
 of relative group
schemes over $B$, because for $t\in T$, and $(b,y)\in B\subset
M_{n\times n}|_Y$, 
\begin{align*}
t\in \Stab_T(b,y)
&\Longleftrightarrow \text{${}^tb=b$ and $ty=y$}\\
&\Longleftrightarrow \text{$tb=bt$ and $t\in\Stab_{\GL_n(y)}$}\\
&\Longleftrightarrow \text{$tb=bt$ and $t\in B^\times(y)$}\\
&\Longleftrightarrow t\in Z_{B^\times(y)}(b)\\
&\Longleftrightarrow t\in C^\times(b,y)\,.
\end{align*}

The subalgebras of $D_n$ are in one-to-one correspondence with
partitions $I=\{I_1,\ldots,I_r\}$ of the set $\ul n=\{1,\ldots,n\}$.
The partition $\ul n=I_1\sqcup\ldots\sqcup I_r$ corresponds to the
subalgebra $D_I$ whose primitive idempotents are the
$e_{I_\rho}=\sum_{i\in I_\rho}e_i$, for $\rho=1,\ldots,r$. Let us
write $T_I=D_I^\times$ for the torus of units in $D_I$. 

Now there is a unique stratification 
\begin{equation}\label{IB}
\coprod_{I} B_I\longrightarrow B\,,
\end{equation}
such that a section $(b,y)$ of $B$ factors through $B_I$, if and only
if the pullback of $C\subset D_n|_B$ via $(b,y)$ is equal to $D_I$. 
The existence of this stratification is proved by passing to the rank
stratification (see Definition~\ref{rankstrat}) of $C$, and observing
that a subalgebra bundle $C\subset D_n|_S$, for any stack $S$,
decomposes $S$ into a disjoint union of open and closed substacks,
such that $C$ is constant over these components. We may reformulate
the defining property of $B_I$ by saying that $(b,y)\in B_I$, if and
only if $Z_{B(y)}(b)\cap D_n=D_I$. We also have, for $(b,y)\in B_I$,
that $\Stab_T(b,y)=C^\times(b,y)=D_I^\times=T_I$. 

The stratification (\ref{IB}) is $T$-equivariant, because for $t\in
T$, and $(b,y)\in B$, we have
\begin{align*}
(b,y)\in B_I
&\Longleftrightarrow Z_{B(y)}(b)\cap D_n=D_I\\
&\Longleftrightarrow Z_{B(ty)}({}^tb)\cap {}^tD_n={}^tD_I\\
&\Longleftrightarrow Z_{B(ty)}({}^tb)\cap D_n=D_I\\
&\Longleftrightarrow ({}^tb,ty)\in B_I\,,
\end{align*}
as $T\subset D_n$, and $D_n$ is commutative. 

(Let us remark that we were not able to
prove that for a general action of $T$ on a Deligne-Mumford stack $Y$,
the stabilizer stratifies $Y$ equivariantly. The fact that the
stabilizer is equal to the units in an algebra helps. Note also, that
we did not prove a defining property for (\ref{IB}) in terms of
stabilizers in $T$.)

So, for every  partition $I$ of $\ul n$, the torus  $T$ acts on 
$B_I\subset B$, with stabilizer $T_I$. We therefore get an induced
action of $T/T_I$ on $B_I$ .
Matrix conjugation preserves 
units, so we get and induced action of $T/T_I$ on 
$$B_I^{\times}=B_I\cap B^{\times}\,.$$
In fact, this action even respects $B_I^{\times,\ss}=B_I\cap
B^{\times,\ss}$, but the following modification does not.

Consider the action of $T_I\subset T$ on $B_I$ by left
multiplication:
$$t(b,y)=(tb,y)\,.$$
This is a well-defined action, because 
$t\in T_I$, and $(b,y)\in B_I$, implies that 
$t\in Z_{B(y)}(b)$.  In
particular, $t\in B(y)$, so that $(t,y)\in B$, and the product
$(t,y)(b,y)=(tb,y)$ in $B$ exists. Moreover, $(tb,y)\in B_I$, because
$Z_{B(y)}(tb)\cap D_n=Z_{B(y)}(b)\cap D_n$.
This action of $T_I$ on $B_I$ preserves $B_I^\times$. 
Over {\em fields}, it also preserves 
sections 
which are semi-simple, because the product of two semi-simple
commuting matrices is again a semi-simple matrix. 
(Note that this does not imply that $T_I$ acts  on
$B_I^\ss$, because even if $(b,y)$ is a strict section of $B$, the
product $(tb,y)$ may not be strict.) 

We finally consider the action of $T(I)=T_I\times T/T_I$ on
$B_I^{\times}$, defined by 
\begin{equation}\label{actionfree}
(t',t)\ast(b,y)=(t'\,{}^tb,ty)\,.
\end{equation}

The quotient stack $Z_I=B_I^\times/T(I)$ is a finite type scheme over $R$,
so its Zariski topological space $|Z_I|$ is a Zariski space
(see~\cite{LMB}, Chapter~5). By Chevalley's theorem (see [ibid.]), 
the image of $|B_I^{\times,\ss}|$ in 
$|Z_I|$ is constructible, so we 
can find disjoint, locally closed (reduced) algebraic substacks
$Z_1,\ldots,Z_n\subset Z_I$, such that this image is equal to 
$|Z_1|\sqcup\ldots\sqcup|Z_n|\subset |Z_I|$. Let
$$\tilde Z_I=Z_1\amalg\ldots\amalg Z_n\,.$$ 
This is a finite type algebraic stack
endowed with a representable monomorphism $\tilde Z_I\to Z_I$. 

We claim that $\tilde Z_I$ is an algebraic stack with quasi-finite
stabilizer. This will
follow from the fact that, for field valued points, the action of
$T(I)$ on $B_I^{\times,\ss}$ has finite stabilizers. To see this,
assume that 
$$(t' \,{}^{t}b,t y)=(b,y)\,,$$
for $(t',t)\in T_I\times T/T_I$, and $(b,y)\in B^{\times,\ss}_I$.
This implies that $ty=y$, hence conjugation by $t$ preserves the fiber  $B(y)$ of
$B$ over $y$. 
We have
$$t' \,{}^{t}b=b\,,$$
where $t'$ commutes with $b$. 
Rewriting as ${}^{t}b={t'}^{-1}b$, we see that ${}^{t}b$ commutes
with $b$. Changing basis, if necessary, we can diagonalize the three
matrices $b$, ${}^{t}b$ and $t'$, simultaneously. Since  $b$
and ${}^{t}b$ have the same eigenvalues, we see that the entries of
the diagonal matrix ${}^{t}b$, are obtained from those of $b$ by a
permutation. Hence there are at most $n!$ possible values for
${}^{t}b$, and hence for $t'=b\,{}^{t}b^{-1}$. For every one of
these possible values of $t'$, there is at most one $t\in T/T_I$,
such that ${}^{t}b={t'}^{-1}b$. Thus the action (\ref{actionfree})
has finite stabilizers, at least on field valued points of
$B_I^{\times,\ss}$, as required.

Consider the cartesian diagram
$$\xymatrix{
\tilde B_I^{\times,\ss}\rto^{\tilde\beta} \dto_{\alpha'} & \tilde
B_I^\times\rto^{\tilde \pi}\dto & \tilde 
Z_I\dto^\alpha \\
B_I^{\times,\ss}\rto^\beta & B_I^\times\rto^\pi & Z_I\rlap{\,,}}$$
obtained by pulling back $B_I^{\times,\ss}\to B_I^\times\to Z_I$ via
$\tilde Z_I\to Z_I$. 
The morphisms  $\alpha$ and $\beta$ are disjoint unions of
isomorphisms onto locally closed substacks, so the same is true for $\alpha'$ and
$\tilde\beta$. But both $\alpha'$ and $\tilde\beta$ are surjective on
underlying Zariski topological spaces, so by the scissor relations, we
have 
$$[B_I^{\times,\ss}]=[\tilde B_I^{\times,\ss}]=[\tilde
  B_I^\times]\,,$$
in $K(\St)$.  
The morphism $\pi$ is a principal $T(I)$-bundle, so the
same is true for $\tilde\pi$, and so by the bundle relations, we have
$$[B_I^{\times,\ss}]=[\tilde B_I^\times]=[T(I)][\tilde
  Z_I]=(q-1)^n[\tilde Z_I]\,,$$
in $K(\St)$. 
It follows, that we have
\begin{multline*}
[A^{\times,\ss}]=[\GL_n\backslash
  B^{\times,\ss}]=\frac{[B^{\times,\ss}]}{[\GL_n]}
=\frac{1}{[\GL_n]}\sum_{I}[B^{\times,\ss}_I]\\
=\frac{1}{[\GL_n]}\sum_{I}(q-1)^n[\tilde Z_I]
=q^{-\frac{1}{2}n(n-1)}\prod_{i=1}^{n-1}\frac{1}{q^i+\ldots+1}\sum_{I}[\tilde
  Z_I]\,.
\end{multline*}
The claim now follows from Lemma~\ref{qfinreg}.
\end{proof} 

\begin{cor}\label{talk}
The multiple $(q-1)^k\int$ of the integral takes regular values on 
$K^{\leq k}(\MM)$, for every $k\geq0$. 
$$\xymatrix{
K^{\leq k}(\MM)\drto^{(q-1)^k\int}\dto\\
K(\St)_\reg\ar@{^{(}->}[r]& K(\St)\rlap{\,.}}$$
\end{cor}
\begin{proof}
Consider the following diagram:
$$\xymatrix{
K^k(\MM)\drrto^{(q-1)^k\int}\dto\ar@{.>}[dr]\\
\bigoplus_{|\lambda|=k}K^\lambda(\MM)_{(q-1)}\drto_{\bigoplus_{|\lambda|=k}\frac{k!}{\tilde
    Q_\lambda(q)}\int\circ I^\ss\phantom{mmi}}
&K(\St)_\reg\ar@{^{(}->}[r]\dto &K(\St)\dto\\ 
&K(\Var)_{(q-1)}\ar@{^{(}->}[r] & K(\St)(q)\rlap{\,.}}$$  
The dotted arrow exists because the square in the lower right
of this diagram is cartesian, and the outer part of the diagram
commutes. (Here we have identified the localization $K(\Var)_{(q-1)}$
with its image in $K(\St)(q)=K(\St)\otimes_{\qq[q]}\qq(q)$.)
\end{proof}

\begin{cor}
Defining $\int t=q-1$ extends the integral to a $K(\Var)$-linear
homomorphism
$$\int:\K(\MM)\longrightarrow K(\St)_\reg\,.$$
\end{cor}

\subsection{The integral vs. the Hall product}
\label{hereditary}

\subsubsection{The $\Gamma$-indexed integral}

Let  $\Gamma$ be a grading group for  $\MM$, as in
Section~\ref{ealog}. 
We assume, in addition, that $\Gamma$ is endowed with a $\zz$-valued bilinear
form $\chi$.

\begin{defn}
We call $\MM$ {\bf hereditary}, if for every $\gamma,\beta\in \Gamma$,
the morphism $\EE_{\gamma,\beta}\to
\MM_\gamma\times \MM_\beta$  is a vector bundle
stack (\cite{INC}, Definition~1.9)  of rank
$-\chi(\beta,\gamma)$.
\end{defn}

Let us assume henceforth that $\MM$ is hereditary. 

We define $K(\St)[\Gamma]$ to be the free $K(\St)$-module
on the symbols $u^\gamma$, for $\gamma\in\Gamma$, and introduce 
an associative product on $K(\St)[\Gamma]$ by
the formula
$$u^\gamma\ast u^\beta=q^{-\chi(\beta,\gamma)}u^{\gamma+\beta}\,,$$
and extend it linearly, to make
$K(\St)[\Gamma]$ a $K(\St)$-algebra. 
Regular coefficients form a subalgebra $K(\St)_\reg[\Gamma]$. 

We define the $\Gamma$-indexed integral
\begin{align}\label{theintint}
\int :\K(\MM)&\longrightarrow K(\St)_\reg[\Gamma]\\
\sum_{\gamma\in \Gamma}x_\gamma &\longmapsto \sum_{\gamma\in
  \Gamma} u^\gamma\int x_\gamma\,.\nonumber
\end{align}

\begin{prop}
If $\MM$ is hereditary the $\Gamma$-indexed integral
preserves the star product.  In fact,  for $x,y\in \K(\MM)$ we have
$$\int x\ast y=\int x\ast\int y\quad\in
K(\St)_\reg[\Gamma]\,.$$
\end{prop}
\begin{proof}
This is a straightforward calculation.  One uses the fact that for
$X\to \MM_\gamma$ and $Y\to \MM_\beta$, the morphism $X\ast Y\to
X\times Y$ is a vector bundle stack of rank $-\chi(\gamma,\beta)$, and
hence, in $K(\St)$, we have $[X\ast
  Y]=q^{-\chi(\gamma,\beta)}[X][Y]$. 
\end{proof}

\subsubsection{Semi-classical limit}

We will pass to the semi-classical limit of the integral
$\int :\K(\MM)\to K(\St)_\reg[\Gamma]$, by
setting $t=0$ (in the source), and hence $q=1$ (in the target).
As $\int $ respects the $\ast$-product, the semi-classical limit
will be a morphism of Poisson algebras. 

Modulo $(q-1)$, the star product on $K(\St)_\reg[\Gamma]$ is
commutative, in fact, modulo $(q-1)$ it is given by the
commutative product $u^\gamma\cdot u^\beta=u^{\gamma+\beta}$. Hence
the quotient $K(\St)_\reg[\Gamma]/(q-1)$
inherits a Poisson bracket, defined by 
$$x\ast y-y\ast x\equiv \{x,y\}(q-1)\mod (q-1)^2\,.$$  
Explicitly, it is given by 
\begin{equation}\label{poissfish}
\{u^\gamma,u^\beta\}=-\tilde \chi(\beta,\gamma)\,u^{\gamma+\beta}\,,
\end{equation}
where $\tilde \chi$ is (twice) the anti-symmetrization of $\chi$:
$$\tilde\chi(\beta,\gamma)=
\chi(\beta,\gamma)-\chi(\gamma,\beta)\,.$$

We conclude:
\begin{thm}\label{fish}
If $\MM$ is hereditary, we have a morphism of Poisson algebras
\begin{equation}\label{thefish}
\int_{q=1}: K(\MM)\longrightarrow K(\St)_\reg/(q-1)[\Gamma]\,.
\end{equation}
The Poisson structure on $K(\MM)$ is described in
Section~\ref{semiclass}, the one on
$K(\St)_\reg/\allowbreak{(q-1)}\allowbreak[\Gamma]$, above, 
see~(\ref{poissfish}). The $u^\gamma$ coefficient of the integral
$\int_{q=1}$ may be expressed as 
$$\int_{q=1}x=\Big(\int\pi_{q-1}(x)\Big)\Big|_{q=1}
=\sum_{n=0}^\infty\binom{q-1}{n}\int E_n(x)\Big|_{q=1}\,,$$
for $x\in K(\MM)_\gamma$. Here we have used the operator $\pi_t$ of
Definition~\ref{formalop}, and substituted $t=q-1$. 
\end{thm}
\begin{proof}
The homomorphism (\ref{thefish}) is obtained 
by setting $t=0$ in~(\ref{theintint}).  Note that the deformation
parameter $t$ is mapped to the deformation parameter $(q-1)$, so that
the Poisson bracket (which depends on the choice of the deformation
parameter) is preserved. 

To calculate $\int_{q=1}$, note that
$x\mapsto \pi_t(x)$ is a section of the quotient map $\K(\MM)\to
K(\MM)$, obtained by setting $t=0$. This gives rise to displayed
formula. 
\end{proof}

\begin{rmk}
Note that the diagram
$$\xymatrix{
K(\MM)\otimes \qq[t,\frac{1}{t}]\dto_\int&&\llto_-{\text{invert
    $t$}}\kK(\MM)\rrto^{t\,\longmapsto\,0} 
\dto_\int && K(\MM)\dto^{\int_{q=1}}\\ 
K(\St)[\Gamma]&&\llto
K(\St)_\reg[\Gamma]\rrto^{q\,\longmapsto\,1} &&
K(\Var)_\reg/(q-1)[\Gamma]}$$ 
commutes. The central  column is a morphism of one-parameter families
of non-commutative algebras.  The left hand column is the general
fibre, and a morphism of non-commutative $K(\Var)\otimes
\qq[t,\frac{1}{t}]$-algebras, the right hand column is the
semi-classical limit, and  hence a morphism of Poisson algebras. 
\end{rmk}

Restricting the theorem to the virtually indecomposable elements, we
obtain:

\begin{cor}
The semi-classical limit of the integral defines a morphism of Lie
algebras over $K(\Var)$: 
$$\int_{q=1}: K^\vir(\MM)\longrightarrow K(\St)_\reg/(q-1)[\Gamma]\,.$$
The bracket in $K^\vir(\MM)$ is the commutator bracket of the Hall
product, the bracket in $K(\Var)_\reg/(q-1)[\Gamma]$ is given
in~(\ref{poissfish}). The integral $\int_{q=1}$ is given by the
formula
$$\int_{q=1}x=
\sum_{\gamma\in\Gamma}u^\gamma \Big((q-1)\int
x_\gamma\Big)\Big|_{q=1}\,,$$
for a virtually indecomposable Hall algebra element $x\in K^\vir(\MM)$. 
\end{cor}

\begin{rmk}
We have a surjective morphism of $K(\Var)$-algebras
$$K(\Var)/(q-1)\stackrel{\sim}{\longrightarrow}
K(\St)_\reg/(q-1)\,.$$
This morphism is (most likely), not injective, because there is no
(obvious) reason why elements in $\Ann(q-1)\subset K(\Var)$ should map
to zero in $K(\Var)/(q-1)$, although they certainly map to zero in
$K(\St)_\reg/(q-1)$. 

Without too much more effort, it is possible to prove that  the
semi-classical limit of the integral lifts to a
$K(\Var)$-linear homomorphism
$$\int_{q=1}:\K(\MM)\longrightarrow K(\Var)/(q-1)[\Gamma]\,.$$
Unfortunately, we cannot, at the moment, prove that this lift is a
morphism of Poisson algebras. 
\end{rmk}

\begin{rmk}
We leave it to the (interested) reader to write down the analogue of
Theorem~\ref{fish} for the case that $\MM$ is Calabi-Yau-3, rather
than hereditary.  This will include proving our main theorems for an
equivariant version $K(\MM)^\muhat$ of $K(\MM)$, and including
vanishing cycle and orientation data weights in the integral.
\end{rmk}

\appendix
\setcounter{secnumdepth}{0}

\section{Appendix. Comparison with Joyce's virtual projections in an
  example} \label{smallish}

Let us write $[n]=[B\GL_n\to\mf{Vect}]\in K(\mf{Vect})$. We
have (cf.\ Remark~\ref{cytgo})
$$E_r[n]=\sum_{\stack{\ell_1+\ldots+\ell_r=n}{\ell_1\ldots\ell_r>0}}
[\ell_1]\ldots[\ell_r]\,.$$
This gives us
\begin{align*}
\pi_k[n]&=\sum_r\frac{s(r,k)}{r!} E_r[n]\\
&=\sum_r\frac{s(r,k)}{r!}\sum_{\stack{\ell_1+\ldots+\ell_r=n}{\ell_1\ldots\ell_r>0}}
   [\ell_1]\ldots[\ell_r]\\
&=\sum_{\lambda\vdash n}\frac{s(|\lambda|,k)}{|\Aut\lambda|}
   \prod_i[\lambda_i]\,.   
\end{align*}
We remark also, that the  formula of
Remark~\ref{convenient} gives us
$$\pi_t\Big(\sum_{n\geq0}[n]\,u^n\Big)=\Big(\sum_{n\geq0}[n]\,u^n\Big)^t\,,$$
which contains the above formulas for $\pi_k[n]$.

In \cite[\S 5.2]{JoyceSF}, Joyce defines projection operators
$\Pi_n^\vi: K(\MM)\to  K(\MM)$, which pairwise commute, and add up to
the identity (although he works with bare algebraic stacks, not
algebroids). We expect that modulo this difference, we have
$$\pi_k=\Pi^\vi_k\,.$$
We will prove that these operators take the same values on the
elements $[n]\in K(\mf{Vect})$.

\begin{prop}
  We have
  $$\pi_k[n]=\Pi_k^\vi[n]\,,$$
  for all $k$ and $n$.
\end{prop}
\begin{proof}
Let $T_n$ be the $n$-dimensional torus of diagonal
matrices inside $\GL_n$.

Joyce's $\mc P$ set \cite[Definition 5.3]{JoyceSF} is trivial in this
case because as a quotient stack $\BGL_n= \ast / \GL_n$ where $\ast$
is a point so $\mc P (*, T_n) = \{ T_n\}$. The $\mc Q$ set $\mc Q
(\GL_n, T_n)$, is computed in \cite[Example 5.7]{JoyceSF} to be the
set of all tori
$$T_\phi:= \{ \mr{diag} (z_1,\cdots, z_n): z_i \in \mb G_m, z_i = z_j
\mbox{ if } \phi (i) = \phi(j),\quad \forall i, j\}.$$ where $\phi$
ranges over all surjection maps $\phi: \ul n \to \ul r$. Finally the
$\mc R$ set coincides with $\mc Q$. Joyce's definition then needs
computation of $M_G^X (P, Q, R)$ where $P$, $Q$, and $R$ are selected
respectively from $\mc P$, $\mc Q$ and $\mc R$. In our case this is
$$M_{\GL_n}^* (T_n, Q, R) = \left| \frac{N_{\GL_n} (T_n)}{C_{\GL_n}
  (Q) \cap N_{\GL_n} (T_n)} \right|\inv n^{\GL_n}_{T_n}(R, Q)$$ 
for all choices of $R, Q \in \mc Q(\GL_n, T_n)$ such that $R \subseteq
Q$. Now we unwind the definition of $n^{\GL_n}_{T_n}(R, Q)$.   
\begin{align*}
n(R, Q) &= \sum_{ \stack{B \subseteq \{\hat Q \in \mc Q:\hat Q
    \subseteq Q \}}{Q \in B, \cap_{\hat Q \in B} \hat Q = R}}
(-1)^{|B| -1} 
\end{align*}
We can finally define the virtual projections of $\BGL_n$ as 
$$\Pi^{vi}_k (\BGL_n) = \sum_{R: \dim R = k} \sum_{Q: R \subseteq Q}
M_{\GL_n}^* (T_n, Q, R) [B C_G (Q)].$$ 
We say $Q \in \mc Q (\GL_n, T_n)$ is of `type $\sigma$' if the
corresponding surjection $\phi: \{1, \cdots , n\} \to \{1, \cdots,
r\}$ induces the partition $\sigma\vdash n$. Note that, 
there are 
$\frac{n!}{\sigma_1 ! \cdots \sigma_n ! (1!)^{\sigma_1} \cdots
  (n!)^{\sigma_n}}$ 
of them. Also $C_{\GL_n} Q$ only depends on the type of $Q$ and is
isomorphic to $\prod_{i=1}^n [\GL_i]^{\sigma_i}$. The normalizer of
$T_n$ is $S_n \ltimes T^n$ and therefore  
$$C_{\GL_n} (T_\phi) \cap N_{\GL_n} (T_n) = \prod_{i=1}^n [S_i \ltimes
  T_i]^{\sigma_i}$$ 
and  
$$\left| \frac{N_{\GL_n} (T^n)}{C_{\GL_n} (Q) \cap N_{\GL_n} (T_n)}
\right|\inv = \frac{(1!)^{\sigma_1} \cdots (n!)^{\sigma_n}}{n!}.$$ 
We have
\begin{align*} 
\Pi^{vi}_k (\BGL_n) 
&= \sum_{R: \dim R = k} \sum_{Q: R \subseteq Q} M_{\GL_n}^* (T_n, Q,
R) [B C_G (Q)]\\ 
&= \sum_{Q}  \left| \frac{N_{\GL_n} (T_n)}{C_{\GL_n} (Q) \cap
  N_{\GL_n} (T_n)} \right|\inv \left(\sum_{\stack{R\subseteq Q}{\dim R
    = k}}   n^{\GL_n}_{T_n}(R, Q)\right) [B C_G (Q)] \\ 
&= \sum_{\sigma} ({\# Q \text{ of type }
  \sigma}). \frac{(1!)^{\sigma_1} \cdots (n!)^{\sigma_n}}{n!}
s(|\sigma|, k)  \prod_{i=1}^n [\BGL_i]^{\sigma_i} \\ 
&= \sum_{\sigma} \frac{1}{\sigma_1 ! \cdots \sigma_n ! }  s(|\sigma|,
k)  \prod_{i=1}^n [\BGL_i]^{\sigma_i}. 
\end{align*}
where the third line follow from the lemma below. We conclude that
Joyce's virtual projections of $\BGL_n$ are identical to our
eigenprojections.
\end{proof}

\begin{lem} For a $Q$ of type $\sigma$, we have 
$$\sum_{R: \dim R = k} n(R, Q) = s (|\sigma|, k).$$
\end{lem}
\begin{proof} We let $m= \dim Q = |\sigma|$ in this proof. Obviously
  if $|\sigma| < k$, there is no possible choice of  
$$B \subseteq \{ \hat Q: \hat Q \subseteq Q\} : Q \in B, \bigcap_{\hat
    Q \in B} \hat Q = R$$ 
therefore proving 
$$\sum_{R: \dim R = k} n(R, Q) = s (|\sigma|, k) \quad \text{ if }
\quad |\sigma|< k.$$ 
In the case that $|\sigma |= k$, the only choice of $R$ is $Q$ itself
and the only choice of  
$B$ is the set $B= \{ Q\}$. This proves  
$$\sum_{R: \dim R = k} n(R, Q) = s (|\sigma|, k) \quad \text{ if }
\quad |\sigma|= k.$$ 
All other values of $s(|\sigma|, k)$ are defined recursively by  
$$s(m, k) = s(m-1, k-1) - (m-1) s(m-1, k).$$ 
So it suffices to show that $\sum_{R: \dim R = k} n(R, Q)$ satisfies
the same recursive relation.   

For any choice of $R$,  
$$n(R, Q) = \sum_{ \stack{B \subseteq \{\hat Q \in \mc Q:\hat Q
    \subseteq Q \}}{Q \in B, \cap_{\hat Q \in B} \hat Q = R}}
(-1)^{|B| -1}$$ 
can be computed also form only choosing those $\hat Q$ that are
codimension 1 inside $Q$. This is because for every $\hat Q$ of
codimension $>2$ the number $t$ of intermediate subtori $Q'$  
$$\hat Q \subset Q' \subseteq Q$$
is positive and therefore $B$ containing $\hat Q$ is included in $2^t$
possible choices of $B$ with cancelling size parities.  

Let's write the points of $Q$ as $m$-tuples $(x_1, \ldots, x_m)$ with
$x_i \in \mb G_m$. Let $W$ be the $m-1$ dimensional torus consisting
of points $(x_1, \ldots, x_{m-1})$. Any $R$ with $\dim R= k$ is given
by a set of defining equations  
$$x_{i_1} = \cdots = x_{i_{k_i}}, i=1, 2, \ldots.$$
In defining equation of $R$ with $\dim R = k$, either $x_m$ does not
appear in which case $R|_{W}$ is $k-1$ dimensional. The second case is
if $x_m$ appears in defining equation of $R$, in which case $R|_W$ is
$k$ dimensional and any choice of $B$ consisting of only codimension 1
elements, satisfying $\bigcap_{\hat Q \in B} \hat Q = R$ loses one of
its elements after restriction to $R|_W$. This shows that  
\begin{align*}
 \sum_{ \stack{B \subseteq \{\hat Q \in \mc Q:\hat Q \subseteq Q \}}{Q
     \in B, \dim \bigcap_{\hat Q \in B} \hat Q = k}} (-1)^{|B| -1} 
&= 
 \sum_{ \stack{B \subseteq \{\hat Q \in \mc Q:\hat Q \subseteq W \}}{W
     \in B, \dim \bigcap_{\hat Q \in B} \hat Q = k-1}} (-1)^{|B|
   -1}\\ 
&- (m-1)
 \sum_{ \stack{B \subseteq \{\hat Q \in \mc Q:\hat Q \subseteq W \}}{W
     \in B, \dim \bigcap_{\hat Q \in B} \hat Q = k}} (-1)^{|B| -1} 
\end{align*}
which completes the proof. 
\end{proof}

\bibliographystyle{plain}
\bibliography{../unifiedbib}

\end{document}